\documentclass[12pt]{article}
\usepackage{preprint-layout}

\makeatletter
\newcommand{\proofparagraph}[1]{%
  \ifnum\prevgraf=0\relax
    {\bfseries\boldmath{#1}}\hspace{0.6em}%
  \else
    \par\addvspace{.5\baselineskip}%
    \noindent{\bfseries\boldmath{#1}}\hspace{0.6em}%
  \fi
}
\makeatother

\usepackage{preprint-notation}
\usepackage{hyperref}
\usepackage{bbm}
\usepackage{booktabs}
\usepackage{doi}

\newcommand{\Nt}{N_t}
\newcommand{\Nx}{d}
\newcommand{\uh}{u_{h}}

\newcommand{\barronf}{f}
\newcommand{\hatuh}{\widehat{u}_h}
\newcommand{\steplen}{\rho}

\DeclareMathOperator*{\argmin}{arg\,min}

\title{Solving Inverse Parametrized Problems via Finite Elements and Extreme Learning Networks}

\author{Erik Burman, Mats G. Larson, Karl Larsson, Jonatan Vallin}
\date{\today}

\begin{document}

\maketitle

\begin{abstract}
We develop an interpolation-based modeling framework for para\-meter-dependent partial differential equations arising in control, inverse problems, and uncertainty quantification. The solution is discretized in the physical domain using finite element methods, while the dependence on a finite-dimensional parameter is approximated separately. We establish existence, uniqueness, and regularity of the parametric solution and derive rigorous error estimates that explicitly quantify the interplay between spatial discretization and parameter approximation.
  
In low-dimensional parameter spaces, classical interpolation schemes yield algebraic convergence rates based on Sobolev regularity in the parameter variable. In higher-dimensional parameter spaces, we replace classical interpolation by extreme learning machine (ELM) surrogates and obtain error bounds under explicit approximation and stability assumptions. The proposed framework is applied to inverse problems in quantitative photoacoustic tomography, where we derive potential and parameter reconstruction error estimates and demonstrate substantial computational savings compared to standard approaches, without sacrificing accuracy.
\end{abstract}

\section{Introduction}
\label{sec:intro}

\paragraph{Parameter-Dependent PDE.}
Many problems in scientific computing require repeatedly solving partial differential equations (PDEs) that depend on physical or geometrical parameters. In this paper, we consider a Schr\"odinger-type (or more broadly, elliptic-type) PDE defined over a convex polygonal (or polyhedral) spatial domain \(\Omega_x \subset \mathbb{R}^\Nx\) with \(\Nx \in \{2,3\}\), see Figure~\ref{fig:physical-domain}. Let \(\Omega_t = [0,1]^{\Nt}\) denote the parameter domain, where \(\Nt \ge 1\) may be large, see Figure~\ref{fig:param-domain}. Denote \(x \in \Omega_x\) as the spatial variable and \(t \in \Omega_t\) as the parameter vector. Consider the parameterized problem
\begin{equation}
-\Delta u(x,t) + \mu(x,t)\,u(x,t) = f(x) \quad \text{for } (x,t) \in \Omega_x \times \Omega_t
\label{eq:main}
\end{equation}
subject to Dirichlet boundary conditions on \(\partial\Omega_x\). Here, \(f\) is a prescribed source term and \(\mu(x,t)\) is a spatially and parametrically varying potential, which we express as
\begin{equation}
\mu(x,t) = \mu_0(x) + \sum_{i=1}^{\Nt} t_i\,\mu_i(x)
\label{eq:potential}
\end{equation}
where \(t = [t_i]_{i=1}^{\Nt}\) is the parameter vector and \(\mu_i \ge 0\) are given linearly independent potentials. Problems of this kind arise in contexts such as stationary quantum mechanics, diffusion with spatially varying coefficients, and other elliptic operators that depend linearly or nonlinearly on parameters. In inverse reconstruction problem measured data is used to recover the potential $\mu$. The inverse problem where the coefficient is sought in an infinite dimensional space is known to have poor stability properties and often requires regularization \cite{EHN96,IJ15,EH18} or convexification \cite{KLZ20} on the continuous level. If on the other hand the target quantity has finite dimension it is well known that the inverse problem typically is Lipschitz stable \cite{AV05,Sin07,AHFGS19,RS22}, and this observation motivates both the parametrization \eqref{eq:potential} and the subsequent error analysis. We note that it is enough for the unknown parameter to be close to the finite dimensional space in some suitable topology. The distance enters as a perturbation and limits the accuracy of the reconstruction. As a consequence it should be possible to design accurate reconstruction methods for inverse problems where the target quantity is finite dimensional. The exploitation of Lipschitz stability in the design and analysis of numerical methods is recent, and has in particular been studied in the context of unique continuation problems \cite{Harr21,BO24,BOZ25,BCJZ25, BKO25}. The analysis of more general inverse problem is complicated by the need of data in the form of the complete Dirichlet to Neumann map. For finite dimensional target quantities it is known that finite dimensional data is sufficient for the reconstruction \cite{AS19,AS22,AAS23}.
In the classical Calder\'on-Gelfand problem the finite dimensional potential $\mu$ is reconstructed from data on the boundary, in which case a sufficiently rich, but finite, sample of the Dirichlet to Neumann map is required \cite{AS19}. Even this leads to an additional level of computational complexity since several data points have to be considered, the choice of which is unclear. To avoid difficulties related to this aspect the target application that we will consider herein is the optical reconstruction problem in quantitative photo acoustic tomography (QPAT). This is a well studied imaging problem \cite{AS99,Arr99,BU10,BR11,TC23} which consists of two steps. First acoustic measurements on the boundary of the domain are used to reconstruct an initial pressure distribution $p_0$ by solving the acoustic unique continuation problem. It is known that the initial pressure is related to the absorbed optical energy $H(x)$ by the Gr\"uneisen parameter $\Gamma$ such that $p_0 = \Gamma H$. The energy $H$ finally enters the equation \eqref{eq:main} through the relation $H = \mu u$. The inverse problem that we are interested in is to reconstruct $\mu$ in some finite dimensional space, given $H$, or its projection on some other finite dimensional space. In its simplest form, and under idealized data assumptions, this inverse problem is linear. Indeed, given $H$ we can solve the elliptic problem for $u$ and then set $\mu = H/u$. This approach is however usually not possible due to lack of data or resolution of data. For a finite element analysis of the potential reconstruction problem using finite element methods with distributed data we refer to \cite{JLQZ23}.

Our prime objective in the present paper is to develop further the ideas of \cite{BLLL25} in the context of potential reconstruction in finite dimensional space. In \cite{BLLL25} we considered a boundary reconstruction problem for a non-linear elliptic problem. Given a collective data set of measurements on the boundary we first identified a low dimensional manifold of the target quantity using principal component analysis or auto encoders. A network was then trained to map from the target space to the associated finite element solution of the non-linear problem. The inverse problem was finally solved efficiently using the data to solution map. The cost of the training process makes the approach prohibitively expensive, unless the operator can be used sufficiently many times. Here the aim is to design a method that has much lower training cost, but can still deliver sufficiently accurate results to allow for fast optimization up to a certain accuracy of the minimizer. The idea is to construct a fast method that allows to explore the parameter space in a global optimization step, resulting in a first approximation not far from the global minimizer. If a more accurate solution to the inverse problem is required Newton iterations can be applied, this time solving the finite element problems during iteration \cite{BKO25}. To this end we will exploit that the non-linear reconstruction problem \eqref{eq:main} can be written as a linear problem in high dimensional space by identifying the $t_i$ in \eqref{eq:potential} with dimensions in a parameter space $\Omega_t$. We then use a finite element method to discretize the solution in $\Omega_x$. In the parameter space on the other hand one may use for example tensor product orthogonal polynomials, or as we do below, so called Extreme Learning Machines \cite{HZS04,HZS06}. This is based on collocation over network approximation spaces that use randomized linear basis functions inspired by the activation functions of neural networks. Extreme Learning Machines was first used for the approximation of PDEs in low dimensions in \cite{DS20,CFS21,PK21,DS22}. It was applied to the approximation of high dimensional PDE in \cite{WD23} and more recently some approximation theory was proposed \cite{DMSW25}. A key feature of the present work is that the proposed reduced-order models are supported by a rigorous analysis of parametric regularity and by explicit error estimates that distinguish between low-dimensional and high-dimensional parameter regimes.

\begin{figure}
\centering
\begin{minipage}[b]{0.35\linewidth}
\begin{subfigure}{\linewidth}
\centering
\includegraphics[width=0.6\textwidth]{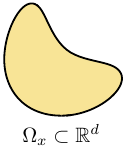}
\caption{Spatial domain}
\label{fig:physical-domain}
\end{subfigure}
\end{minipage}
\qquad
\begin{minipage}[b]{0.35\linewidth}
\begin{subfigure}{\linewidth}
\centering
\includegraphics[width=0.9\textwidth]{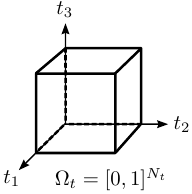}
\caption{Parameter domain}
\label{fig:param-domain}
\end{subfigure}
\end{minipage}
\caption{\textit{Domains.} \textbf{(a)} The spatial domain $\Omega_x$ is a bounded subset in $\IR^\Nx$ for $\Nx\in \{2,3\}$. \textbf{(b)} The parameter domain $\Omega_t$ is the hypercube $[0,1]^{\Nt}$ where $\Nt$ may be large.}
\label{fig:domains}
\end{figure}

\paragraph{Motivation for Interpolation-Based Modeling.}
Directly solving \eqref{eq:main} for many parameter values can be computationally prohibitive, particularly in higher-dimensional parameter spaces or when repeated evaluations are required (e.g., in inverse problems or optimization loops). 

A natural strategy is therefore to construct a surrogate model for the parameter-to-solution map \(t \mapsto u(\cdot,t)\), enabling fast evaluation once an initial offline stage has been completed. This leads to an offline--online decomposition: the surrogate is constructed using a set of precomputed solutions, after which new evaluations can be performed at significantly reduced cost. Importantly, the surrogate is given in an explicit form that is differentiable with respect to the parameters, which enables efficient gradient-based solution of inverse problems.

Such an approach is particularly advantageous when many queries of the solution map are needed or when real-time evaluation is required. In low-dimensional parameter spaces, classical interpolation or reduced-basis methods \cite{HRS16} can be effective. However, in higher dimensions, the curse of dimensionality motivates the use of approximation spaces whose complexity does not grow exponentially with \(\Nt\). 

In this work, we propose an interpolation-based approach that combines classical discretization techniques with a simple neural-network-based surrogate model. In particular, we employ Extreme Learning Machines (ELMs), which provide a computationally efficient and analytically tractable framework for high-dimensional interpolation. While more advanced learning approaches exist, the use of ELMs allows us to develop a rigorous approximation theory while retaining an explicit and differentiable surrogate representation. This approach combines:
\begin{itemize}
\item Finite element methods (FEM) for spatial discretization over \(\Omega_x\),
\item tensor-product or simplicial interpolation for lower-dimensional parameter spaces,
\item ELM-based interpolation for higher-dimensional parameter spaces.
\end{itemize}

\paragraph{Contributions and Significance.}
Our main contributions are:
\begin{itemize}
\item A unified FEM-based strategy for the spatial dimension and an interpolation approach for the parameter dimension, suitable for both low- and high-dimensional parameter spaces.
\item Rigorous regularity results for the parametric PDE and joint error estimates that quantify the interaction between spatial discretization and parameter approximation.
\item A distinction between low-dimensional parameter regimes, where classical interpolation yields algebraic convergence rates, and high-dimensional regimes, where ELM-based surrogates achieve dimension-independent approximation rates under explicit assumptions.
\item An application of the proposed framework to inverse problems in quantitative photoacoustic tomography, including provable potential and parameter reconstruction error bounds.
\end{itemize}
Throughout the paper, we assume basic familiarity with finite element theory and standard elliptic regularity. For brevity, generic constants are denoted by \(C\), with dependencies on domain geometry and problem coefficients omitted, and \(a \le C b\) is written \(a \lesssim b\).

\paragraph{Paper Outline.}
In Section \ref{sec:d-plus-n-formulation}, we present the full \(\Nx+\Nt\)-dimensional PDE formulation, including weak forms and regularity estimates. Section \ref{sec:interpolation-rom} introduces our interpolation-based reduced-order approach, focusing on how we handle high-dimensional parameter spaces via neural networks, and provides rigorous error estimates combining spatial and parametric approximation errors. In Section \ref{sec:qpat}, we show how this framework is applied to inverse problems in quantitative photoacoustic tomography. Section \ref{sec:numerical-experiments} presents numerical results illustrating accuracy and efficiency. Finally, Section \ref{sec:conclusions} summarizes the main contributions and discusses potential extensions.

\section{\boldmath The $\Nx+\Nt$ Dimensional Problem}
\label{sec:d-plus-n-formulation}

We rewrite the PDE on the product domain \(\mathcal{O}=\Omega_x\times\Omega_t\),
\begin{equation}
-\Delta u(x,t) + \mu(x,t)\,u(x,t) = f(x,t) \quad \text{for } (x,t)\in \mathcal{O}
\label{eq:reduced}
\end{equation}
with Dirichlet boundary conditions on \(\partial\Omega_x\), which we for simplicity assume to be homogeneous throughout the theoretical exposition. In contrast to \eqref{eq:main}, we here let \(f\) depend on the parameter \(t\) for later applicability in Section~\ref{sec:qpat}.
The weak formulation reads: find \(u\in W\) such that
\begin{equation}
(\nabla u,\nabla v)_{\mathcal{O}} + (\mu\,u, v)_{\mathcal{O}} = \langle f, v\rangle_{W'\times W} \quad \forall v\in W
\label{eq:weak-form}
\end{equation}
where
\begin{equation}
W = L^2(\Omega_t;H_0^1(\Omega_x)) \quad \text{and}\quad W' \text{ denotes its dual}
\label{eq:Wdef}
\end{equation}
the product-domain inner product is defined by
\begin{equation}
(v,w)_{\mathcal{O}} := \int_{\Omega_t}\!\int_{\Omega_x} v(x,t)\,w(x,t)\,dx\,dt
\label{eq:mcO-ip}
\end{equation}
and the duality pairing is defined by
\begin{equation}
\langle F,v\rangle_{W'\times W}:=\int_{\Omega_t}\langle F(\cdot,t),v(\cdot,t)\rangle_{H^{-1}(\Omega_x)\times H_0^1(\Omega_x)}\,dt
\end{equation}

We equip $W$ with the norm $\|v\|_W:=\|\nabla v\|_{L^2(\mathcal O)}$, which is equivalent to the standard Bochner norm on $L^2(\Omega_t;H_0^1(\Omega_x))$ by Poincar\'e's inequality; its dual $W'$ is canonically identified with $L^2(\Omega_t;H^{-1}(\Omega_x))$.
More explicitly, for a functional $F\in W'$ we define
\begin{equation}
\|F\|_{W'}:=\sup_{0\neq v\in W}\frac{\langle F,v\rangle_{W'\times W}}{\|v\|_{W}}
\end{equation}
Under the identification $W'\cong L^2(\Omega_t;H^{-1}(\Omega_x))$, this norm is equivalently given by
\begin{equation}
\|F\|_{W'}^2=\int_{\Omega_t}\|F(\cdot,t)\|_{H^{-1}(\Omega_x)}^2\,dt
\end{equation}

For functions $u(\cdot, t), v(\cdot , t)$ defined on $\Omega_x$ for fixed $t\in \Omega_t$, we denote by 
\begin{equation}
(u,v)_{\Omega_x}:=\int_{\Omega_x}u(x,t)v(x,t) \,dx
\label{eq:Omega-ip}
\end{equation}
the standard $L^2(\Omega_x)$ inner product, and by \(\|v\|_{\Omega_x}=\|v\|_{L^2(\Omega_x)}\) the standard $L^2(\Omega_x)$ norm.

\begin{ass}[Standing Assumptions] \label{ass:assumptions}
We assume:
\begin{itemize}
\item {\bf Domain and Boundary.}\ \ \(\Omega_x\subset\mathbb{R}^\Nx\) is either convex polygonal/polyhedral or a bounded \(C^{1,1}\) domain, with Dirichlet boundary conditions on \(\partial\Omega_x\).

\item {\bf Coefficients (Boundedness and Coercivity).}\ \ There exist constants \(0<\mu_{\min}\le \mu_{\max}<\infty\) such that
\begin{equation}
\mu_{\min} \le \mu(x,t) \le \mu_{\max} \quad \text{for a.e. } (x,t)\in \mathcal{O}
\label{eq:mu-bounds}
\end{equation}
where
\begin{equation}
\mu(x,t)=\mu_0(x)+\sum_{i=1}^{\Nt} t_i\,\mu_i(x) \quad \text{for } t\in\Omega_t
\label{eq:mu-affine-local}
\end{equation}
with \(\mu_i\in L^\infty(\Omega_x)\) for \(i=0,\dots,\Nt\).

\item {\bf Right-Hand Side.}\ \ $f=f(x,t)$ is affine in $t$, i.e.,
\begin{equation}
 f(x,t)=f_0(x)+\sum_{i=1}^{\Nt} t_i\,f_i(x)
\end{equation}
with $f_i\in H^{-1}(\Omega_x)$ for $i=0,\dots,\Nt$. Since $f$ is affine in $t$, we have $\partial_t^{\alpha}f\equiv 0$ for all $|\alpha|\ge 2$; in particular $f\in H^k(\Omega_t;H^{-1}(\Omega_x))$ for any $k\ge 0$.
\end{itemize}

\end{ass}

We then have the following existence and regularity results.

\begin{thm}[Existence, Uniqueness, Regularity]
\label{thm:regularity-Hm}
Under Assumption~\ref{ass:assumptions}, there exists a unique solution $u\in W$ to \eqref{eq:reduced}. Moreover, for every integer $k\ge 0$ the solution belongs to $H^k(\Omega_t;H_0^1(\Omega_x))$ and satisfies the estimate
\begin{equation}
\sum_{|\alpha|\le k}\|\partial_t^{\alpha}u\|_{L^2(\Omega_t;H_0^1(\Omega_x))}
\lesssim \ \ \ \sum_{\mathclap{|\alpha|\le \min(k,1)}}\|\partial_t^{\alpha}f\|_{L^2(\Omega_t;H^{-1}(\Omega_x))}
\label{eq:regularity-H1}
\end{equation}
with a constant independent of $t$.

If, in addition, $m\ge 2$, $f\in H^k(\Omega_t;H^{m-2}(\Omega_x))$,
$\mu_i\in W^{m-2,\infty}(\Omega_x)$ for $i=0,\dots,\Nt$, and the operator
$A(t):=-\Delta+\mu(\cdot,t)$ satisfies a uniform elliptic shift property of order $m$, i.e., for all $g\in H^{m-2}(\Omega_x)$, the solution $w\in H_0^1(\Omega_x)$ to
\(
A(t)w=g
\)
satisfies
\(
\|w\|_{H^m(\Omega_x)}\lesssim \|g\|_{H^{m-2}(\Omega_x)}
\)
with a constant independent of $t$.
Then for every integer $k\ge 0$ the solution belongs to $H^k(\Omega_t;H^m(\Omega_x))$ and satisfies
\begin{equation}
\sum_{|\alpha|\le k}\|\partial_t^{\alpha}u\|_{L^2(\Omega_t;H^m(\Omega_x))}
\lesssim \ \ \ \sum_{\mathclap{|\alpha|\le \min(k,1)}}\|\partial_t^{\alpha}f\|_{L^2(\Omega_t;H^{m-2}(\Omega_x))}
\label{eq:regularity-Hm-general}
\end{equation}
with a constant independent of $t$.
\end{thm}

\begin{rem}[Elliptic Shift Assumption]
\label{rem:elliptic-shift}
For $m=2$, the elliptic shift property required above holds under the standing assumptions that
$\Omega_x$ is convex polygonal/polyhedral or of class $C^{1,1}$ and $\mu(\cdot,t)\in L^\infty(\Omega_x)$ uniformly in $t$, by standard elliptic regularity.
For $m>2$, the shift property typically requires additional smoothness of the spatial domain and of the coefficient functions $\mu_i$, for instance $\Omega_x$ of class $C^{m-1,1}$ and $\mu_i\in W^{m-2,\infty}(\Omega_x)$.
\end{rem}

\begin{proof}
\proofparagraph{1. Existence and Uniqueness.}
Define the bilinear form $a(\cdot,\cdot)$ on $W\times W$ by
\begin{align}
\label{eq:bilinear}
a(u,v) = (\nabla u,\nabla v)_{\mathcal{O}} + (\mu u, v)_{\mathcal{O}}
\end{align}
By \eqref{eq:mu-bounds}, Cauchy--Schwarz, and Poincar\'e,
\begin{equation}
|a(u,v)|\le \|\nabla u\|_{\mathcal{O}}\,\|\nabla v\|_{\mathcal{O}}+\mu_{\max}\,\|u\|_{\mathcal{O}}\,\|v\|_{\mathcal{O}}\lesssim \|u\|_{W}\,\|v\|_{W}
\end{equation}
so $a$ is bounded on $W$. Moreover, using \eqref{eq:mu-bounds},
\begin{equation}
 a(v,v)=\|\nabla v\|_{\mathcal{O}}^2+(\mu v,v)_{\mathcal{O}}\ge \|\nabla v\|_{\mathcal{O}}^2+\mu_{\min}\,\|v\|_{\mathcal{O}}^2\gtrsim \|v\|_{W}^2
\end{equation}
so $a$ is coercive on $W$. Since the right-hand side functional $\ell(v):=\langle f,v\rangle_{W'\times W}$ is bounded on $W$ with $\|\ell\|_{W'}=\|f\|_{W'}$, Lax--Milgram yields a unique $u\in W$ solving \eqref{eq:weak-form}.

\proofparagraph{2. Differentiated Equations.}
Rather than assuming classical differentiability in $t$, we proceed as follows. We set $w_0(\cdot,t):=u(\cdot,t)$. For any multi-index $\alpha\in\mathbb{N}_0^{\Nt}$ with $|\alpha|\ge 1$, we first define $w_\alpha(\cdot,t)\in H_0^1(\Omega_x)$ as the unique solution of the problem \eqref{eq:multi-derivative-equation-Hm}--\eqref{eq:Falpha-Hm} below. A standard difference-quotient argument in the parameter variable then shows that these functions coincide with the weak parameter derivatives $\partial_t^{\alpha}u$. More precisely, for each coordinate direction $e_i$ and sufficiently small $\tau$, one considers the difference quotient
\begin{equation}
\delta_{\tau,i}u(\cdot,t):=\frac{u(\cdot,t+\tau e_i)-u(\cdot,t)}{\tau}
\end{equation}
on the shifted domain $\Omega_t\cap(\Omega_t-\tau e_i)$, derives from the weak formulation a uniform bound for $\delta_{\tau,i}u$ in $L^2(\Omega_t;H_0^1(\Omega_x))$, and then applies the standard difference-quotient characterization of Sobolev regularity to the Hilbert-valued map $t\mapsto u(\cdot,t)$ to conclude that $\partial_{t_i}u\in L^2(\Omega_t;H_0^1(\Omega_x))$, cf. \cite[\S5.8.2]{MR1625845}. Passing to the limit $\tau\to 0$ in the difference-quotient weak formulation yields the differentiated equation below. Iterating this argument yields the corresponding identities for higher-order derivatives $w_\alpha$.
Using the affine dependence \eqref{eq:mu-affine-local}, this leads for a.e. $t\in\Omega_t$ to
\begin{equation}
-\Delta w_\alpha(\cdot,t)+\mu(\cdot,t)\,w_\alpha(\cdot,t)=F_\alpha(\cdot,t)\quad\text{in }\Omega_x,\qquad w_\alpha(\cdot,t)|_{\partial\Omega_x}=0
\label{eq:multi-derivative-equation-Hm}
\end{equation}
where, for all $|\alpha|\ge 0$,
\begin{equation}
F_\alpha(\cdot,t)=\partial_t^{\alpha}f(\cdot,t)-\sum_{i=1}^{\Nt}\alpha_i\,\mu_i\,w_{\alpha-e_i}(\cdot,t)
\label{eq:Falpha-Hm}
\end{equation}
In the affine case, $\partial_t^{\alpha}f\equiv 0$ for all $|\alpha|\ge 2$.

In the remainder of the proof we therefore identify $w_\alpha$ with $\partial_t^{\alpha}u$.

\proofparagraph{3. Energy Estimate Argument ($m=1$).}
This argument is purely energy-based and does not rely on any elliptic regularity of the spatial operator.
Test \eqref{eq:multi-derivative-equation-Hm} with $w_\alpha(\cdot,t)$ and use \eqref{eq:mu-bounds}, duality, and Poincar\'e's inequality to obtain
\begin{equation}
\|\nabla w_\alpha(\cdot,t)\|_{L^2(\Omega_x)}^2+\mu_{\min}\,\|w_\alpha(\cdot,t)\|_{L^2(\Omega_x)}^2
\le \langle F_\alpha(\cdot,t),w_\alpha(\cdot,t)\rangle_{H^{-1}\times H_0^1}
\end{equation}
Hence
    \begin{equation}
    \|w_\alpha(\cdot,t)\|_{H_0^1(\Omega_x)}\lesssim \|F_\alpha(\cdot,t)\|_{H^{-1}(\Omega_x)}
    \label{eq:H1-energy-est}
    \end{equation}
Here we use that on $H_0^1(\Omega_x)$ the $H^1$- and $H_0^1$-norms are equivalent.
For $\alpha=0$, \eqref{eq:H1-energy-est} yields
\(\|u(\cdot,t)\|_{H_0^1(\Omega_x)}\lesssim \|f(\cdot,t)\|_{H^{-1}(\Omega_x)}.\)
Taking the $L^2(\Omega_t)$-norm gives
\begin{equation}
\|u\|_{L^2(\Omega_t;H_0^1(\Omega_x))}\lesssim \|f\|_{L^2(\Omega_t;H^{-1}(\Omega_x))}
\end{equation}
For $|\alpha|\ge 1$, by \eqref{eq:Falpha-Hm} we have
\begin{equation}
\|F_\alpha(\cdot,t)\|_{H^{-1}(\Omega_x)}
\le \|\partial_t^{\alpha}f(\cdot,t)\|_{H^{-1}(\Omega_x)}
+ \sum_{i=1}^{\Nt}\alpha_i\,\|\mu_i\|_{L^\infty(\Omega_x)}\,\|w_{\alpha-e_i}(\cdot,t)\|_{H_0^1(\Omega_x)}
\label{eq:Falpha-Hm-bound}
\end{equation}
since a multiplication by $\mu_i\in L^\infty(\Omega_x)$ defines a bounded operator on $H^{-1}(\Omega_x)$ (i.e., $\|\mu_i v\|_{H^{-1}}\lesssim \|\mu_i\|_{L^\infty}\,\|v\|_{H^{-1}}$) and we have continuous embeddings $H^1(\Omega_x)\hookrightarrow L^2(\Omega_x)\hookrightarrow H^{-1}(\Omega_x)$. 
Taking $L^2(\Omega_t)$-norms in \eqref{eq:H1-energy-est}--\eqref{eq:Falpha-Hm-bound} and arguing by induction on $|\alpha|$ yields
\begin{equation}
\sum_{|\alpha|\le k}\|\partial_t^{\alpha}u\|_{L^2(\Omega_t;H_0^1(\Omega_x))}
\lesssim \sum_{|\alpha|\le \min(k,1)}\|\partial_t^{\alpha}f\|_{L^2(\Omega_t;H^{-1}(\Omega_x))}
\end{equation}
which is \eqref{eq:regularity-H1}.

\proofparagraph{4. Elliptic Shift Argument ($m\ge 2$).}
Assume that the operator $A(t)=-\Delta+\mu(\cdot,t)$ satisfies a uniform elliptic shift property of order $m\ge 2$.
Applying this property to \eqref{eq:multi-derivative-equation-Hm} yields
\begin{equation}
\|w_\alpha(\cdot,t)\|_{H^m(\Omega_x)}\lesssim \|F_\alpha(\cdot,t)\|_{H^{m-2}(\Omega_x)}
\label{eq:Hm-via-shift}
\end{equation}
with a constant independent of $t$.

For $\alpha=0$ this gives
\begin{equation}
\|u\|_{L^2(\Omega_t;H^m(\Omega_x))}\lesssim \|f\|_{L^2(\Omega_t;H^{m-2}(\Omega_x))}
\end{equation}
For $|\alpha|\ge 1$, using \eqref{eq:Falpha-Hm} and assuming $\mu_i\in W^{m-2,\infty}(\Omega_x)$ yields
\begin{align}
\|F_\alpha(\cdot,t)\|_{H^{m-2}(\Omega_x)}
&\le \|\partial_t^{\alpha}f(\cdot,t)\|_{H^{m-2}(\Omega_x)}
\\&\quad
+ \sum_{i=1}^{\Nt}\alpha_i\,\|\mu_i\|_{W^{m-2,\infty}(\Omega_x)}
\,\|w_{\alpha-e_i}(\cdot,t)\|_{H^m(\Omega_x)}
\label{eq:Falpha-Hm-bound-general}
\end{align}
Taking $L^2(\Omega_t)$-norms in \eqref{eq:Hm-via-shift} and \eqref{eq:Falpha-Hm-bound-general}, and arguing by induction on $|\alpha|$, yields
\begin{equation}
\sum_{|\alpha|\le k}\|\partial_t^{\alpha}u\|_{L^2(\Omega_t;H^m(\Omega_x))}
\lesssim \sum_{|\alpha|\le \min(k,1)}\|\partial_t^{\alpha}f\|_{L^2(\Omega_t;H^{m-2}(\Omega_x))}
\end{equation}
which proves \eqref{eq:regularity-Hm-general}.
\end{proof}

\begin{rem}[From Weak to Classical Parameter Regularity]
\label{rem:weak-to-classical-parameter-regularity}
Theorem~\ref{thm:regularity-Hm} provides weak parameter regularity in the Bochner--Sobolev sense, namely that the solution satisfies $u\in H^k(\Omega_t;H_0^1(\Omega_x))$, and, under the additional assumptions, $u\in H^k(\Omega_t;H^m(\Omega_x))$. In Theorem~\ref{thm:Ck-continuity} below, this weak parameter regularity is upgraded to classical $C^k$-regularity of the solution map $t\mapsto u(\cdot,t)$.
\end{rem}

\begin{rem}[Non-Affine Right-Hand Side]
The affine dependence of $f$ on $t$ is used only to simplify the regularity analysis, as it implies that higher-order parameter derivatives of $f$ vanish. The surrogate construction itself does not rely on this assumption. For non-affine dependence, the same arguments apply under corresponding Sobolev regularity assumptions on $f$, at the expense of more involved expressions in the parameter-derivative estimates.
\end{rem}

\begin{rem}[CutFEM and Smooth Geometry]
\label{rem:cutfem-geometry}
The regularity results above are statements about the \emph{continuous} boundary value problem posed on the spatial domain $\Omega_x$ and therefore do not depend on whether the discretization is mesh-fitted or unfitted. In particular, if $\Omega_x$ is a $C^{1,1}$ domain (and the coefficients satisfy the assumptions of the regularity theorem), the solution enjoys the stated $H^2$-regularity implied by the elliptic shift property for $m=2$, irrespective of whether one uses standard fitted finite elements or CutFEM \cite{MR3416285}. In CutFEM implementations based on an exact (or sufficiently accurate) geometry description and sufficiently accurate quadrature on cut cells \cite{MR3682761,MR3338676}, geometry and quadrature errors can be made higher order, so that they do not mask this continuous-regularity property in the discretization error.
\end{rem}

The previous theorem provides weak parameter regularity of the solution in the Bochner--Sobolev sense. We now show that, under the same assumptions, the solution map is in fact classically differentiable in the parameter variable and satisfies quantitative Lipschitz bounds.

\begin{thm}[Regularity and Continuity in Parameter Variable]
\label{thm:Ck-continuity}
Under Assumption~\ref{ass:assumptions}, the solution mapping
\begin{equation}
t \mapsto u(\cdot, t)
\label{eq:solution-map}
\end{equation}
to \eqref{eq:reduced} is \(C^k\bigl(\Omega_t;H_0^1(\Omega_x)\bigr)\) for any integer \(k\ge 0\). Moreover, for all multi-indices \(\alpha\) with \(|\alpha|\le k\) and all \(t^{(1)},t^{(2)}\in\Omega_t\),
    \begin{equation}
    \bigl\|\partial_t^{\alpha}u(\cdot, t^{(1)})-\partial_t^{\alpha}u(\cdot,t^{(2)})\bigr\|_{H_0^1(\Omega_x)} \le C_{\alpha}\,\|t^{(1)}-t^{(2)}\|_{\mathbb{R}^{\Nt}}
    \label{eq:ck-continuity-estimate}
    \end{equation}
where $C_{\alpha}$ depends only on the bounds in \eqref{eq:mu-bounds}, the coefficient functions $\mu_i$, the domain geometry, and the affine coefficients $f_i$.
\end{thm}

\begin{proof}
For each $t\in\Omega_t$, let $u(\cdot,t)\in H_0^1(\Omega_x)$ denote the unique weak solution to \eqref{eq:reduced}. By Theorem~\ref{thm:regularity-Hm}, the solution map possesses weak parameter regularity in the Bochner--Sobolev sense. We now show that it is in fact classically differentiable in $t$ and satisfies the Lipschitz estimate \eqref{eq:ck-continuity-estimate}.
Let $t^{(1)},t^{(2)}\in\Omega_t$ and set $u^{(i)}:=u(\cdot,t^{(i)})\in H_0^1(\Omega_x)$, $i=1,2$. For $i=1,2$, $u^{(i)}$ satisfies the weak formulation
\begin{equation}
(\nabla u^{(i)},\nabla v)_{\Omega_x} + (\mu(\cdot,t^{(i)})u^{(i)},v)_{\Omega_x}
    = \langle f(\cdot,t^{(i)}),v\rangle_{H^{-1}(\Omega_x)\times H_0^1(\Omega_x)}
\qquad \forall v\in H_0^1(\Omega_x)
\end{equation}
Subtracting the two weak formulations yields
\begin{align}
&(\nabla(u^{(1)}-u^{(2)}),\nabla v)_{\Omega_x} + (\mu(\cdot,t^{(1)})(u^{(1)}-u^{(2)}),v)_{\Omega_x} \nonumber
\\&\qquad\qquad \label{eq:ck-difference-weak}
= -\sum_{j=1}^{\Nt}(t_j^{(1)}-t_j^{(2)})(\mu_j u^{(2)},v)_{\Omega_x}
\\&\qquad\qquad\quad\ \nonumber
+ \langle f(\cdot,t^{(1)})-f(\cdot,t^{(2)}),v\rangle_{H^{-1}(\Omega_x)\times H_0^1(\Omega_x)}
\qquad \forall v\in H_0^1(\Omega_x)
\end{align}
Choosing $v=u^{(1)}-u^{(2)}$ in \eqref{eq:ck-difference-weak} and using coercivity \eqref{eq:mu-bounds} yields
\begin{align}
\|u^{(1)}-u^{(2)}\|_{H_0^1(\Omega_x)}
&\lesssim
\Bigl|\sum_{j=1}^{\Nt}(t_j^{(1)}-t_j^{(2)})(\mu_j u^{(2)},u^{(1)}-u^{(2)})_{\Omega_x}\Bigr|
\\&\quad\nonumber
+\bigl|\langle f(\cdot,t^{(1)})-f(\cdot,t^{(2)}),u^{(1)}-u^{(2)}\rangle_{H^{-1}(\Omega_x)\times H_0^1(\Omega_x)}\bigr|
\end{align}
The first term is bounded by Cauchy--Schwarz and the embedding $H_0^1(\Omega_x)\hookrightarrow L^2(\Omega_x)$ as
\begin{align}
&\Bigl|\sum_{j=1}^{\Nt}(t_j^{(1)}-t_j^{(2)})(\mu_j u^{(2)},u^{(1)}-u^{(2)})_{\Omega_x}\Bigr|
\nonumber
\\&\qquad\qquad\quad
\le \|t^{(1)}-t^{(2)}\|_{\mathbb{R}^{\Nt}}\Bigl\|\Bigl(\sum_{j=1}^{\Nt}\mu_j^2\Bigr)^{1/2}u^{(2)}\Bigr\|_{L^2(\Omega_x)}\,\|u^{(1)}-u^{(2)}\|_{H_0^1(\Omega_x)}
\end{align}
The second term is estimated by duality as
\begin{align}
&\bigl|\langle f(\cdot,t^{(1)})-f(\cdot,t^{(2)}),u^{(1)}-u^{(2)}\rangle_{H^{-1}(\Omega_x)\times H_0^1(\Omega_x)}\bigr|
\nonumber
\\&\qquad\qquad\qquad
\le \|f(\cdot,t^{(1)})-f(\cdot,t^{(2)})\|_{H^{-1}(\Omega_x)}\,\|u^{(1)}-u^{(2)}\|_{H_0^1(\Omega_x)}
\end{align}
Since $f$ is affine in $t$, we have
\begin{equation}
\|f(\cdot,t^{(1)})-f(\cdot,t^{(2)})\|_{H^{-1}(\Omega_x)}
\le \|t^{(1)}-t^{(2)}\|_{\mathbb{R}^{\Nt}}\sum_{j=1}^{\Nt}\|f_j\|_{H^{-1}(\Omega_x)}
\end{equation}
Using the pointwise stability estimate from Lax--Milgram, $\|u(\cdot,t)\|_{H_0^1(\Omega_x)}\lesssim \|f(\cdot,t)\|_{H^{-1}(\Omega_x)}$, together with the boundedness of the affine map $t\mapsto f(\cdot,t)$ on $\Omega_t$, and cancelling the common factor $\|u^{(1)}-u^{(2)}\|_{H_0^1(\Omega_x)}$ yields
\begin{equation}
\|u^{(1)}-u^{(2)}\|_{H_0^1(\Omega_x)}
\lesssim \|t^{(1)}-t^{(2)}\|_{\mathbb{R}^{\Nt}}
\label{eq:lipschitz-final}
\end{equation}
which proves Lipschitz continuity for $|\alpha|=0$.
The same argument applied to the weakly differentiated equation \eqref{eq:multi-derivative-equation-Hm} shows that, for each coordinate direction $e_i$, the weak derivative $\partial_{t_i}u$ admits an $H_0^1(\Omega_x)$-valued representative satisfying the Lipschitz estimate
\begin{equation}
\|\partial_{t_i}u(\cdot,t^{(1)})-\partial_{t_i}u(\cdot,t^{(2)})\|_{H_0^1(\Omega_x)}
\lesssim \|t^{(1)}-t^{(2)}\|_{\mathbb{R}^{\Nt}}
\end{equation}
In particular, this representative is continuous in $t$. Since $\partial_{t_i}u$ is already known to exist in the weak sense by Theorem~\ref{thm:regularity-Hm}, this representative coincides with the classical derivative of the solution map $t\mapsto u(\cdot,t)$. Iterating the same argument for higher-order weak derivatives $\partial_t^\alpha u$ yields \eqref{eq:ck-continuity-estimate} for all $|\alpha|\le k$. Since $f$ is affine in $t$ and therefore $\partial_t^\alpha f\equiv 0$ for all $|\alpha|\ge 2$, the induction closes.
\end{proof}

\section{\boldmath Interpolation-Based Reduced-Order Approximation}
\label{sec:interpolation-rom}

\subsection{Finite Element Method}
\label{subsec:fem}

We discretize the spatial domain \(\Omega_x\) using a finite element mesh \(\mathcal{T}_{h_x}\) with mesh size \(h_x\), see Figure~\ref{fig:domain-mesh}. Let \(V_{h_x} \subset H_0^1(\Omega_x)\) be the associated finite element space. For each parameter value \(t\in\Omega_t\), the finite element approximation \(u_{h_x}(\cdot,t)\in V_{h_x}\) to the solution \(u(\cdot,t)\) is defined by
\begin{equation}
(\nabla u_{h_x}(\cdot,t),\nabla v)_{\Omega_x} + (\mu(\cdot,t)\,u_{h_x}(\cdot,t),v)_{\Omega_x}
= \langle f(\cdot,t),v\rangle_{H^{-1}(\Omega_x)\times H_0^1(\Omega_x)} \quad \forall v\in V_{h_x}
\label{eq:fem}
\end{equation}
where $f=f(x,t)\in H^{-1}(\Omega_x)$ may depend on the parameter $t$.

Let \(\{\varphi_k\}_{k=1}^{K}\) denote the basis of \(V_{h_x}\). The discrete solution can then be written as
\begin{equation}
u_{h_x}(x,t) = \sum_{k=1}^{K} u_k(t)\,\varphi_k(x)
\label{eq:fem-solution-expansion}
\end{equation}
where the coefficient vector \(\mathbf{u}(t)=(u_1(t),\dots,u_{K}(t))^\top\) is determined by solving the linear system of equations associated with \eqref{eq:fem} for each $t\in \Omega_t$. Later on we will need that the finite element coefficients $u_k(t)$ are sufficiently regular (in Sobolev spaces) in the parameter variable $t$. This we prove in the following theorem and subsequent lemma.

\begin{figure}
\centering
\includegraphics[width=0.25\linewidth]{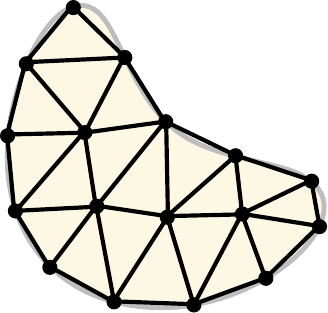}
\caption{\emph{Spatial domain discretization}. The spatial domain $\Omega_x$ is discretized using a finite element mesh $\mathcal{T}_{h_x}$ of mesh size $h_x$.}
\label{fig:domain-mesh}
\end{figure}

\begin{thm}[Parameter Regularity of Finite Element Solution]
\label{thm:Hk-regularity-fem}
For each $t\in\Omega_t$, let $u_{h_x}(\cdot,t)\in V_{h_x}$ denote the unique Galerkin solution of \eqref{eq:fem}.
Under Assumption~\ref{ass:assumptions}, the mapping $t\mapsto u_{h_x}(\cdot,t)$ belongs to
\(H^k\bigl(\Omega_t;V_{h_x}\bigr)\subset H^k\bigl(\Omega_t;H_0^1(\Omega_x)\bigr)\) for every integer $k\ge 0$. More precisely, for every multi-index $\alpha$ with $|\alpha|\le k$,
\(\partial_t^{\alpha}u_{h_x}\in L^2\bigl(\Omega_t;V_{h_x}\bigr)\)
with the stability estimate
\begin{equation}
\sum_{|\alpha|\le k}\|\partial_t^{\alpha}u_{h_x}\|_{L^2(\Omega_t;H_0^1(\Omega_x))}
\le C_k\,\Bigl(\|f\|_{W'}+\sum_{i=1}^{\Nt}\|\partial_{t_i}f\|_{W'}\Bigr)
\label{eq:Hk-regularity-est-fem}
\end{equation}
where $C_k$ depends only on $\Omega_x$, $\mu_{\min}$, $\mu_{\max}$ and $\{\|\mu_i\|_{L^\infty(\Omega_x)}\}_{i=0}^{\Nt}$ (and on $k$), but is independent of $t$ and of the mesh parameter $h_x$.
\end{thm}

\begin{proof}
For $k=0$, the claim follows directly from Lax--Milgram applied to \eqref{eq:fem}, yielding the stability estimate $\|u_{h_x}\|_{L^2(\Omega_t;H_0^1(\Omega_x))}\lesssim \|f\|_{W'}$; hence we assume $k\ge 1$ in the remainder of the proof.
We work with the discrete weak formulation \eqref{eq:fem}. For a.e. $t\in\Omega_t$, define the bilinear form on $V_{h_x}$
\begin{equation}
    a_t(p,q):=(\nabla p,\nabla q)_{\Omega_x}+(\mu(\cdot,t)p,q)_{\Omega_x}
\end{equation}
and the linear functional $\ell_t(q):=\langle f(\cdot,t),q\rangle_{H^{-1}(\Omega_x)\times H_0^1(\Omega_x)}$. By \eqref{eq:mu-bounds}, $a_t$ is bounded and coercive uniformly in $t$ (and hence also when restricted to $V_{h_x}$), with coercivity constant inherited from $H_0^1(\Omega_x)$ and therefore independent of $h_x$.

\proofparagraph{1. Classical Parameter Regularity.}
Let $\mathbf{A}(t)\mathbf{u}(t)=\mathbf{F}(t)$ denote the linear system equivalent to \eqref{eq:fem} in the basis $\{\varphi_j\}_{j=1}^K$.
By the affine dependence of $\mu(\cdot,t)$ and $f(\cdot,t)$, both $\mathbf{A}(t)$ and $\mathbf{F}(t)$ depend affinely on $t$.
By uniform coercivity of $a_t$, $\mathbf{A}(t)$ is invertible for all $t\in\Omega_t$, and the map $t\mapsto \mathbf{A}(t)^{-1}$ is smooth. Hence
\begin{equation}
\mathbf{u}(t)=\mathbf{A}(t)^{-1}\mathbf{F}(t)
\end{equation}
The map $t\mapsto \mathbf{u}(t)$ therefore belongs to $C^\infty(\Omega_t;\mathbb{R}^K)$, and hence $t\mapsto u_{h_x}(\cdot,t)$ is smooth in $t$, in particular $C^k$ as a map into $V_{h_x}$.

\proofparagraph{2. Differentiation in $t$.}
Differentiating \eqref{eq:fem} with respect to $t_i$ yields, for all $v\in V_{h_x}$,
\begin{equation}
a_t(\partial_{t_i}u_{h_x},v)
=-(\mu_i u_{h_x},v)_{\Omega_x}
+\langle \partial_{t_i}f,v\rangle_{H^{-1}(\Omega_x)\times H_0^1(\Omega_x)}
\label{eq:weak-derivative-i-fem-classical}
\end{equation}
Testing with $v=\partial_{t_i}u_{h_x}(\cdot,t)$ and using coercivity gives
\begin{equation}
\|\partial_{t_i}u_{h_x}(\cdot,t)\|_{H_0^1(\Omega_x)}
\lesssim \|\mu_i\|_{L^\infty(\Omega_x)}\,\|u_{h_x}(\cdot,t)\|_{H_0^1(\Omega_x)}
+ \|\partial_{t_i}f(\cdot,t)\|_{H^{-1}(\Omega_x)}
\end{equation}
Integrating in $t$ and using the $k=0$ bound yields
\begin{equation}
\|\partial_{t_i}u_{h_x}\|_{L^2(\Omega_t;H_0^1(\Omega_x))}
\lesssim \|f\|_{W'}+\|\partial_{t_i}f\|_{W'}
\end{equation}

For higher-order derivatives, repeated differentiation, using that $\mu$ is affine in $t$, gives, for all $v\in V_{h_x}$,
\begin{equation}
a_t(\partial_t^\alpha u_{h_x},v)
=\langle F_{\alpha,h_x},v\rangle_{H^{-1}(\Omega_x)\times H_0^1(\Omega_x)}
\end{equation}
with
\begin{equation}
F_{\alpha,h_x}
=\partial_t^\alpha f
-\sum_{i=1}^{\Nt}\alpha_i\,\mu_i\,\partial_t^{\alpha-e_i}u_{h_x}
\end{equation}
Testing with $v=\partial_t^\alpha u_{h_x}(\cdot,t)$ and using coercivity gives
\begin{equation}
\|\partial_t^\alpha u_{h_x}(\cdot,t)\|_{H_0^1(\Omega_x)}
\lesssim
\|\partial_t^\alpha f(\cdot,t)\|_{H^{-1}(\Omega_x)}
+\sum_{i=1}^{\Nt}\alpha_i\,\|\mu_i\|_{L^\infty(\Omega_x)}\,
\|\partial_t^{\alpha-e_i}u_{h_x}(\cdot,t)\|_{H_0^1(\Omega_x)}
\end{equation}
Taking $L^2(\Omega_t)$-norms and arguing inductively on $|\alpha|$ yields
\begin{equation}
\|\partial_t^\alpha u_{h_x}\|_{L^2(\Omega_t;H_0^1(\Omega_x))}
\le C_{|\alpha|}\sum_{|\beta|\le|\alpha|}\|\partial_t^\beta f\|_{W'}
\end{equation}
Because $f$ is affine, the terms with $|\beta|\ge 2$ vanish, and summing over $|\alpha|\le k$ gives \eqref{eq:Hk-regularity-est-fem}.

Thus $\partial_t^\alpha u_{h_x}\in L^2(\Omega_t;V_{h_x})$ for all $|\alpha|\le k$, and hence
\begin{equation}
u_{h_x}\in H^k(\Omega_t;V_{h_x})
\end{equation}
\end{proof}

\begin{lem}[Parameter Regularity of Finite Element Coefficients]
\label{lem:coef-Hk}
Let \(u_{h_x}(x,t)=\sum_{k=1}^{K}u_k(t)\,\varphi_k(x)\) be the finite element solution to \eqref{eq:fem} on the mesh $\mathcal{T}_{h_x}$.
Assume $\mathcal{T}_{h_x}$ is shape-regular and quasi-uniform, so that the coefficient extraction map
\(E:V_{h_x}\to\IR^K\), $E\bigl(\sum_{k=1}^{K}a_k\varphi_k\bigr)=(a_1,\dots,a_K)^\top$, satisfies the stability estimate (mass-matrix stability on quasi-uniform meshes)
\begin{equation}
\|E(v_{h_x})\|_{\ell^2}\ \lesssim\ h_x^{-d/2}\,\|v_{h_x}\|_{L^2(\Omega_x)}\ \lesssim\ h_x^{-d/2}\,\|v_{h_x}\|_{H_0^1(\Omega_x)}\qquad \forall v_{h_x}\in V_{h_x}
\label{eq:coef-extraction-stability}
\end{equation}
where the second inequality follows from Poincar\'e's inequality.
Then, for every integer $k\ge 1$, the coefficient vector $\mathbf{u}(t):=(u_1(t),\dots,u_K(t))^\top$ belongs to $H^k(\Omega_t;\ell^2)$, and
    \begin{equation}
\sum_{|\alpha|\le k}\|\partial_t^{\alpha}\mathbf{u}\|_{L^2(\Omega_t;\ell^2)}
\le C_k\,h_x^{-d/2}\,\Bigl(\|f\|_{W'}+\sum_{i=1}^{\Nt}\|\partial_{t_i}f\|_{W'}\Bigr)
\label{eq:coef-Hk-est}
\end{equation}
where $C_k$ is independent of $t$ and $h_x$ (up to the explicit factor $h_x^{-d/2}$).
In particular, the scalar coefficients $u_1(t), \hdots, u_K(t)\in H^k(\Omega_t)$.
\end{lem}

\begin{proof}
Since $V_{h_x}$ is finite-dimensional and $E$ is linear, differentiation in $t$ commutes with $E$ in the weak sense:
\(\partial_t^{\alpha}\mathbf{u}(t)=E\bigl(\partial_t^{\alpha}u_{h_x}(\cdot,t)\bigr)\) for all multi-indices $\alpha$.
Therefore, by \eqref{eq:coef-extraction-stability},
\begin{align}
\|\partial_t^{\alpha}\mathbf{u}(t)\|_{\ell^2}\le C\,h_x^{-d/2}\,\|\partial_t^{\alpha}u_{h_x}(\cdot,t)\|_{H_0^1(\Omega_x)}\qquad\text{for a.e. }t\in\Omega_t
\end{align}
Taking $L^2(\Omega_t)$-norms and summing over $|\alpha|\le k$, the estimate \eqref{eq:coef-Hk-est} follows from \eqref{eq:Hk-regularity-est-fem}.
Finally, membership $\mathbf{u}\in H^k(\Omega_t;\ell^2)$ implies that the scalar coefficients $u_1(t), u_2(t), \hdots,\allowbreak u_K(t)\in H^k(\Omega_t)$.
\end{proof}

\subsection{Interpolation in the Parameter Domain}
\label{subsec:param-interp}

We propose an interpolation-based approach to efficiently approximate \(u(x,t)\) for any \(t\in\Omega_t\). Choose interpolation points \(\{t_j^*\}_{j=1}^{J}\subset\Omega_t\) and solve the PDE at each point, producing finite element solutions \(\{u_{h_x,j}^*(x)\}_{j=1}^{J}\subset V_{h_x}\). With the basis \(\{\varphi_k\}_{k=1}^{K}\), each solution takes the form
\begin{equation}
u_{h_x,j}^*(x) = \sum_{k=1}^{K} u_{j,k}\,\varphi_k(x)
\label{eq:interp-solutions}
\end{equation}
where \(u_{j,k}\) are the scalar coefficients (and \(\mathbf{u}_j=(u_{j,1},\dots,u_{j,K})^{\top}\) denotes the corresponding coefficient vector). To approximate the solution at an arbitrary \(t\in\Omega_t\), we interpolate these precomputed solutions in parameter space.

\paragraph{Low-Dimensional Parameter Spaces.}
For \(\Nt\le 4\), classical interpolation is effective, and two standard options are:
\begin{itemize}
\item \emph{Tensor-product interpolation}: When the interpolation points \(\{t_j^*\}\) form a structured tensor grid, we can use tensor-product basis functions for the interpolation. This option readily facilitates higher regularity approximation spaces for the interpolation, but is best suited for box-shaped parameter domains.

\item \emph{Simplicial interpolation}: When \(\Omega_t\) is partitioned into simplices with vertices at the interpolation points \(\{t_j^*\}\), it is natural to use piecewise-linear interpolation. This option also gives flexibility in the geometry of the parameter domain, for example allowing non-rectangular domains, constraints on admissible parameter values, or exclusion of regions of no interest (e.g., ``holes'' in \(\Omega_t\)). It is also advantageous when refinement is needed only in localized regions of the parameter space.
\end{itemize}

A simplicial discretization of $\Omega_t$ is illustrated in Figure~\ref{fig:param-mesh}. Let \(V_{h_t}\) denote the corresponding approximation space, with discretization parameter $h_t$, over the parameter domain with basis \(\{\phi_j(t)\}_{j=1}^{J}\). We approximate each coefficient function \(u_k(t)\) using interpolation:
\begin{equation}
\widehat{u}_k(t) = \sum_{j=1}^{J} c_{j,k}\phi_j(t)
\label{eq:interp-lowdim-condition}
\end{equation}
with scalar coefficients \(c_{j,k}\) determined by the interpolation conditions
\begin{equation}
\widehat{u}_k(t_j^*) = u_{j,k}, \quad j=1,\dots,J
\end{equation}

\begin{figure}
\centering
\includegraphics[width=0.35\linewidth]{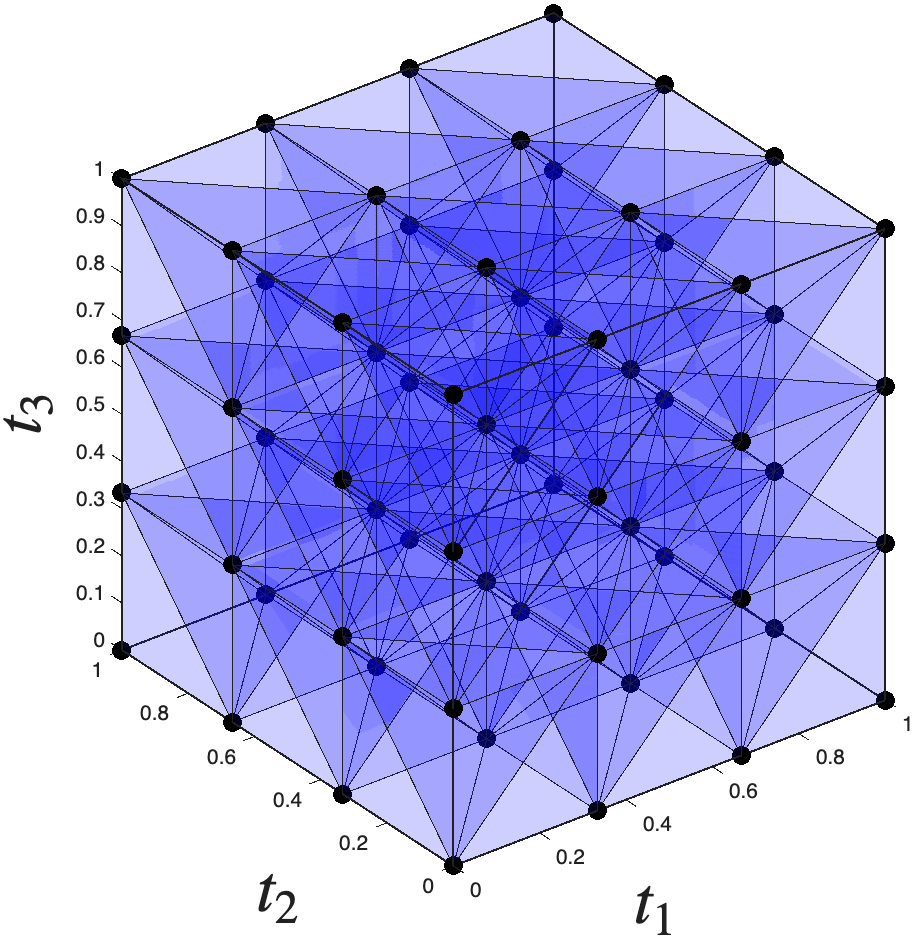}
\caption{\textit{Low-dimensional interpolation in $\Omega_t$.} For $\Nt\leq 4$, the parameter domain can be partitioned into simplices with interpolation points at the vertices. This partition constitutes a mesh $\mathcal{T}_{h_t}$ of the parameter domain with mesh size $h_t$ on which we may employ standard piecewise-linear interpolation. }
\label{fig:param-mesh}
\end{figure}

\paragraph{High-Dimensional Parameter Spaces with Extreme Learning Machines.}
For \(\Nt>4\), classical interpolation becomes infeasible. We approximate each coefficient function \(u_k(t)\) by an Extreme Learning Machine (ELM) ansatz:
\begin{equation}
\widehat{u}_k(t) = \sum_{m=1}^{M} w_{k,m}\,\psi_m(t)
\label{eq:elm-highdim}
\end{equation}
with feature functions
\begin{equation}
\psi_m(t) = \sigma(b_m\cdot t + c_m)
\label{eq:elm-basis}
\end{equation}
where \(\sigma\) is a fixed activation (e.g., sigmoid or ReLU) and the parameters \((b_m,c_m)\in S^{\Nt-1}\times \IR\) are drawn once at random and then held fixed. It's also common to add an outer bias term in the ansatz \eqref{eq:elm-highdim} so that $\widehat{u}_k(t) =v_k + \sum_{m=1}^{M} w_{k,m}\,\psi_m(t)$, $v_k\in \IR$. Denoting the vector of coefficient functions $\widehat{\mathbf{u}}(t)=(\widehat{u}_1(t),\hdots,\widehat{u}_{K}(t))^\top$, we can express
\begin{align}
\widehat{\mathbf{u}}(t)=W\sigma(Bt+c)
\label{eq:elm-matix-form}
\end{align}
where $W\in \IR^{K\times M}$ has elements $[W]_{k,m}=w_{k,m}$, $B\in \IR^{M\times \Nt}$ has the unit vectors $b_m$ as rows and $c\in \IR^M$ is the vector with $c_m$ as elements. In \eqref{eq:elm-matix-form} it is understood that $\sigma$ is applied elementwise. Enforcing interpolation at \(\{t_j^*\}_{j=1}^{J}\) leads to the linear systems
\begin{equation}
\Psi\,w_k = u_{\cdot,k}, \quad k=1,\dots,K
\label{eq:elm-systems}
\end{equation}
where \(\Psi\in\mathbb{R}^{J\times M}\) has entries
\begin{equation}
\Psi_{j,m} = \psi_m(t_j^*) \quad \text{for } j=1,\dots,J,\ m=1,\dots,M
\label{eq:elm-matrix}
\end{equation}
and \(u_{\cdot,k}=(u_{1,k},\dots,u_{J,k})^{\top}\). To get smooth solutions $\widehat{u}_k(t)$, we generally want $M>J$ (more feature functions than interpolation points). In this overparametrized regime the system \eqref{eq:elm-systems} is generally underdetermined (assuming appropriate sampling of $(b_m,c_m)$). We choose the minimal norm solution by solving, for each $k\in \{1,\hdots,K\}$, the constrained problem
\begin{equation}
\label{eq:constrained-problem}
w_k = \argmin_{w \in \IR^M}
\left\{ \frac{1}{2}\,\lVert w \rVert^2 \;\middle|\; \Psi w = u_{\cdot,k} \right\}
\end{equation}

This hybrid FEM–ELM strategy is computationally efficient and scalable for high-dimensional \(\Omega_t\), i.e., for large $\Nt$. In the numerical experiments in Section~\ref{sec:numerical-experiments} we will use the ReLU activation. Then, each feature function in \eqref{eq:elm-basis} can be associated to a hyperplane intersecting $\Omega_t$, as illustrated in Figure~\ref{fig:param-elm}. These hyperplanes define the boundaries separating the regions where the ReLU units are active (linear) and inactive (zero).

\begin{rem}[Regularized Regression vs.\ Interpolation]
\label{rem:elm-regularization}
Instead of enforcing exact interpolation in \eqref{eq:elm-systems}, one may determine the coefficients $w_k$ by solving a regularized least-squares problem, for example of ridge (Tikhonov) type,
\begin{equation}
w_k^\lambda = \argmin_{w\in\mathbb{R}^M}
\left\{
\frac{1}{2}\|\Psi w - u_{\cdot,k}\|_2^2
+ \frac{\lambda}{2}\|w\|_2^2
\right\}, \qquad \lambda>0.
\end{equation}
Such regularized regression can improve numerical stability and robustness, particularly when the feature matrix $\Psi$ is ill-conditioned.

In this work, we focus on interpolatory surrogates and minimum-norm solutions, which fit naturally with the interpolation-based construction. A regularized formulation is also possible, but would introduce an additional bias term in the approximation error depending on $\lambda$.
\end{rem}

\begin{rem}[Choice of Interpolation Points in Higher Dimensions]
\label{rem:points-high-d}
In high-dimensional settings, structured grids are impractical. Quasi-random or randomized designs, such as Sobol' sequences or Latin hypercube sampling, yield well-distributed interpolation points, improving accuracy and numerical stability for moderate \(J\).
\end{rem}

\begin{rem}[Suitability of ELM Interpolation]
\label{rem:elm-smoothness}
The randomized ELM features in \eqref{eq:elm-basis} are global in \(t\), so the target map \(t\mapsto u(\cdot,t)\) must be globally smooth on \(\Omega_t\) for interpolation to be accurate. Under Assumption~\ref{ass:assumptions}, Theorem~\ref{thm:Ck-continuity} guarantees the needed \(t\)-smoothness. However, if those assumptions are violated, for instance, allowing \(\mu_{\min}<0\) so that the operator becomes indefinite on a subset of \(\Omega_t\), the solution can lose global regularity in \(t\). In that case, even a small irregular region contaminates a global ELM interpolant through its global feature functions, degrading accuracy elsewhere in the parameter domain.
\end{rem}

\begin{rem}[Computational Complexity and Practical Recommendations]
\label{rem:complexity}
The computational cost naturally separates into an offline and an online stage.

In the offline stage, the dominant cost lies in solving the PDE at the interpolation points $\{t_j^*\}_{j=1}^J$. This requires $J$ independent PDE solves, and therefore scales linearly with $J$. The subsequent ELM fitting step reduces to a linear least-squares problem with feature matrix $\Psi$, whose cost is typically negligible compared to the PDE solves.

In the online stage, the evaluation of the surrogate is inexpensive: it amounts to evaluating $M$ nonlinear features and forming a linear combination, and therefore scales linearly with $M$. In particular, once the surrogate has been constructed, repeated evaluations of the parameter-to-solution map can be performed at significantly reduced cost.

As a heuristic guideline (based on numerical experiments and common practice in
high-dimensional surrogate modeling \cite{santner2019design}), one may choose $J$ proportional to $N_t$
(e.g., $J \approx 10N_t$--$100N_t$). The optimal choice depends on the
regularity and complexity of the parameter-to-solution map, and may require
problem-specific tuning.
\end{rem}

\begin{rem}[Alternatives to ELM Interpolation]
\label{rem:elm-alternatives}
The interpolating ELM approach adopted here should be viewed as a simple and analytically tractable surrogate modeling strategy rather than a state-of-the-art machine learning method. Its main advantages are its conceptual simplicity, the explicit linear structure in the output weights, and the fact that it enables a rigorous approximation theory with dimension-independent rates under Barron-type assumptions. In addition, the resulting surrogate is explicitly differentiable with respect to the parameters, which is advantageous in inverse problems and optimization settings.

At the same time, the approach has several limitations. The use of random, non-adaptive features may lead to suboptimal performance compared to modern machine learning methods, and the interpolation points $\{t_j^*\}$ are typically chosen a priori rather than in a data-driven or adaptive manner. Consequently, the number of required samples may be larger than in methods that exploit structure in the parameter-to-solution map more efficiently.

More advanced approaches include projection-based reduced-order models 
\cite{HRS16}, operator-learning methods such as neural operators 
\cite{li2020fno,lu2021deeponet}, and adaptive or active learning strategies 
(including Bayesian approaches) that iteratively select informative parameter samples \cite{santner2019design}. These methods can offer improved empirical performance and reduced sample complexity, but typically come at the cost of increased algorithmic complexity and a more limited theoretical understanding, particularly in high-dimensional parameter spaces.

The present work focuses on the ELM-based construction as a baseline method that allows for a clear mathematical analysis and provides insight into the approximation properties of interpolation-based surrogates in high-dimensional parameter domains.
\end{rem}

\begin{figure}
\centering
\begin{subfigure}[ht]{0.45\linewidth}
\centering
\includegraphics[width=0.8\textwidth]{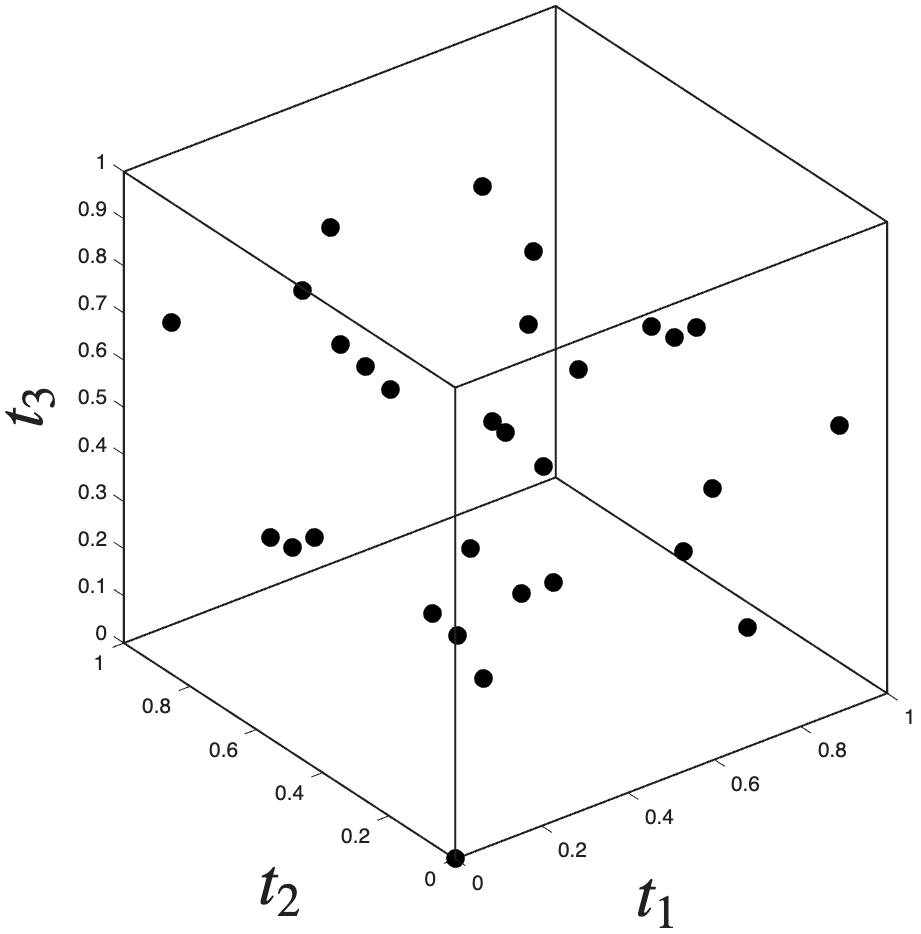}
\caption{}
\end{subfigure}
\quad
\begin{subfigure}[ht]{0.45\linewidth}
\centering
\includegraphics[width=0.8\textwidth]{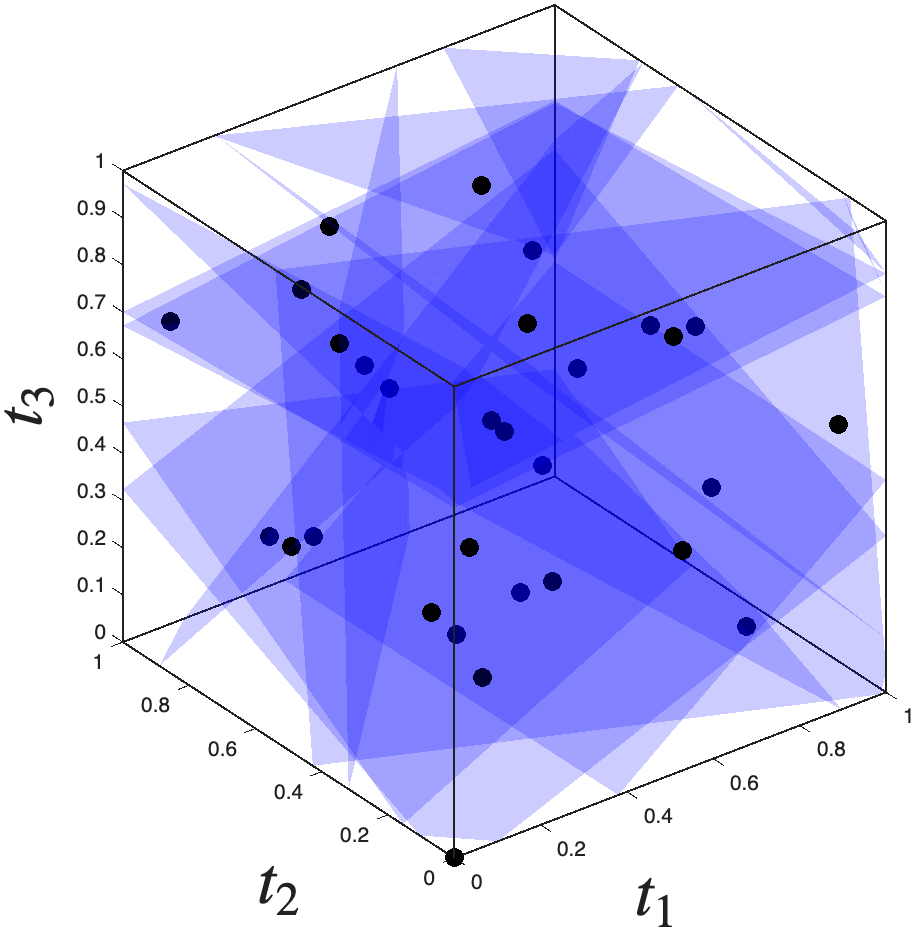}
\caption{}
\end{subfigure}
\caption{\textit{High-dimensional interpolation in \(\Omega_t\).} 
For \(\Nt>4\), an ELM surrogate is used for interpolation.
\textbf{(a)} $J$ quasi-random (Sobol') points in \(\Omega_t\) are used as interpolation points. \textbf{(b)} Random ReLU feature hyperplanes for the ELM surrogate in \(\Omega_t\); each hyperplane marks where one unit switches between active and inactive regions.}
\label{fig:param-elm}
\end{figure}

\subsection{Error Estimates for the Interpolation-Based Approximation}
\label{subsec:interp-errors}

\paragraph{Low-Dimensional Case.}
Let \(\mathcal{T}_{h_t}\) be a partition of \(\Omega_t\) with mesh size \(h_t\), 
and let \(V_{h_t}\) denote the Lagrange finite element space of degree \(q\) on \(\mathcal{T}_{h_t}\) with 
nodal points \(\{t_j^*\}_{j=1}^J\). We take \(\pi_{h_t}: C^0(\Omega_t;H_0^1(\Omega_x))\to V_{h_t}\otimes H_0^1(\Omega_x)\) 
to be the nodal Lagrange interpolation operator in the parameter variable, defined by
\begin{equation}
(\pi_{h_t} v)(t_j^*) = v(t_j^*) \quad \text{for } j=1,\dots,J
\label{eq:piN-definition}
\end{equation}
By Theorem \ref{thm:Ck-continuity}, the map \(t\mapsto u(\cdot,t)\) is at least continuous, so point values in \(t\) are well defined and \eqref{eq:piN-definition} is meaningful.
We can then express the discrete parameter dependent solution as
\begin{equation}
\uh = \pi_{h_t} u_{h_x}
\label{eq:interpolated-solution}
\end{equation}
\begin{thm}[Error Estimates: Low Dimensional Case]
    \label{thm:error-estimates-low-dim}
We have the error estimates
\begin{align}
    \label{eq:error-estimate L2}
    \|u - \uh\|_{\mathcal{O}} \lesssim h_x^{s} \|u\|_{L^2(\Omega_t;H^{s}(\Omega_x))} + h_t^{q+1}\Bigl(\|f\|_{W'}+\sum_{i=1}^{\Nt}\|\partial_{t_i}f\|_{W'}\Bigr)
\end{align}
and
\begin{align}
    \label{eq:error-estimate Linfty}
    \|u - \uh\|_{L^{\infty}(\Omega_t;L^2(\Omega_x))} \lesssim h_x^{s} \|u\|_{L^{\infty}(\Omega_t;H^{s}(\Omega_x))} + h_t^{q+1} |u|_{W^{q+1,\infty}(\Omega_t;L^2(\Omega_x))}
\end{align}
where $q$ is the degree of the parameter FEM space $V_{h_t}$, $s=\min(p+1,m)$, $p$ is the degree of the spatial FEM space $V_{h_x}$ and $m$ is the regularity of the weak solution $u(\cdot,t)\in H^m(\Omega_x)$, i.e., $m=1$ for the assumption $f\in W'$ and $m=2$ if we have the stronger assumption $f\in L^2(\Omega_x)$.
\end{thm}

\begin{proof}
Recall that $\uh=\pi_{h_t}u_{h_x}$. Note that for a.e.\ $t\in\Omega_t$, standard finite element error estimates yield
\begin{align}
    \label{eq:FEM-L2-bound}
\|u(\cdot,t)-u_{h_x}(\cdot,t)\|_{L^2(\Omega_x)}
\lesssim h_x^{s}\,\|u(\cdot,t)\|_{H^{s}(\Omega_x)},
\qquad s=\min(p+1,m)        
\end{align}
We next proceed slightly differently for the $L^2$- and $L^\infty$-estimates, where the $L^2$-estimate avoids the use of $L^2(\Omega_t)$-stability of the nodal Lagrange  interpolation operator and relies only on standard Sobolev approximation estimates in the parameter variable.

\proofparagraph{$L^2$-Estimate (\ref{eq:error-estimate L2}).}
We decompose the error as
\begin{equation}
u-\uh = (u-u_{h_x}) + (u_{h_x}-\pi_{h_t}u_{h_x})
\label{eq:lowdim-error-decomp}
\end{equation}
Squaring \eqref{eq:FEM-L2-bound}, integrating in $t$, and taking square roots gives
\begin{equation}
\|u-u_{h_x}\|_{\mathcal O}
\lesssim h_x^{s}\,\|u\|_{L^2(\Omega_t;H^{s}(\Omega_x))}
\label{eq:lowdim-spatial-L2}
\end{equation}
For the parameter interpolation term, we utilize the standard interpolation estimate for nodal Lagrange interpolation of degree $q$, which in this context reads
\begin{align}
\|v-\pi_{h_t}v\|_{\mcO} &\lesssim h_t^{r}\,|v|_{H^{r}(\Omega_t;L^2(\Omega_x))},
\qquad 1\le r\le q+1
\label{eq:lowdim-int-L2}
\end{align}
with a constant depending only on $q$ and the shape-regularity of $\mathcal T_{h_t}$.
Consequently,
\begin{align}
\|u_{h_x}-\pi_{h_t}u_{h_x}\|_{\mathcal O}
&\lesssim h_t^{r}\,|u_{h_x}|_{H^{r}(\Omega_t;L^2(\Omega_x))}
\\
&\lesssim h_t^{r}\,|u_{h_x}|_{H^{r}(\Omega_t;H_0^1(\Omega_x))}
\\
&\lesssim h_t^{q+1}\Bigl(\|f\|_{W'}+\sum_{i=1}^{\Nt}\|\partial_{t_i}f\|_{W'}\Bigr)
\end{align}
where we in the last inequality chose $r=q+1$ and apply Theorem~\ref{thm:Hk-regularity-fem}.
This concludes the proof of \eqref{eq:error-estimate L2}.

\proofparagraph{Max Norm Estimate (\ref{eq:error-estimate Linfty}).}
We use the decomposition
\begin{equation}
u-\uh = (u-\pi_{h_t}u) + \pi_{h_t}(u-u_{h_x})
\label{eq:lowdim-error-decomp-B}
\end{equation} 
Standard interpolation estimates for nodal Lagrange
interpolation of degree $q$ yield
\begin{align}
\|u-\pi_{h_t}u\|_{L^\infty(\Omega_t;L^2(\Omega_x))}
&\lesssim h_t^{r}\,|u|_{W^{r,\infty}(\Omega_t;L^2(\Omega_x))} , \qquad 1\le r\le q+1
\label{eq:lowdim-B-int-Linfty}
\end{align}
and choosing $r=q+1$ gives the desired $h_t^{q+1}$ rate.
Utilizing the $L^\infty$ stability of $\pi_{h_t}$ and the finite element error estimate \eqref{eq:FEM-L2-bound} we have
\begin{align}
    \|\pi_{h_t}(u-u_{h_x})\|_{L^\infty(\Omega_t;L^2(\Omega_x))} \lesssim \|u-u_{h_x}\|_{L^\infty(\Omega_t;L^2(\Omega_x))}
    \lesssim h_x^{s}\,\|u\|_{L^\infty(\Omega_t;H^{s}(\Omega_x))}
\end{align}
This concludes the proof of \eqref{eq:error-estimate Linfty}.
\end{proof}

 \paragraph{High-Dimensional Case (ELM Approximation Property).}

For \(\Nt>4\), we approximate each coefficient function \(u_k(t)\) by an ELM as in \eqref{eq:elm-highdim}–\eqref{eq:elm-matrix}. 
As in the low-dimensional case, we call the full resulting approximation $\uh$. 
To connect the ELM approximation ansatz with quantitative, dimension-independent error rates, we briefly recall the Barron space framework for two-layer ReLU networks.

\begin{rem}[Barron Spaces and Dimension-Independent Rates]
A foundational result in neural network theory is the universal approximation theorem of Cybenko \cite{cybenko1989approximation}, which states that a two-layer (single hidden layer) neural network with a suitable nonlinear activation 
is capable of approximating any continuous function on a compact domain to arbitrary accuracy, provided a sufficiently large number of hidden units. However, this theorem does not provide any relation between the 
approximation error and the number of units \(M\).

The first significant advance in this direction was made by Barron \cite{barron2002universal}, who introduced a spectral condition on the target function which allows one to quantify not only 
whether a function can be approximated by a shallow sigmoidal neural network, but also how rapidly the error decays as \(M\) increases. Specifically, if the Fourier transform of a function satisfies a weighted integrability 
property, then the \(L^2\)-approximation error decreases at a dimension-independent rate of \(O(M^{-1/2})\). 

Inspired by Barron's original work, E and collaborators formalized the Barron space \(\mathcal{B}\) 
consisting of functions that can be approximated efficiently by two-layer neural networks~\cite{ma2022barron}.
\end{rem}
\begin{definition}[Barron Space]
    \label{def:Barron-space}
The Barron space \(\mathcal{B}(\Omega_t)\) consists of all functions \(\barronf:\Omega_t \to \mathbb{R}\) admitting a representation of the form
\begin{equation}
\barronf(t) = \int_{\mathbb{R}^{\Nt} \times \mathbb{R} \times \mathbb{R}} w\,\sigma(b \cdot t + c)\,\rho(db,\,dc,\,dw), \quad t\in\Omega_t
\end{equation}
where \(\sigma\) is the ReLU activation, \(\rho\) is a probability distribution on \(\mathbb{R}^{\Nt} \times \mathbb{R} \times \mathbb{R}\), and the Barron norm
\begin{equation}
\|\barronf\|_{\mathcal{B}(\Omega_t)} := \inf_{\rho} \int_{\mathbb{R}^{\Nt} \times \mathbb{R} \times \mathbb{R}} |w|\,(\|b\| + |c|)\,\rho(db,\,dc,\,dw)
\end{equation}
is finite. The infimum is taken over all such representations of \(\barronf\).
\end{definition}
For this function space, both direct and inverse approximation theorems hold \cite{ma2020towards}. In that sense, the Barron space is a suitable function class 
for studying approximation by two-layer neural networks. We restate Theorem 1 in \cite{ma2022barron} and Theorem 12 in \cite{ma2020towards} below.
\begin{thm}[Approximation in $L^2$-norm]
    \label{thm:Approximation in L^2-norm}
Let \(\barronf \in \mathcal{B}(\Omega_t)\), then there exists a two-layer neural network $\hat{\barronf}_M$ with $M$ hidden ReLU units such that
\begin{align}
    \|\barronf - \hat{\barronf}_M\|_{L^2(\Omega_t)} \lesssim \frac{\|\barronf\|_{\mathcal{B}(\Omega_t)}}{\sqrt{M}}
\end{align}
\end{thm}
\begin{thm}[Approximation in $L^\infty$-norm]
    \label{thm:Approximation inmax-norm}
    Let \(\barronf \in \mathcal{B}(\Omega_t)\), then there exists a two-layer neural network $\hat{\barronf}_M$ with $M$ hidden ReLU units such that
    \begin{align}
        \|\barronf - \hat{\barronf}_M\|_{L^\infty(\Omega_t)} \lesssim \frac{\sqrt{\Nt+1}}{\sqrt{M}}\,\|\barronf\|_{\mathcal{B}(\Omega_t)}
    \end{align}
\end{thm}
To determine whether a given function belongs to $\mathcal{B}(\Omega_t)$ directly from Definition~\ref{def:Barron-space} can be a difficult task. 
However, Proposition 2 in \cite{ma2022barron} (which is a restated version of Theorem 6 in \cite{klusowski2016risk}) provides a sufficient spectral condition for membership in $\mathcal{B}(\Omega_t)$.
We present it below.
\begin{prop}[Spectral Criterion]
    \label{prop:spectral-criterion}
Let $\barronf$ be a continuous real valued function on $\Omega_t$ and let
 \begin{equation}
S[\barronf] := \inf_{\barronf_e} \int_{\mathbb{R}^{\Nt}} \|\omega\|^2_1\,|\widehat{\barronf}_e(\omega)|\,d\omega
\label{eq:spectral-norm}
\end{equation}
where the infimum is taken over all extensions $\barronf_e$ of $\barronf$ to $\mathbb{R}^{\Nt}$ and \(\widehat{\barronf}_e(\omega) = \int_{\mathbb{R}^{\Nt}} \barronf_e(t)e^{-i\omega\cdot t}\,dt\) denotes the Fourier transform of \(\barronf_e\). 
If $S[\barronf]$ is finite, then \(\barronf\in\mathcal{B}(\Omega_t)\). Moreover we have the bound
\begin{equation}
\|\barronf\|_{\mathcal{B}(\Omega_t)} \lesssim S[\barronf]+\|\nabla_t \barronf(0)\|+|\barronf(0)|
\end{equation}
\end{prop}
A direct consequence of Proposition~\ref{prop:spectral-criterion} is that sufficiently regular Sobolev functions belong to the Barron space.
\begin{lem}[Sobolev-to-Barron Embedding]
    \label{lem:sobolev-to-barron}
If $\barronf\in H^l(\Omega_t)$ with $l>\Nt/2+2$ then $\barronf\in\mathcal{B}(\Omega_t)$.
\end{lem}

\begin{rem}[Regularity Requirement]
\label{rem:sobolev-barron-discussion}
The condition $l>\Nt/2+2$ in Lemma~\ref{lem:sobolev-to-barron} ensures, via a Sobolev embedding argument, that sufficiently smooth functions belong to the Barron space. Although this requirement becomes more restrictive as $\Nt$ increases, it is not a limitation in the present setting. Indeed, by Lemma~\ref{lem:coef-Hk}, the finite element coefficients $u_k(t)$ belong to $H^k(\Omega_t)$ for all $k\ge 1$, and hence the condition is satisfied for any finite $\Nt$ by choosing $k$ sufficiently large.

More generally, if such high Sobolev regularity is not available, membership in the Barron space is no longer guaranteed by this argument. In that case, the dimension-independent approximation rates associated with Barron functions may deteriorate, and global surrogate models such as ELMs may require larger feature dimensions or alternative constructions to achieve comparable accuracy.
\end{rem}

\begin{proof}
First we apply Stein's extension theorem to extend $\barronf$ to $\mathbb{R}^{\Nt}$, call this extension $\barronf_e$. The extension operator is bounded, thus $\|\barronf_e\|_{H^l(\mathbb{R}^{\Nt})} \lesssim \|\barronf\|_{H^l(\Omega_t)}$. 
Then we note that
\begin{align}
        S[\barronf] \leq \int_{\mathbb{R}^{\Nt}} \|\omega\|^2_1\,|\widehat{\barronf}_e(\omega)|\,d\omega &=  \int_{\mathbb{R}^{\Nt}} \|\omega\|^2_1(1+\|\omega\|^2)^{-l/2}\,(1+\|\omega\|^2)^{l/2}|\widehat{\barronf}_e(\omega)|\,d\omega \\
&\lesssim \bigg(\int_{\mathbb{R}^{\Nt}} \|\omega\|_1^4 (1+\|\omega\|^2)^{-l}d\omega \bigg)^{1/2} \|\barronf_e\|_{H^l(\mathbb{R}^{\Nt})}\\
&\lesssim \bigg(\int_{\mathbb{R}^{\Nt}} \|\omega\|_1^4 (1+\|\omega\|^2)^{-l}d\omega \bigg)^{1/2} \|\barronf\|_{H^l(\Omega_t)}
\end{align}
The integral on the last line is finite if and only if $l>\Nt/2+2$. Therefore, if $l>\Nt/2+2$ then
\begin{align}
    S[\barronf]\lesssim \|\barronf\|_{H^l(\Omega_t)}
\end{align}
We conclude that $\barronf\in\mathcal{B}(\Omega_t)$ by invoking Proposition~\ref{prop:spectral-criterion}.
\end{proof}
\begin{cor}[Finite Element Coefficients are Barron Functions]
    \label{cor:finite-element-coefficients-barron}
The finite element coefficients $u_k(t)$ are in $\mathcal{B}(\Omega_t)$. Moreover, we have the bound
\begin{equation}
    \|u_k\|_{\mathcal{B}(\Omega_t)} \lesssim \|u_k\|_{H^l(\Omega_t)}
\end{equation}
for any integer $l> \frac{\Nt}{2}+2$.
\end{cor}
\begin{proof}
By Lemma~\ref{lem:coef-Hk} we have that $u_k\in H^l(\Omega_t)$ for any integer $l\ge 1$, so in particular for any integer $l> \frac{\Nt}{2}+2$. Thus, $u_k\in \mathcal{B}(\Omega_t)$ by Lemma~\ref{lem:sobolev-to-barron}.
To prove the bound, we first note that $S[u_k]\lesssim \|u_k\|_{H^l(\Omega_t)}$ for any $l> \frac{\Nt}{2}+2$ by the argument in the proof of Lemma~\ref{lem:sobolev-to-barron}.
Therefore, from the second part of Proposition~\ref{prop:spectral-criterion} we have that
\begin{align}
\|u_k\|_{\mathcal{B}(\Omega_t)} \lesssim S[u_k]+\|\nabla_t u_k(0)\|_1+|u_k(0)|\lesssim \|u_k\|_{H^l(\Omega_t)} + \|\nabla_t u_k(0)\|_1+|u_k(0)|
\end{align}
The Sobolev embedding theorem yields in particular that $u_k \in C^1(\Omega_t)$ with $\|u_k\|_{C^1(\Omega_t)}\lesssim \|u_k\|_{H^l(\Omega_t)}$ for any $l> \frac{\Nt}{2}+2$. We conclude that
\begin{align}
    \|u_k\|_{\mathcal{B}(\Omega_t)} \lesssim \|u_k\|_{H^l(\Omega_t)}
\end{align}
for any $l> \frac{\Nt}{2}+2$.
\end{proof}

Theorems~\ref{thm:Approximation in L^2-norm} and~\ref{thm:Approximation inmax-norm} provide approximation rates for optimally trained two-layer networks. For interpolating ELMs, we do not assume such training but instead rely on the following assumption, which combines random-feature approximation of Barron functions with stability of the minimum-norm interpolant.

\begin{ass}[Random-Feature Approximation and Stable Interpolation for ELMs]
\label{ass:elm-barron-stable}
Let $\sigma$ be the ReLU activation and let the ELM features be
\begin{align}
\psi_m(t)=\sigma(b_m\cdot t + c_m), \qquad m=1,\dots,M,
\end{align}
with $(b_m,c_m)$ sampled i.i.d.\ from a probability distribution $\pi$ on
$S^{\Nt-1}\times[-C,C]$. Let $\{t_j^*\}_{j=1}^J\subset\Omega_t$ be the interpolation
points and let $\Psi\in\mathbb{R}^{J\times M}$ be the feature matrix
$\Psi_{j,m}=\psi_m(t_j^*)$. Assume an overparameterized regime $M\ge \kappa J$ for some fixed $\kappa>1$, and fix $0<\eta<1$.

With probability at least $1-\eta$ over the draw of $\{(b_m,c_m)\}_{m=1}^M$, the
following hold simultaneously:

\proofparagraph{(i) Random-Feature Approximation of Barron Functions.}
For every $\barronf\in\mathcal{B}(\Omega_t)$, there exists a coefficient vector
$w^\ast\in\mathbb{R}^M$ such that
\begin{align}
\Big\|\barronf-\sum_{m=1}^M w_m^\ast \psi_m\Big\|_{L^2(\Omega_t)}
&\le C_{\mathrm{rf}}\,
\frac{\|\barronf\|_{\mathcal{B}(\Omega_t)}\,\sqrt{\log(2/\eta)}}{\sqrt{M}}
\label{eq:rf-best-L2}\\
\Big\|\barronf-\sum_{m=1}^M w_m^\ast \psi_m\Big\|_{L^\infty(\Omega_t)}
&\le C_{\mathrm{rf}}\,
\frac{\sqrt{\Nt+1}}{\sqrt{M}}\,
\|\barronf\|_{\mathcal{B}(\Omega_t)}\,\sqrt{\log(2/\eta)}
\label{eq:rf-best-Linfty}
\end{align}

\proofparagraph{(ii) Quasi-Optimality of the Minimum-Norm Interpolant.}
The matrix $\Psi$ has full row rank ($M\ge J$), and the minimum-norm interpolant
\begin{align}
\hat \barronf(t)=\sum_{m=1}^M w_m \psi_m(t), \qquad
w=\argmin\{\|w\|_2 : \Psi w = (\barronf(t_1^*),\dots,\barronf(t_J^*))^\top\}
\end{align}
satisfies the quasi-optimality estimate
\begin{align}
\|\barronf-\hat \barronf\|_{L^2(\Omega_t)}
\le C_{\mathrm{stab}}\,
\inf_{v\in\mathrm{span}\{\psi_m\}_{m=1}^M}\|\barronf-v\|_{L^2(\Omega_t)}
\end{align}
and analogously in $L^\infty(\Omega_t)$.
Moreover, we assume that for sequences of discretizations with $J\to\infty$ and $M\ge \kappa J$ (for the fixed $\kappa>1$ above), the stability constant $C_{\mathrm{stab}}$ remains bounded; it may depend on $\kappa$ and on the sampling strategy for $\{t_j^*\}$, but is uniform along such sequences.
\end{ass}

\begin{rem}[Justification and Interpretation of Assumption~\ref{ass:elm-barron-stable}]
\label{rem:elm-barron-stable-refs}
The purpose of Assumption~\ref{ass:elm-barron-stable} is to connect available approximation results for two-layer neural networks with the interpolating ELM construction used in this work. In the Barron-space approximation results recalled above, the approximation is achieved by a two-layer network in which both the internal parameters and the outer coefficients are chosen to minimize the approximation error, typically measured in \(L^2(\Omega_t)\). In contrast, in the interpolating ELM setting considered here, the internal parameters are sampled randomly and then fixed, while the outer coefficients are determined by minimum-norm interpolation at the points \(\{t_j^*\}_{j=1}^J\).

These two differences define the gap between the classical Barron-space setting and the interpolating ELM construction, and Assumption~\ref{ass:elm-barron-stable} is precisely designed to bridge this gap. Part~(i) postulates that random sampling of the internal parameters still yields, with high probability, a random-feature space that achieves the same $M^{-1/2}$ approximation rate as in the Barron theory. Part~(ii) postulates that passing from best approximation in this random-feature space to the minimum-norm interpolant does not significantly deteriorate the error. Specifically, the two parts of the assumption may be interpreted and justified as follows.

\proofparagraph{(i).}
The approximation property in part~(i) reflects the fact that functions in the Barron space admit integral representations in terms of ReLU ridge functions. Sampling the internal parameters \((b_m,c_m)\) i.i.d.\ corresponds to a Monte Carlo discretization of such representations and yields dimension-independent \(M^{-1/2}\) approximation rates with high probability. Results of this type go back to Barron and are now standard in the random-feature literature, see, e.g., \cite{barron2002universal,ma2022barron,liu2025integral}.

\proofparagraph{(ii).}
Part~(ii) postulates a stability/quasi-optimality property of the minimum-norm interpolant in the overparameterized regime. Concretely, it is intended to hold when the number of features is sufficiently large relative to the number of interpolation constraints, e.g.\ \(M\ge \kappa J\) for some fixed \(\kappa>1\), and when the interpolation points are sufficiently well distributed in \(\Omega_t\). In this regime, one expects the feature matrix \(\Psi\) to have favorable conditioning properties, so that the minimum-norm interpolant behaves comparably to the best approximation in the span of the sampled features. Related stability and generalization properties of minimum-norm interpolants in overparameterized random-feature, kernel, and neural-network models have been studied in, for example, \cite{ma2019generalization,liang2020multiple, doi:10.1142/S0219530522400115}.

We emphasize that the \(L^2(\Omega_t)\)-version is the primary one from the approximation-theoretic point of view. The analogous assumption in \(L^\infty(\Omega_t)\) is included because it is needed for the max-norm error estimate in Theorem~\ref{thm:error-estimates-elm-high-dim}.

In practice, these properties are typically assessed indirectly through the observed stability and convergence behavior of the surrogate.
\end{rem}

\begin{thm}[Error Estimates: High Dimensional Case]
    \label{thm:error-estimates-elm-high-dim}
    Fix $0<\eta<1$ and assume that Assumption~\ref{ass:elm-barron-stable} holds.
Then, with probability at least $1-\eta$, we have the error estimates (with constants possibly depending on $\eta$ through $\sqrt{\log(2/\eta)}$)
    \begin{align}
        \label{eq:error-estimate ELM L2}
        \|u - \uh\|_{\mathcal{O}} \lesssim h_x^s \|u\|_{L^2(\Omega_t;H^s(\Omega_x))} + \frac{1}{\sqrt{M}}\Bigl(\|f\|_{W'}+\sum_{i=1}^{\Nt}\|\partial_{t_i}f\|_{W'}\Bigr)
    \end{align}
    and
    \begin{align}
        \label{eq:error-estimate ELM Linfty}
            \|u - \uh\|_{L^{\infty}(\Omega_t;L^2(\Omega_x))} \lesssim h_x^s \|u\|_{L^\infty(\Omega_t;H^s(\Omega_x))} + \frac{\sqrt{\Nt+1}}{\sqrt{M}}\Bigl(\|f\|_{W'}+\sum_{i=1}^{\Nt}\|\partial_{t_i}f\|_{W'}\Bigr)
    \end{align}
    where $s=\min(p+1,m)$, $p$ is the degree of the spatial FEM space $V_{h_x}$, $m$ is the regularity of the weak solution $u(\cdot,t)\in H^m(\Omega_x)$, i.e., 
    $m=1$ for the assumption $f\in W'$ and $m=2$ if we have the 
    stronger assumption $f\in L^2(\Omega_x)$.

Under the regime $M\ge \kappa J$, the term $M^{-1/2}$ corresponds (up to constants depending on $\kappa$) to an overall interpolation error of order $J^{-1/2}$ as $J,M\to\infty$.
\end{thm}
\begin{proof}
Throughout the proof we work on the event of probability at least $1-\eta$ on which Assumption~\ref{ass:elm-barron-stable} holds, and consider sequences with $J\to\infty$ and $M\ge \kappa J$ along which $C_{\mathrm{stab}}$ remains bounded (hence all hidden constants below are uniform along such sequences).
    We will decompose $u - \uh$ as follows
    \begin{align}
        \label{eq:error decomposition}
        u - \uh = (u - u_{h_x}) + (u_{h_x}-\uh)
    \end{align}
    and give bounds for the terms on the right in the two norms. We start by bounding the terms for fixed $t$. To that end, let $t\in \Omega_t$.
For the first term, by standard finite element estimates, we have
    \begin{align}
        \label{eq:spatial L2_x bound ELM}
        \|u(\cdot,t) - u_{h_x}(\cdot,t)\|_{\Omega_x} \lesssim h_x^{s} \|u(\cdot,t)\|_{H^s(\Omega_x)}
    \end{align}
For the second term, we note that
\begin{align}
    u_{h_x}(x,t)- \uh(x,t) = \sum_{k=1}^{K} (u_k(t) - \hat{u}_k(t)) \varphi_k(x)        
    \end{align}
    and as $\mathcal{T}_{h_x}$ is shape-regular and quasi-uniform by assumption, we get
    \begin{align}
        \label{eq:shape-regular ineq}
    \|u_{h_x}(\cdot,t)- \uh(\cdot,t)\|_{\Omega_x} \lesssim h_x^{d/2}\Bigg(\sum_{k=1}^K(u_k(t) - \hat{u}_k(t))^2\Bigg)^{1/2}
    \end{align}
\proofparagraph{$L^2$-Estimate.}
    By first squaring \eqref{eq:spatial L2_x bound ELM}, then integrating over $\Omega_t$ and finally taking square roots we obtain 
    \begin{align}
        \label{eq:spatial L2 bound ELM}
        \|u - u_{h_x}\|_{\mathcal{O}} \lesssim h_x^{s} \|u\|_{L^2(\Omega_t;H^s(\Omega_x))}
    \end{align}
Now, squaring \eqref{eq:shape-regular ineq} and integrating over $\Omega_t$ yields
\begin{align}
    \begin{split}
    \|u_{h_x} - \uh\|^2_{\mathcal{O}} &= \int_{\Omega_t}\|u_{h_x}(\cdot,t)- \uh(\cdot,t)\|^2_{\Omega_x}\dt \lesssim h_x^d\int_{\Omega_t}\sum_{k=1}^K(u_k(t) - \hat{u}_k(t))^2 dt\\
    &=h_x^d\sum_{k=1}^{K}\int_{\Omega_t}(u_k(t) - \hat{u}_k(t))^2 dt=h_x^d\sum_{k=1}^K\|u_k - \hat{u}_k\|_{L^2(\Omega_t)}^2
    \end{split}
\end{align}
By Corollary~\ref{cor:finite-element-coefficients-barron} we have $u_k\in\mathcal{B}(\Omega_t)$.
Here $l>\Nt/2+2$ is chosen such that $u_k\in H^l(\Omega_t)$, as required by Corollary~\ref{cor:finite-element-coefficients-barron}.
By Assumption~\ref{ass:elm-barron-stable}(i), for each $k$ there exists a
random-feature approximation $v_{k,M}\in\mathrm{span}\{\psi_m\}_{m=1}^M$
such that
\begin{align}   
\|u_k - v_{k,M}\|_{L^2(\Omega_t)}
\lesssim \frac{\|u_k\|_{\mathcal{B}(\Omega_t)}}{\sqrt{M}}
\end{align}
By Assumption~\ref{ass:elm-barron-stable}(ii), the minimum-norm interpolant
$\hat u_k$ is quasi-optimal, and therefore satisfies
\begin{align}
\|u_k - \hat u_k\|_{L^2(\Omega_t)}
\lesssim \|u_k - v_{k,M}\|_{L^2(\Omega_t)}
\lesssim \frac{\|u_k\|_{\mathcal{B}(\Omega_t)}}{\sqrt{M}}
\end{align}
Using Corollary~\ref{cor:finite-element-coefficients-barron}, this yields
\begin{align}
\|u_k - \hat u_k\|_{L^2(\Omega_t)}^2
\lesssim \frac{\|u_k\|_{H^l(\Omega_t)}^2}{M}
\end{align}
Thus,
\begin{align}
    \|u_{h_x} - \uh\|^2_{\mathcal{O}} \lesssim \frac{h_x^d}{M}\sum_{k=1}^K\|u_k\|_{H^l(\Omega_t)}^2=\frac{h_x^d}{M}\|\mathbf{u}\|_{H^l(\Omega_t;\ell^2)}^2
\end{align}
Taking square roots and applying Lemma~\ref{lem:coef-Hk} gives
\begin{align}
    \label{eq:u_hx-uh-bound L2 ELM}
    \|u_{h_x} - \uh\|_{\mathcal{O}} \lesssim \frac{1}{\sqrt{M}}\Bigl(\|f\|_{W'}+\sum_{i=1}^{\Nt}\|\partial_{t_i}f\|_{W'}\Bigr)
\end{align}
The triangle inequality on \eqref{eq:error decomposition} with \eqref{eq:spatial L2 bound ELM} and \eqref{eq:u_hx-uh-bound L2 ELM} gives \eqref{eq:error-estimate ELM L2}.
\proofparagraph{Max Norm Estimate.}
If we take the essential supremum over $t\in \Omega_t$ in \eqref{eq:spatial L2_x bound ELM} we obtain
\begin{align}
    \label{eq:spatial Linfty bound ELM}
    \|u - u_{h_x}\|_{L^{\infty}(\Omega_t;L^2(\Omega_x))} \lesssim h_x^{s} \|u\|_{L^\infty(\Omega_t;H^s(\Omega_x))}
\end{align}
Similarly, taking the essential supremum over $t\in \Omega_t$ in \eqref{eq:shape-regular ineq} and squaring yields
\begin{align}
\|u_{h_x}(\cdot,t)- \uh(\cdot,t)\|^2_{L^{\infty}(\Omega_t;L^2(\Omega_x))} &\lesssim h_x^{d} \operatorname*{ess\,sup}_{t\in\Omega_t} \sum_{k=1}^K(u_k(t) - \hat{u}_k(t))^2\\
&\leq h_x^{d}  \sum_{k=1}^K\operatorname*{ess\,sup}_{t\in\Omega_t}|u_k(t) - \hat{u}_k(t)|^2\\
&= h_x^{d}  \sum_{k=1}^K\bigg(\operatorname*{ess\,sup}_{t\in\Omega_t}|u_k(t) - \hat{u}_k(t)|\bigg)^2\\
&=  h_x^{d} \sum_{k=1}^K\|u_k - \hat{u}_k\|^2_{L^{\infty}(\Omega_t)}
\end{align}
By Assumption~\ref{ass:elm-barron-stable}(i), for each $k$ there exists a
random-feature approximation $v_{k,M}\in\mathrm{span}\{\psi_m\}_{m=1}^M$
such that
\begin{align}
\|u_k - v_{k,M}\|_{L^\infty(\Omega_t)}
\lesssim \frac{\sqrt{\Nt+1}}{\sqrt{M}}\,\|u_k\|_{\mathcal{B}(\Omega_t)}
\end{align}
By Assumption~\ref{ass:elm-barron-stable}(ii), the minimum-norm interpolant
$\hat u_k$ satisfies the same bound, i.e.,
\begin{align}
\|u_k - \hat u_k\|_{L^\infty(\Omega_t)}
\lesssim \frac{\sqrt{\Nt+1}}{\sqrt{M}}\,\|u_k\|_{\mathcal{B}(\Omega_t)}
\end{align}
Using Corollary~\ref{cor:finite-element-coefficients-barron}, we conclude
\begin{align}
\|u_k - \hat u_k\|_{L^\infty(\Omega_t)}^2
\lesssim \frac{\Nt+1}{M}\,\|u_k\|_{H^l(\Omega_t)}^2
\end{align}
Hence, 
\begin{align}
    \|u_{h_x}- \uh\|^2_{L^{\infty}(\Omega_t;L^2(\Omega_x))}&\lesssim \frac{h_x^{d}(\Nt+1)}{M}  \sum_{k=1}^K \|u_k\|^2_{H^l(\Omega_t)}\\
    &= \frac{h_x^{d}(\Nt+1)}{M} \|\mathbf{u}\|_{H^l(\Omega_t;\ell^2)}^2
\end{align}
Taking square roots and applying Lemma~\ref{lem:coef-Hk} gives
\begin{align}
    \label{eq:u_hx-uh-bound Linfty ELM}
    \|u_{h_x} - \uh\|_{L^{\infty}(\Omega_t;L^2(\Omega_x))} \lesssim \frac{\sqrt{\Nt+1}}{\sqrt{M}}\Bigl(\|f\|_{W'}+\sum_{i=1}^{\Nt}\|\partial_{t_i}f\|_{W'}\Bigr)
\end{align}
The triangle inequality on \eqref{eq:error decomposition} with \eqref{eq:spatial Linfty bound ELM} and \eqref{eq:u_hx-uh-bound Linfty ELM} gives \eqref{eq:error-estimate ELM Linfty}.
\end{proof}

\section{\boldmath Parameter Reconstruction in QPAT via ROMs}
\label{sec:qpat}

We apply the reduced-order modeling framework to an inverse problem in quantitative photoacoustic tomography (QPAT). The goal is to reconstruct a spatially varying parametrized potential from internal measurements of absorbed energy governed by an elliptic PDE. The reduced-order surrogate enables efficient inversion even in high-dimensional parameter spaces.

\subsection{Model Problem}
\label{subsec:qpat-model}

Let \(\Omega_x \subset \mathbb{R}^d\) be a bounded Lipschitz domain. For each \(t\in\Omega_t\), consider
\begin{equation}
\mu(x,t) = \mu_0(x)+\sum_{j=1}^{\Nt} t_j\,\mu_j(x)
\label{eq:mu-param}
\end{equation}
where \(\{\mu_j\}_{j=0}^{\Nt}\subset L^\infty(\Omega_x)\) are fixed linearly independent functions. The associated state \(u^{\mathrm{phys}}(\cdot,t)\) solves
\begin{equation}
-\Delta u^{\mathrm{phys}}(x,t) + \mu(x,t)\,u^{\mathrm{phys}}(x,t) = 0 \quad \text{in } \Omega_x,\qquad u^{\mathrm{phys}} = g \quad \text{on } \partial\Omega_x
\label{eq:PDE-QPAT}
\end{equation}
with given boundary data \(g \in H^{1/2}(\partial\Omega_x)\) satisfying
\begin{equation}
\min_{x\in\partial\Omega_x} g(x) \ge c > 0
\label{eq:g-positivity}
\end{equation}

To make the analysis of Section~\ref{sec:interpolation-rom} applicable, we reduce \eqref{eq:PDE-QPAT} to homogeneous Dirichlet boundary conditions using a lifting.
Let $g^{\mathrm{lift}}\in H^{1}(\Omega_x)$ satisfy $\gamma(g^{\mathrm{lift}})=g$ on $\partial\Omega_x$, where $\gamma$ denotes the trace operator, and define
\begin{equation}
u(x,t):=u^{\mathrm{phys}}(x,t)-g^{\mathrm{lift}}(x)
\label{eq:tilde-u-def}
\end{equation}
Then $u(\cdot,t)\in H_0^1(\Omega_x)$ solves
\begin{equation}
-\Delta u(x,t)+\mu(x,t)\,u(x,t)=f_g(x,t)
\quad\text{in }\Omega_x,\qquad u=0\ \text{on }\partial\Omega_x
\label{eq:PDE-QPAT-lifted}
\end{equation}
with the (parameter-dependent) right-hand side
\begin{equation}
 f_g(x,t):=\Delta g^{\mathrm{lift}}(x)-\mu(x,t)\,g^{\mathrm{lift}}(x)
\label{eq:fg-def}
\end{equation}

Since $\mu(x,t)$ is affine in $t$, so is $f_g(\cdot,t)$, which means only first-order parameter derivatives are nonzero, and the parametric regularity and approximation results of Section~\ref{sec:interpolation-rom} apply directly.

\paragraph{\boldmath Internal Measurements.}
Below, \(Q\) denotes the \emph{measurement operator}. The measurable quantity is a projection of the absorbed energy,
\begin{equation}
Q[\mu\,u^{\mathrm{phys}}]
\label{eq:internal-data}
\end{equation}
where \(Q:L^2(\Omega_x)\to \mathcal{M}_m\) is a projection onto a finite-dimensional \emph{measurement space} \(\mathcal{M}_m\subset L^2(\Omega_x)\) of dimension \(N_m\). We require \(Q\) to be $L^2(\Omega_x)$-stable in the sense that
\begin{equation}
\|Q[f]\|_{\Omega_x} \lesssim \|f\|_{\Omega_x}
\qquad \text{for all } f\in L^2(\Omega_x)
\label{eq:Q-stability}
\end{equation}

\subsection{Discretization and Parametric Representation}
\label{subsec:qpat-discrete}

We solve the homogenized formulation \eqref{eq:PDE-QPAT-lifted} using the ROMs from Section~\ref{subsec:param-interp} (and recover $u^{\mathrm{phys}}$ by adding the lifting $g^{\mathrm{lift}}$ when needed).
Hence, the discrete parameter dependent solution \(\uh\) has the expansion 
\begin{equation}
\uh(x,t)=\sum_{k=1}^{K} \widehat{u}_k(t)\,\varphi_k(x)
\label{eq:uhN-expansion}
\end{equation}
We write $\uh^{\mathrm{phys}}:=\uh+g^{\mathrm{lift}}$ for the lifted surrogate when forming measurements.
Here $\widehat{u}_k(t)$ are the $t$-interpolated finite element coefficients, using classical interpolation for low-dimensional parameter spaces and the ELM representation for high-dimensional parameter spaces.

\paragraph{Efficient Measurement Evaluation.}
While we assume $\uh^{\mathrm{phys}}:=\uh+g^{\mathrm{lift}}$ for analysis purposes, we in practice do not impose the Dirichlet boundary conditions using a lifting, but rather impose them directly on the discrete solution on the linear algebra level. Hence, 
\begin{equation}
\uh^{\mathrm{phys}}(x,t)=\sum_{k=1}^{K} \widehat{u}^{\mathrm{phys}}_k(t)\,\varphi_k(x)
\label{eq:uh-discrete-no-lift}
\end{equation}
with coefficients for the discrete solution with the non-homogeneous boundary condition.
For notational simplicity, we henceforth drop the superscript ${\mathrm{phys}}$ on the coefficient vector and write
$\widehat u_k(t):=\widehat u_k^{\mathrm{phys}}(t)$ and $\widehat u(t):=(\widehat u_k(t))_{k=1}^K$.
In order to efficiently evaluate the discrete measurement
$Q\bigl[\mu\,\uh^{\mathrm{phys}} \bigr]$
in this situation,
we precompute the \(t\)-independent projections
\begin{equation}
Q_{i,k}:=Q[\mu_i\,\varphi_k] \qquad \text{for } i=0,1,\dots,\Nt,\quad \text{and } k=1,\dots,K
\label{eq:Qik-def}
\end{equation}
with \(\mu_0\) denoting any fixed baseline component if used (otherwise drop \(i=0\)). Then
\begin{equation}
Q\bigl[\mu(\cdot,t)\,\uh^{\mathrm{phys}}(\cdot,t)\bigr]
= \sum_{k=1}^{K} \widehat{u}_k(t)\Bigl(Q_{0,k}+\sum_{i=1}^{\Nt} t_i\,Q_{i,k}\Bigr)
\label{eq:Q-of-t-compact}
\end{equation}

\subsection{Parameter Identification from Internal Data}
\label{subsec:qpat-id}

Given internal data \(Q_{\mathrm{obs}}\in \mathcal{M}_m\), we estimate \(t\) by minimizing the residual functional
\begin{equation}
\hat{t} = \operatorname*{arg\,min}_{t\in\Omega_t} \bigl\|\,Q_{\mathrm{obs}} - Q\bigl[\mu(t)\,\uh^{\mathrm{phys}}(t)\bigr]\,\bigr\|_{\Omega_x}
\label{eq:inverse-objective}
\end{equation}
and recover
\begin{equation}
\widehat{\mu}(x) =\mu_0(x)+ \sum_{j=1}^{\Nt} \hat{t}_j\,\mu_j(x)
\label{eq:recovered-potential}
\end{equation}
For compactness, write \(Q(t):=Q\bigl[\mu(t)\,\uh^{\mathrm{phys}}(t)\bigr]\) and define the loss
\begin{equation}
L(t)=\frac{1}{2}\bigl\|Q_{\mathrm{obs}}-Q(t)\bigr\|_{\Omega_x}^{2}
\label{eq:loss-def}
\end{equation}
A gradient step takes the form
\begin{equation}
t_{\ell+1}=t_{\ell}-\steplen\,\nabla_t L(t_{\ell})
\label{eq:gd-update}
\end{equation}
with step size \(\steplen\), where
\begin{equation}
\nabla_t L(t)= -\bigl(Q_{\mathrm{obs}}-Q(t),\,\nabla_t Q(t)\bigr)_{\Omega_x}
\label{eq:grad_L}
\end{equation}
The parameter \(\steplen\) can either be explicitly set or found by a suitable line-search algorithm. Using \eqref{eq:Q-of-t-compact} and the product rule,
\begin{equation}
\nabla_t Q(t)=\sum_{k=1}^{K}\nabla_t \widehat{u}_k(t)\Bigl(Q_{0,k}+\sum_{i=1}^{\Nt} t_i\,Q_{i,k}\Bigr)
+\sum_{k=1}^{K} \widehat{u}_k(t)\,\bigl(Q_{1,k},\dots,Q_{\Nt,k}\bigr)^{\top}
\label{eq:grad_Q}
\end{equation}
To evaluate \(\nabla_t \widehat{u}_k(t)\) efficiently, we use the ELM representation in \eqref{eq:elm-matix-form} to obtain
\begin{equation}
\nabla_t \widehat{u}_k(t)=B^{\top}D(t)\,w_k
\label{eq:grad_u}
\end{equation}
with \(D(t)=\operatorname{diag}\bigl(\sigma'(b_m\cdot t + c_m)\bigr)\). Define \(z(t)\in\mathbb{R}^{K}\) and \(Z(t)\in\mathbb{R}^{\Nt\times K}\) by
\begin{equation}
z_k(t)=\bigl(Q_{\mathrm{obs}}-Q(t),\,Q_{0,k}+\sum_{i=1}^{\Nt} t_i\,Q_{i,k}\bigr)_{\Omega_x}
\label{eq:z-def}
\end{equation}
\begin{equation}
Z_{i,k}(t)=\bigl(Q_{\mathrm{obs}}-Q(t),\,Q_{i,k}\bigr)_{\Omega_x}
\label{eq:Z-def}
\end{equation}
Then
\begin{equation}
\nabla_t L(t)=-B^{\top}D(t)\,W^{\top}z(t)-Z(t)\,\widehat u(t)
\label{eq:formula_gradL}
\end{equation}
so each gradient step requires one forward ELM pass plus linear algebra with the precomputed \(\{Q_{i,k}\}\).

\subsection{Error Analysis of the Potential Reconstruction}
\label{sec:error-analysis-potential-reconstruction}

Set $u^{\mathrm{phys}}(\cdot,t):=u(\cdot,t)+g^{\mathrm{lift}}$ and $\uh^{\mathrm{phys}}(\cdot,t):=\uh(\cdot,t)+g^{\mathrm{lift}}$.
Let \(t^\dagger\) be the true parameter, and let quantities with superscript $\dagger$ denote the quantities at the true parameter value, for instance \(\mu^\dagger := \mu(\cdot, t^\dagger)\).
Consider the estimator
\begin{equation}
\hat{t} := \operatorname*{arg\,min}_{t\in\Omega_t} \bigl\|\,Q\bigl[\mu(\cdot,t)\,\uh^{\mathrm{phys}}(\cdot,t)\bigr] - Q\bigl[\mu^\dagger\,u^{\dagger,\mathrm{phys}}\bigr]\,\bigr\|_{\Omega_x}
\label{eq:min-problem-error}
\end{equation}
and let quantities with hat denote the quantities at the estimated parameter value, for instance \(\widehat{\mu}=\mu(\cdot,\hat{t})\).

Assume that there exists a neighborhood $U^\dagger \subset \Omega_t$ of $t^\dagger$ such that for all
$t_1,t_2\in U^\dagger$ the stability estimate
\begin{equation}
\|\mu(\cdot,t_1)-\mu(\cdot,t_2)\|_{L^\infty(\Omega_x)} \le C_{\mathrm{stab}} \bigl\|\,Q\bigl[\mu(\cdot,t_1) u^{\mathrm{phys}}(\cdot,t_1)\bigr]-Q\bigl[\mu(\cdot,t_2) u^{\mathrm{phys}}(\cdot,t_2)\bigr]\,\bigr\|_{\Omega_x}
\label{eq:stability-estimate}
\end{equation}
holds (for instance under suitable richness conditions on the
measurement space $\mathcal{M}_m$, see \cite[Proposition~1]{AS22}).

\begin{thm}[Potential Reconstruction Error Estimate]
Assume 
$\hat{t} \in U^\dagger$, and that the surrogate error $\|\widehat{u}^{\mathrm{phys}}-\hatuh^{\mathrm{phys}}\|_{\Omega_x}$ is sufficiently small. Then the error between the reconstructed potential
$\widehat\mu$ and the true potential $\mu^\dagger$ satisfies the following bounds:

\proofparagraph{Low-Dimensional Case.} 
If \(\uh=\pi_{h_t} u_{h_x}\), where \(\pi_{h_t}\) is the nodal Lagrange interpolation operator in the parameter domain, the potential reconstruction error is bounded by
\begin{equation}
\|\widehat{\mu}-\mu^\dagger\|_{L^\infty(\Omega_x)} \lesssim \|\mu^\dagger\|_{L^\infty(\Omega_x)}\Bigl(
    h_x^{s} \|u\|_{L^{\infty}(\Omega_t;H^{s}(\Omega_x))} + h_t^{q+1} |u|_{W^{q+1,\infty}(\Omega_t;L^2(\Omega_x))}
    \Bigr)
\label{eq:error-estimate}
\end{equation}
and the parameter reconstruction error is bounded by
\begin{equation}
\|\hat{t}-t^\dagger\|_{\mathbb{R}^{\Nt}} \lesssim \|\mu^\dagger\|_{L^\infty(\Omega_x)}\Bigl(
    h_x^{s} \ \|u\|_{L^{\infty}(\Omega_t;H^{s}(\Omega_x))} + h_t^{q+1} |u|_{W^{q+1,\infty}(\Omega_t;L^2(\Omega_x))}
\Bigr)
\label{eq:explicit-parameter-error}
\end{equation}
    
\proofparagraph{High-Dimensional Case.} 
If $\uh$ is the ELM-based surrogate in the parameter domain and Assumption~\ref{ass:elm-barron-stable} holds,
the potential reconstruction error is bounded by
\begin{equation}
\|\widehat{\mu}-\mu^\dagger\|_{L^\infty(\Omega_x)} \lesssim \|\mu^\dagger\|_{L^\infty(\Omega_x)}\Bigl(
    h_x^s \|u\|_{L^\infty(\Omega_t;H^s(\Omega_x))} + \frac{\sqrt{\Nt+1}}{\sqrt{M}}\Bigl(\|f_g\|_{W'}+\sum_{i=1}^{\Nt}\|\partial_{t_i} f_g\|_{W'}\Bigr)
\Bigr)
\label{eq:error-estimate-elm}
\end{equation}
and the parameter reconstruction error is bounded by
\begin{equation}
\|\hat{t}-t^\dagger\|_{\mathbb{R}^{\Nt}} \lesssim \|\mu^\dagger\|_{L^\infty(\Omega_x)}\Bigl(
    h_x^s \|u\|_{L^\infty(\Omega_t;H^s(\Omega_x))} + \frac{\sqrt{\Nt+1}}{\sqrt{M}}\Bigl(\|f_g\|_{W'}+\sum_{i=1}^{\Nt}\|\partial_{t_i} f_g\|_{W'}\Bigr)
\Bigr)
\label{eq:explicit-parameter-error-elm}
\end{equation}
    
\end{thm}

\begin{proof}
\proofparagraph{Potential Reconstruction.}
Since \(\hat{t}\) is a global minimizer, it satisfies
\begin{equation}
\bigl\|\,Q\bigl[\widehat{\mu}\hatuh^{\mathrm{phys}}\bigr]-Q\bigl[\mu^\dagger u^{\dagger,\mathrm{phys}}\bigr]\,\bigr\|_{\Omega_x}
\le \bigl\|\,Q\bigl[\mu^\dagger \uh^{\dagger,\mathrm{phys}} \bigr]-Q\bigl[\mu^\dagger u^{\dagger,\mathrm{phys}} \bigr]\,\bigr\|_{\Omega_x}
\label{eq:global-min-estimate}
\end{equation}

The stability estimate \eqref{eq:stability-estimate}, the triangle inequality, \eqref{eq:global-min-estimate}, and the \(Q\)-stability \eqref{eq:Q-stability} yield
\begin{align}
    \label{eq:potential-reconstruction-error-estimate-first}
&\|\widehat{\mu}-\mu^\dagger\|_{L^\infty(\Omega_x)}
\nonumber
\\
&\qquad\quad\lesssim
\bigl\|\,Q\bigl[\widehat{\mu}\widehat{u}^{\mathrm{phys}}\bigr]-Q\bigl[\mu^\dagger u^{\dagger,\mathrm{phys}}\bigr]\,\bigr\|_{\Omega_x}
\\&\qquad\quad\le
\bigl\|\,Q\bigl[\widehat{\mu} \hatuh^{\mathrm{phys}}\bigr]-Q\bigl[\mu^\dagger u^{\dagger,\mathrm{phys}}\bigr]\,\bigr\|_{\Omega_x}
+
\bigl\|\,Q\bigl[\widehat{\mu}\widehat{u}^{\mathrm{phys}}\bigr]-Q\bigl[\widehat{\mu} \hatuh^{\mathrm{phys}}\bigr]\,\bigr\|_{\Omega_x}
\\&\qquad\quad\le
\bigl\|\,Q\bigl[\mu^\dagger \uh^{\dagger,\mathrm{phys}} \bigr]-Q\bigl[\mu^\dagger u^{\dagger,\mathrm{phys}} \bigr]\,\bigr\|_{\Omega_x}
+
\bigl\|\,Q\bigl[\widehat{\mu}\widehat{u}^{\mathrm{phys}}\bigr]-Q\bigl[\widehat{\mu} \hatuh^{\mathrm{phys}}\bigr]\,\bigr\|_{\Omega_x}
\\&\qquad\quad\lesssim
\|\mu^\dagger\|_{L^\infty(\Omega_x)} \| \uh^{\dagger,\mathrm{phys}}-u^{\dagger,\mathrm{phys}} \|_{\Omega_x}
+
\|\widehat{\mu}\|_{L^\infty(\Omega_x)} \| \hatuh^{\mathrm{phys}}-\widehat{u}^{\mathrm{phys}} \|_{\Omega_x}
\\&\qquad\quad=
\|\mu^\dagger\|_{L^\infty(\Omega_x)} \| \uh^{\dagger}-u^{\dagger} \|_{\Omega_x}
+
\|\widehat{\mu}\|_{L^\infty(\Omega_x)} \| \hatuh-\widehat{u} \|_{\Omega_x}
\\&\qquad\quad\lesssim
\|\mu^\dagger\|_{L^\infty(\Omega_x)}
\bigl(
\| \uh^{\dagger}-u^{\dagger} \|_{\Omega_x}
+
\| \hatuh-\widehat{u} \|_{\Omega_x}
\bigr)
+
\|\widehat{\mu} - \mu^\dagger\|_{L^\infty(\Omega_x)} \| \hatuh-\widehat{u} \|_{\Omega_x}
\label{eq:potential-reconstruction-error-estimate-last}
\end{align}
where we in the final step used the triangle inequality once more. Moving the last term to the left-hand side, we obtain
\begin{align}
&\|\widehat{\mu}-\mu^\dagger\|_{L^\infty(\Omega_x)}
\bigl(
1 - C_1 \| \widehat{u} - \hatuh \|_{\Omega_x}
\bigr)
\lesssim \|\mu^\dagger\|_{L^\infty(\Omega_x)}
\bigl(
\| \uh^{\dagger}-u^{\dagger} \|_{\Omega_x}
+
\| \widehat{u} - \hatuh \|_{\Omega_x}
\bigr)
\label{eq:mu-error-pre-new}
\end{align}
where $C_1$ is the accumulated hidden constant in \eqref{eq:potential-reconstruction-error-estimate-first}--\eqref{eq:potential-reconstruction-error-estimate-last}.
Given a sufficiently small surrogate error, for instance
$\|\widehat{u}-\hatuh\|_{\Omega_x}<(2C_1)^{-1}$, and bounding the right-hand side surrogate error using the max norm, we have
\begin{equation}
\|\widehat{\mu}-\mu^\dagger\|_{L^\infty(\Omega_x)}
\lesssim \|\mu^\dagger\|_{L^\infty(\Omega_x)}
\|u - \uh\|_{L^{\infty}(\Omega_t;L^2(\Omega_x))} 
\end{equation}
Finally, applying Theorem~\ref{thm:error-estimates-low-dim} in the low-dimensional case, and Theorem~\ref{thm:error-estimates-elm-high-dim} in the high-dimensional case, we obtain the error bounds \eqref{eq:error-estimate} and \eqref{eq:error-estimate-elm}, respectively.

\proofparagraph{Parameter Reconstruction.}
Recall that
\(\mu(\cdot,t)=\mu_0+\sum_{j=1}^{\Nt} t_j\,\mu_j\), hence
\(
\widehat{\mu}-\mu^\dagger=\sum_{j=1}^{\Nt}(\hat{t}_j-t_j^\dagger)\,\mu_j.
\)
Since the map
\(T:\mathbb{R}^{\Nt}\to L^\infty(\Omega_x)\),
\(T(s)=\sum_{j=1}^{\Nt} s_j\mu_j\), is linear and injective by the linear independence of
\(\{\mu_j\}_{j=1}^{\Nt}\), norm equivalence on finite-dimensional spaces implies the existence
of a constant \(C_\mu>0\) such that
\begin{equation}
\|s\|_{\mathbb{R}^{\Nt}}\le C_\mu\,\|T(s)\|_{L^\infty(\Omega_x)}
\qquad\text{for all } s\in\mathbb{R}^{\Nt}
\label{eq:param-map-inverse}
\end{equation}
Applying \eqref{eq:param-map-inverse} with \(s=\hat{t}-t^\dagger\) yields
\[
\|\hat{t}-t^\dagger\|_{\mathbb{R}^{\Nt}}
\lesssim \|\widehat\mu-\mu^\dagger\|_{L^\infty(\Omega_x)}
\]
Combining this with the potential reconstruction bounds proved above gives
\eqref{eq:explicit-parameter-error} in the low-dimensional case and
\eqref{eq:explicit-parameter-error-elm} in the high-dimensional case.
\end{proof}

\section{\boldmath Numerical Experiments}
\label{sec:numerical-experiments}
We now present a series of numerical experiments designed to assess the performance of the proposed reduced-order models. The experiments fall into two categories.

First, we demonstrate the approximation properties of the surrogate both in the low- and high-dimensional settings. In particular, we study the convergence of the model with respect to the finite element mesh size $h_x$ together with the partition resolution $h_t$ of $\Omega_t$ (the low-dimensional cases) or the number of ReLU units $M$ in the ELM (the high-dimensional cases).

Second, we consider the inverse problem, where an unknown parameter vector is reconstructed from pixel-based observations of the absorbed energy. Here we examine how different choices in the observation operator $Q$ affect reconstruction quality.

Throughout this section, $u$ denotes the physical solution satisfying the prescribed Dirichlet boundary data $g$, and $\uh$ denotes its numerical surrogate; for simplicity, we suppress the superscript ${\mathrm{phys}}$ used in Section~4.

\subsection{Convergence}
\label{subsec:numerical-convergence}
Throughout the numerical examples we considered the spatial domain $\Omega_x=[-1,1]^2$ and the FEM-meshes of mesh size $h_x$ were generated using uniform grids of points in $\Omega_x$. To study the convergence, we employed a manufactured solution $u(x,t)$ to \eqref{eq:main}. Once the potentials $\mu_i(x)$ in \eqref{eq:potential} were specified, the corresponding source $f$ and boundary data $g$ were constructed so that $u(x,t)$ satisfied the PDE.

\begin{rem}[Parameter-Dependent Data]
The construction described above generally introduces dependencies of the source term $f$ and boundary data $g$ on the parameter $t$. This is a consequence of enforcing that the manufactured solution $u(x,t)$ satisfies \eqref{eq:main} exactly. In principle, one would prefer an example where $f$ and $g$ depend solely on the spatial variable $x$ (or at least where $f$ remains affine in $t$), as this is the setting covered by the theoretical analysis. However, finding a nontrivial closed-form solution in that case is difficult. For this reason, we adopt the manufactured-solution approach for the convergence study, even though it introduces additional $t$-dependencies not covered by the affine-in-$t$ data assumptions used in the theory.

If one were to extend the theory to non-affine parameter dependence in $f$, the surrogate construction itself would remain unchanged, but the parameter-regularity estimates in theorems~\ref{thm:regularity-Hm} and \ref{thm:Hk-regularity-fem} would have to retain higher-order derivatives $\partial_t^\alpha f$ up to the order under consideration. In particular, the recursive differentiated equations would involve these higher-order terms explicitly, leading to more cumbersome notation and constants, but not to a conceptual change in the approximation framework.
\end{rem}

\paragraph{Low-Dimensional Setting.}
To evaluate the performance in a low-dimensional parameter space $\Omega_t$, we considered the case where $\Nt=2$. The exact solution was prescribed as
\begin{align}
u(x,t)=\sin\!\big(2\pi(t\cdot x)\big)+2\Big(\exp(-5\|x-t\|^2)+\exp(-5\|x+t\|^2)\Big)+1
\end{align}
The potentials in \eqref{eq:potential} were chosen as
\begin{align}
\mu_0=0.1, \quad \mu_1(x)=\exp(-2\|x-p\|^2),\quad \mu_2(x)=\exp(-2\|x+p\|^2) 
\end{align}
where $p=[0.5,0.5]^{\top}$. For the discretization of the parameter domain, we partitioned $\Omega_t$ into a simplicial mesh of mesh size $h_t$ generated from a uniform grid of points. The FEM solution to \eqref{eq:fem} was computed at each node in the mesh, and to evaluate the model at an arbitrary point in $\Omega_t$ we used piecewise-linear interpolation. For each different choice of discretization parameter pairs $(h_x,h_t)$, we computed the $L^2$-error 
\begin{align}
\|u-\uh\|_{\mathcal{O}}
\end{align}
The result can be seen in Figure~\ref{fig:triangulation-convergence_N=2}. We can see that the error levels out when refining $h_t$ or $h_x$ while keeping the other parameter fixed. From the log–log plot of the error versus the spatial mesh size $h_x$ for the finest parameter discretization $h_t$, the observed convergence rate is approximately 1.90, which is close to the theoretical second-order rate. In Table \ref{tab:errors_N=2} we have summarized the relative $L^2$-errors for the different runs. 
\begin{figure}
\centering
\begin{subfigure}[ht]{0.33\textwidth}
\centering
\includegraphics[width=\textwidth]{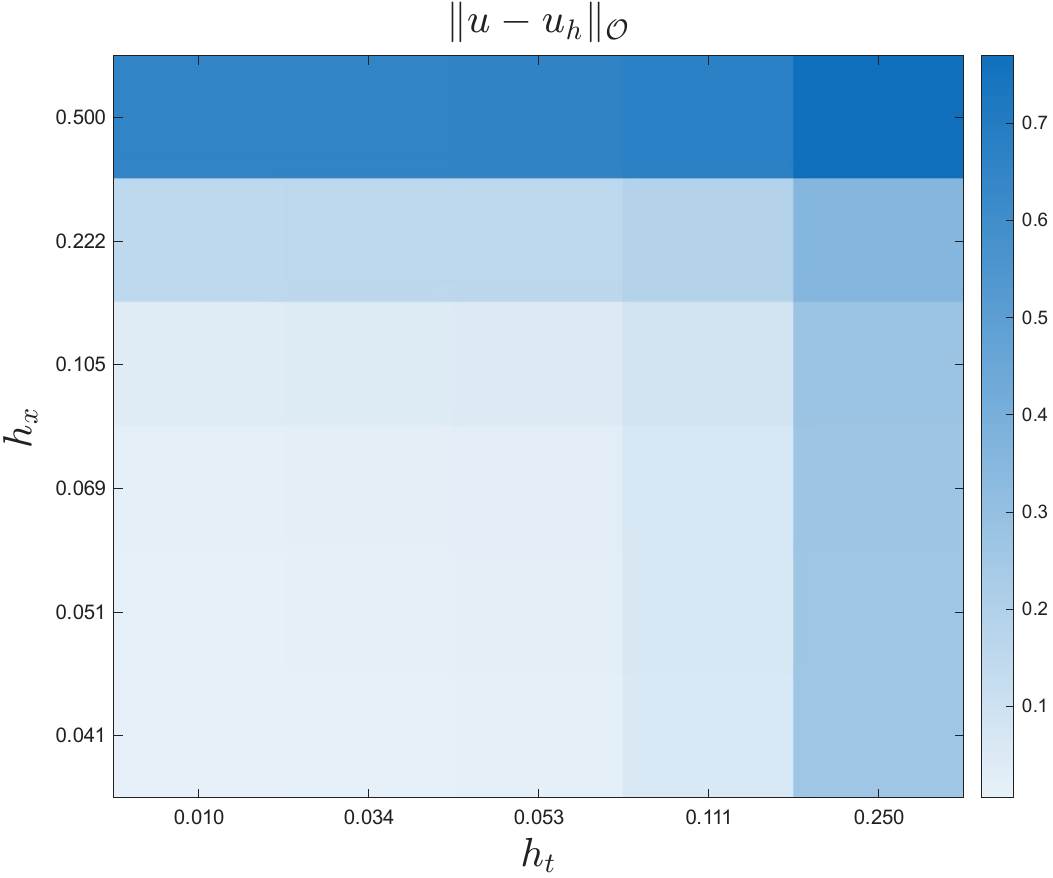}
\caption{}
\end{subfigure}
\begin{subfigure}[ht]{0.32\textwidth}
\centering
\includegraphics[width=\textwidth]{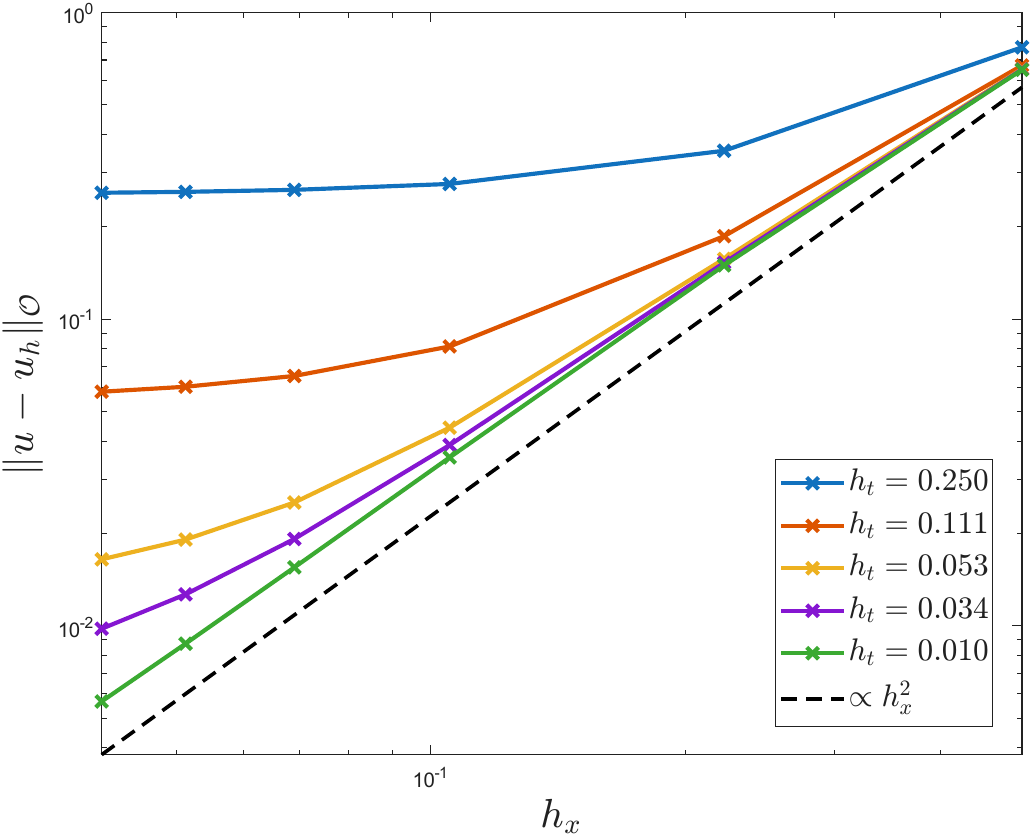}
\caption{ }
\end{subfigure}
\begin{subfigure}[ht]{0.32\textwidth}
\centering
\includegraphics[width=\textwidth]{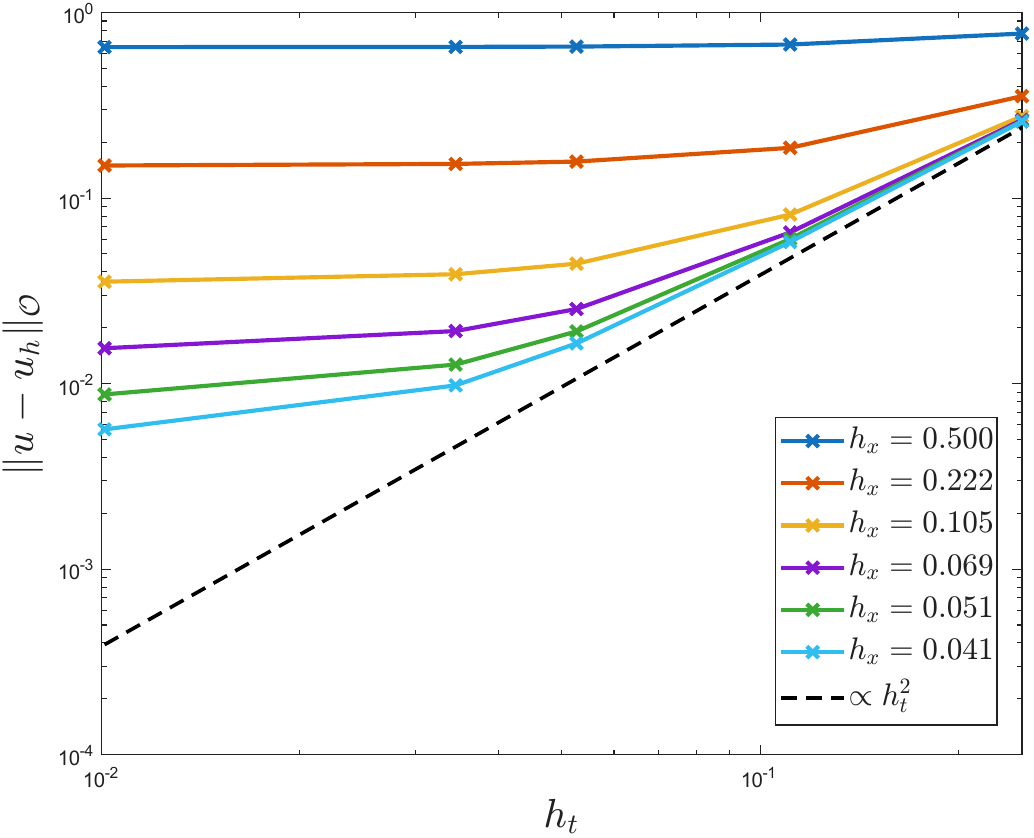}
\caption{ }
\end{subfigure}
\caption{\emph{Convergence in low-dimensional setting.} \textbf{(a)} The heatmap shows approximation error for different values of the pair $(h_x,h_t)$. \textbf{(b)} The figure shows the approximation error as a function of $h_x$ for different choices of $h_t$ including a reference line to indicate the theoretical convergence rate of $2.0$ with respect to $h_x$. \textbf{(c)} The figure shows the approximation error as a function of $h_t$ for different choices of $h_x$ including a reference line to indicate the theoretical convergence rate of $2.0$ with respect to $h_t$.}
\label{fig:triangulation-convergence_N=2}
\end{figure}

\begin{table}
\centering
\caption{\emph{Error in low-dimensional setting.} The relative $L^2$-error for different grids on $\Omega_x$ and $\Omega_t$.}
\label{tab:errors_N=2}
\small
\begin{tabular}{lcccccc}
\toprule
$\Omega_x$-grid & $5\times 5$ & $10\times 10$ & $20\times 20$ & $30\times 30$ & $40\times 40$ & $50\times 50$ \\
\midrule
$\Omega_t$-grid & error & error & error & error & error & error \\
\midrule
$5\times 5$     & 2.2e--1 & 1.02e--1 & 7.94e--2 & 7.59e--2 & 7.48e--2 & 7.43e--2 \\
$10\times 10$   & 1.93e--1 & 5.36e--2 & 2.34e--2 & 1.88e--2 & 1.73e--2 & 1.67e--2 \\
$20\times 20$   & 1.88e--1 & 4.52e--2 & 1.27e--2 & 7.26e--3 & 5.49e--3 & 4.74e--3 \\
$30\times 30$   & 1.88e--1 & 4.39e--2 & 1.12e--2 & 5.52e--3 & 3.64e--3 & 2.81e--3 \\
$100\times 100$ & 1.87e--1 & 4.31e--2 & 1.02e--2 & 4.46e--3 & 2.51e--3 & 1.63e--3 \\
\bottomrule
\end{tabular}
\end{table}

\paragraph{High-Dimensional Setting.}
To test the ELM surrogate used for higher-dimensional parameter spaces $\Omega_t$, we fixed $\Nt=10$ and prescribed the exact solution 
\begin{align}
u(x,t)=\sin\big((t\cdot v^{(1)})\pi x_1\big)\cos\big((t\cdot v^{(2)})\pi x_2\big)+\frac{10x_1^2x_2}{1+6\|t\|}+\|t\|\exp{\bigg(\small{-\frac{4(x_1^2+x_2^2)}{3}}\bigg)}
\end{align}
where $v^{(1)}, v^{(2)}\in\IR^{10}$ are two unit vectors (initialized randomly). We specified the potentials $\mu_i(x)$ as
\begin{align}
\begin{cases}
\mu_0(x)=0.1\\
\mu_i(x)=500\exp(-50\|x-p_i\|^2), \quad i=1,2,\dots,10
\end{cases}
\end{align}
where $p_i$ are random points in $\Omega_x$. For each choice of mesh size $h_x$ and number of ReLU units $M$, we generated 50 independent ELMs by randomly assigning their inner weights and biases, while fitting the outer parameters to a training dataset by solving \eqref{eq:constrained-problem}. The training dataset for each ELM was generated by first sampling $J$ interpolation points in $\Omega_t$ using scrambled Sobol sequences. We imposed the fixed relationship $M=4J$ between the number of interpolation points and ReLU units. Equivalently, for each run we take $J=M/4$ interpolation points. We then computed the corresponding finite element solutions for each parameter.

To evaluate the accuracy of the ELMs, we computed the $L^2$-error $\|u-\uh\|_{\mathcal{O}}$. The integral over the 10-dimensional parameter space was approximated by Monte Carlo integration with 10 000 randomly chosen samples of $t$, while the spatial integral was approximated by Gaussian quadrature on the FEM mesh. For each ELM, the Monte Carlo procedure was repeated 10 times and the resulting errors were averaged to obtain a stable estimate of the error for each ELM. Finally, we report the mean of these errors across the 50 independently constructed ELMs, to reduce the effect of random ReLU features and to provide a representative measure of typical model performance. The errors, averaged over the 50 ELMs for various choices of $(h_x,M)$ are shown in Figure~\ref{fig:elm-convergence_N=10}. 

Here, we also see that the error levels out when refining $h_x$ while keeping $M$ fixed. For the finest spatial discretization $h_x$, the observed convergence rate for the ELM surrogate (with respect to $M$) is approximately $-0.47$, which is close to the theoretical rate $-0.5$ in Theorem~\ref{thm:error-estimates-elm-high-dim}. 

In the convergence plots, we also include results obtained with fully trained neural networks having the same architecture as the ELMs, but with adaptive inner weights and biases. For the finest $h_x$ we varied $M$ and for each $M$ we trained 50 independent neural networks, with randomly initialized parameters, using the Adam optimizer (a variant of stochastic gradient descent). For each network we performed 100 000 iterations in the optimization, a batch size of 100 and a learning rate of 5e--4 which was reduced by a factor of $0.25$ after half of the iterations. A similar experiment was conducted varying $h_x$ instead while keeping $M$ fixed at its largest value. 

We can see that the error for the fully trained networks is smaller than that of the corresponding ELMs, demonstrating the benefits of allowing the inner weights and biases to be adaptive. On the other hand, training a neural network is computationally more expensive and challenging than training an ELM, as the optimization problem is non-convex and the number of trainable parameters is larger.

\begin{figure}
\centering
\begin{subfigure}[ht]{0.33\textwidth}
\centering
\includegraphics[width=\textwidth]{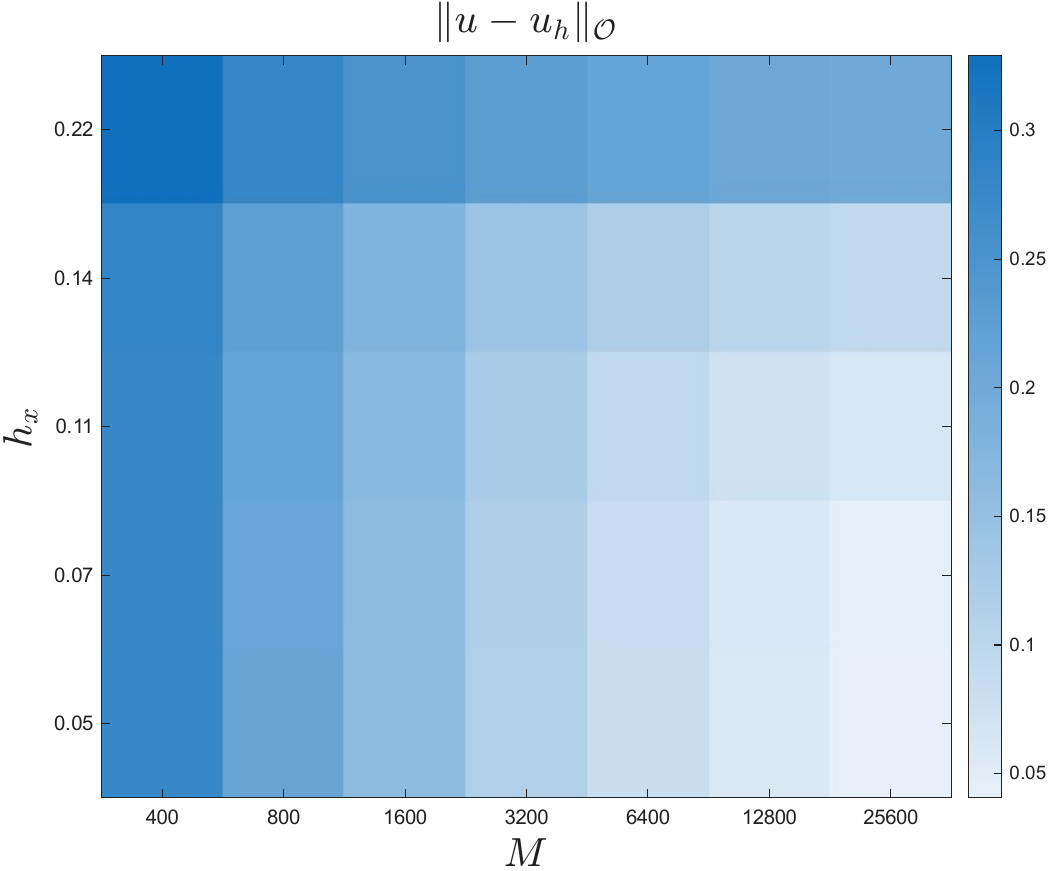}
\caption{}
\end{subfigure}
\begin{subfigure}[ht]{0.32\textwidth}
\centering
\includegraphics[width=\textwidth]{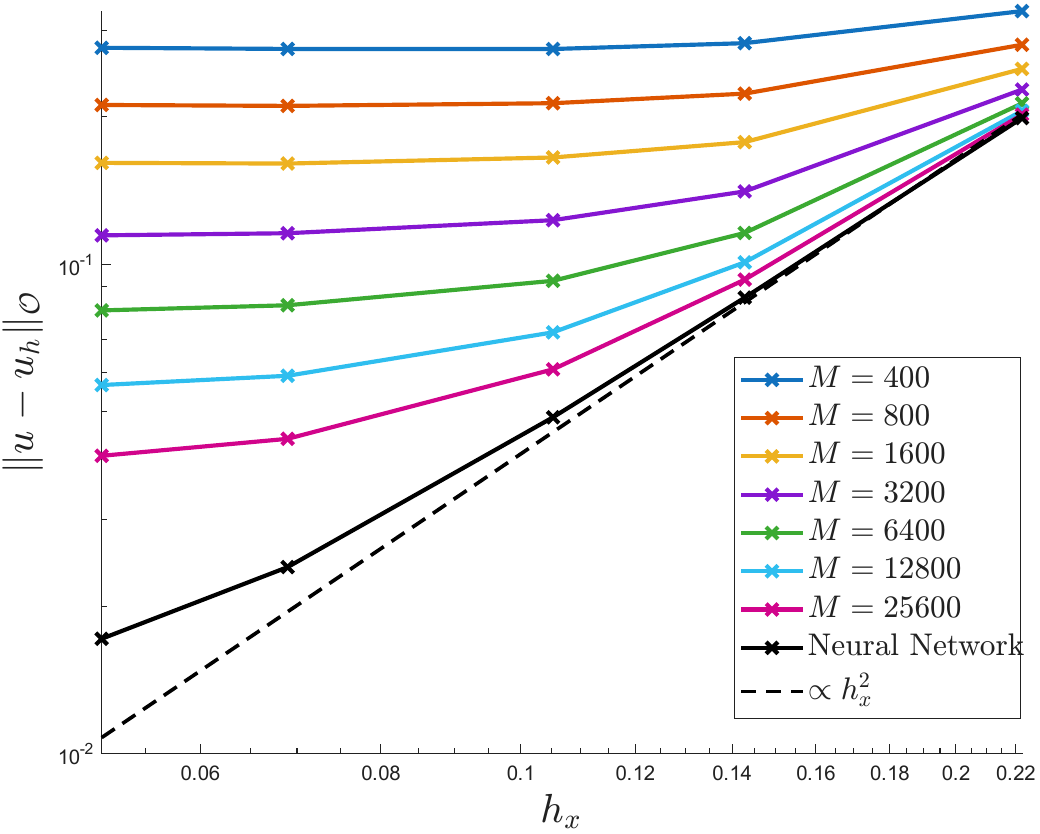}
\caption{ }
\end{subfigure}
\begin{subfigure}[ht]{0.32\textwidth}
\centering
\includegraphics[width=\textwidth]{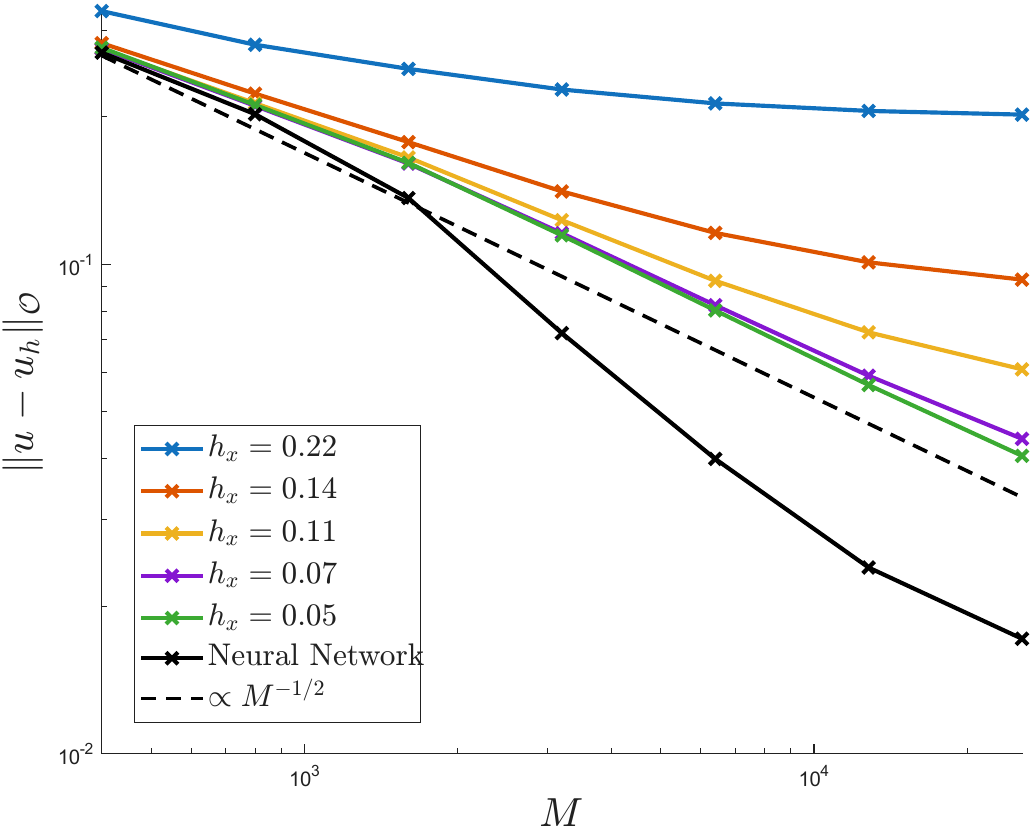}
\caption{ }
\end{subfigure}
\caption{\emph{Convergence in high-dimensional setting.} \textbf{(a)} The heatmap shows approximation error for the ELMs for different values of the pair $(h_x,M)$. \textbf{(b)} The figure shows the approximation error as a function of $h_x$ for different choices of $M$. A reference line is included to indicate the theoretical convergence rate of $2.0$ with respect to $h_x$. The errors for the fully trained networks with $M=25600$ ReLU units are also included. \textbf{(c)} The figure shows the approximation error as a function of $M$ for different choices of $h_x$. A reference line is included to indicate the theoretical convergence rate of $-0.5$ with respect to $M$. The errors for the fully trained networks with $h_x=0.05$ are also included.}
\label{fig:elm-convergence_N=10}
\end{figure}

In Table \ref{tab:errors_N=10}, we summarize the averaged relative $L^2$-errors for the different runs. The reported errors and standard deviations are the averages over the 50 independent ELMs for each pair $(h_x,M)$.

These numerical results are also consistent with the assumptions underlying the theoretical analysis in the high-dimensional setting. In particular, the observed $M^{-1/2}$-type convergence for the ELM surrogate is in line with Assumption~\ref{ass:elm-barron-stable}, and the comparable behavior of the fully trained two-layer networks supports the connection to the Barron-space approximation framework discussed in Remark~\ref{rem:elm-barron-stable-refs}.

\begin{table}
\centering
\caption{\emph{Error in high-dimensional setting.} The average of the relative $L^2$-error and the estimated standard deviation for the ELMs for different $\Omega_x$-grids and numbers of ReLU units.}
\label{tab:errors_N=10}
\small
\begin{tabular}{lcccccccccc}
\toprule
$\Omega_x$-grid & \multicolumn{2}{c}{$10\times 10$} & \multicolumn{2}{c}{$15\times 15$} & \multicolumn{2}{c}{$20\times 20$} & \multicolumn{2}{c}{$30\times 30$} & \multicolumn{2}{c}{$40\times 40$} \\
\cmidrule(lr){2-3}\cmidrule(lr){4-5}\cmidrule(lr){6-7}\cmidrule(lr){8-9}\cmidrule(lr){10-11}
$M$ & error  & std & error  & std & error & std & error & std & error & std \\
\midrule
400  & 0.148 & 6.2e--3 & 0.127 & 7.1e--3 & 0.124 & 8.4e--3 & 0.124 & 6.7e--3 & 0.124 & 7.7e--3 \\
800  & 0.126 & 3.3e--3 & 0.100 & 4.0e--3 & 0.096 & 4.1e--3 & 0.095 & 4.4e--3 & 0.095 & 4.4e--3 \\
1600 & 0.113 & 2.0e--3 & 0.080 & 2.8e--3 & 0.074 & 3.3e--3 & 0.072 & 3.7e--3 & 0.072 & 3.9e--3 \\
3200 & 0.102 & 1.2e--3 & 0.063 & 2.0e--3 & 0.055 & 2.2e--3 & 0.052 & 2.3e--3 & 0.051 & 2.5e--3 \\
6400 & 0.096 & 5.6e--4 & 0.052 & 1.0e--3 & 0.042 & 1.4e--3 & 0.037 & 1.6e--3 & 0.036 & 1.6e--3 \\
12800 & 0.092& 2.3e--4 & 0.045 & 4.2e--4 & 0.033 & 6.0e--4 & 0.027 & 8.1e--4 & 0.025 & 8.1e--4 \\
25600 & 0.091& 9.9e--5 & 0.042 & 2.0e--4 & 0.027 & 2.7e--4 & 0.020 & 4.0e--4 & 0.018 & 4.6e--4 \\
\bottomrule
\end{tabular}
\end{table}

\subsection{Parameter Reconstruction}
\label{subsec:numerical parameter-reconstruction}
The methodology developed in this work applies to general finite-rank observation operators $Q:L^2(\Omega_x)\to \mathcal{M}_m$, where $\dim(\mathcal{M}_m)=N_m$. In the numerical experiments, however, we focus on a specific choice based on pixel averaging. The computational domain $\Omega_x=[-1,1]^2$ is partitioned into a set of $N_m$ axis-aligned square pixels, enumerated as $P_1,P_2,\hdots,P_{N_m}$ and with equal area $|P|$. We then set
\begin{align}
\mathcal{M}_m=\text{span}\{\mathbbm{1}_{P_i}: i=1,2,\dots,N_m\}
\end{align}
Thus, $\mathcal{M}_m$ is defined to be the span of the indicator functions associated to the pixels. Given the absorbed energy $\mu u$, the function $Q[\mu u]$ is a piecewise-constant function on $\Omega_x$ subordinate to the pixel partition. Specifically, we define $Q$ to be the $L^2$-projection onto $\mathcal{M}_m$ so the measurement on a pixel $P$ is the mean of the absorbed energy over $P$
\begin{align}
Q[\mu u]|_P=\frac{1}{|P|}\int_{P}\mu u \, dx
\end{align}
The action of $Q$ is illustrated in Figure~\ref{fig:Q-operator}.

\begin{figure}
\pdfpxdimen=\dimexpr 1in/100\relax
\centering
\begin{subfigure}[h]{0.45\textwidth}
\centering
\includegraphics[trim={159px 100px 116px 200px},clip,width=\textwidth]{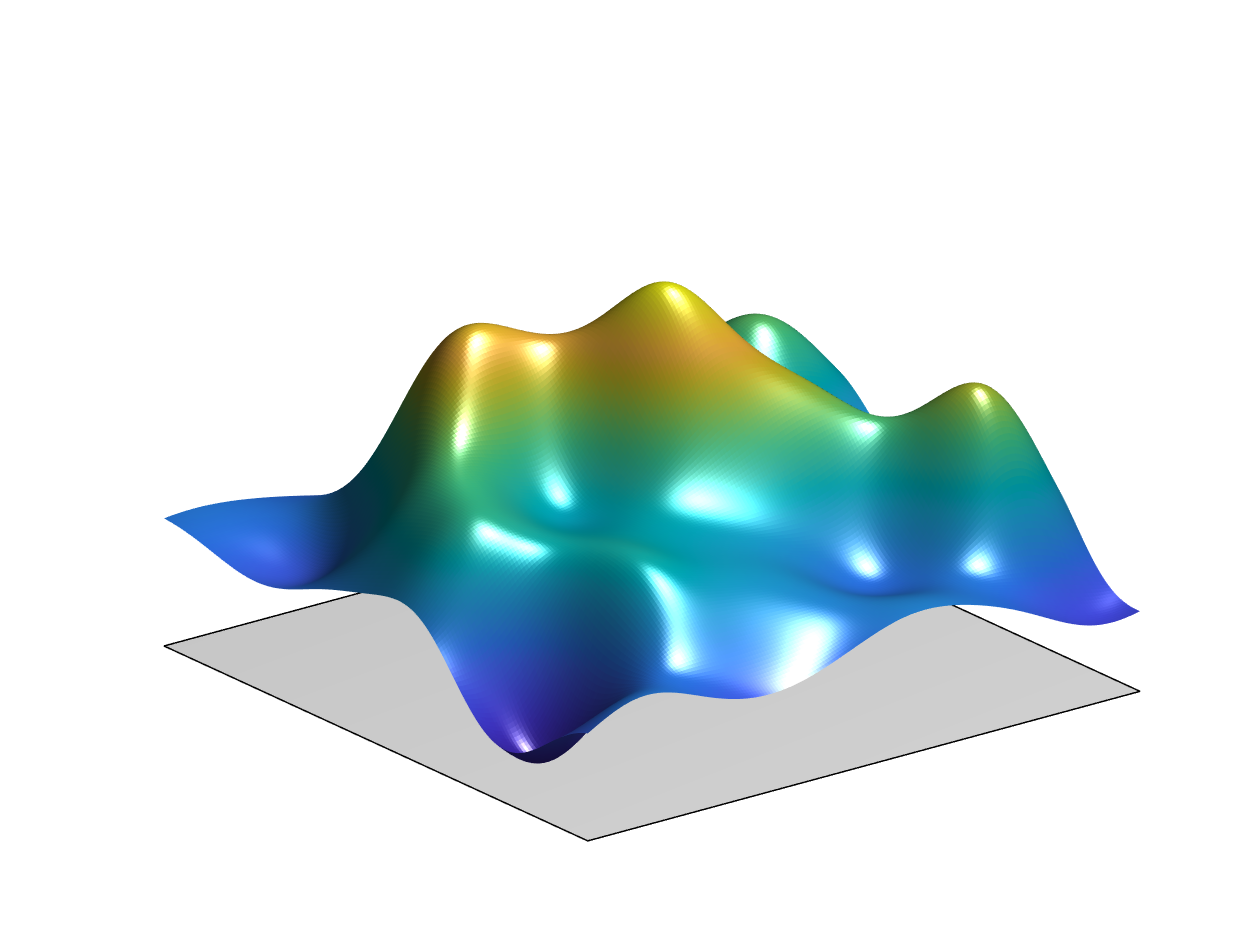}
\caption{$\mu(x,t) u(x,t)$}
\end{subfigure}
\qquad
\begin{subfigure}[h]{0.45\textwidth}
\centering
\includegraphics[trim={159px 100px 116px 200px},clip,width=\textwidth]{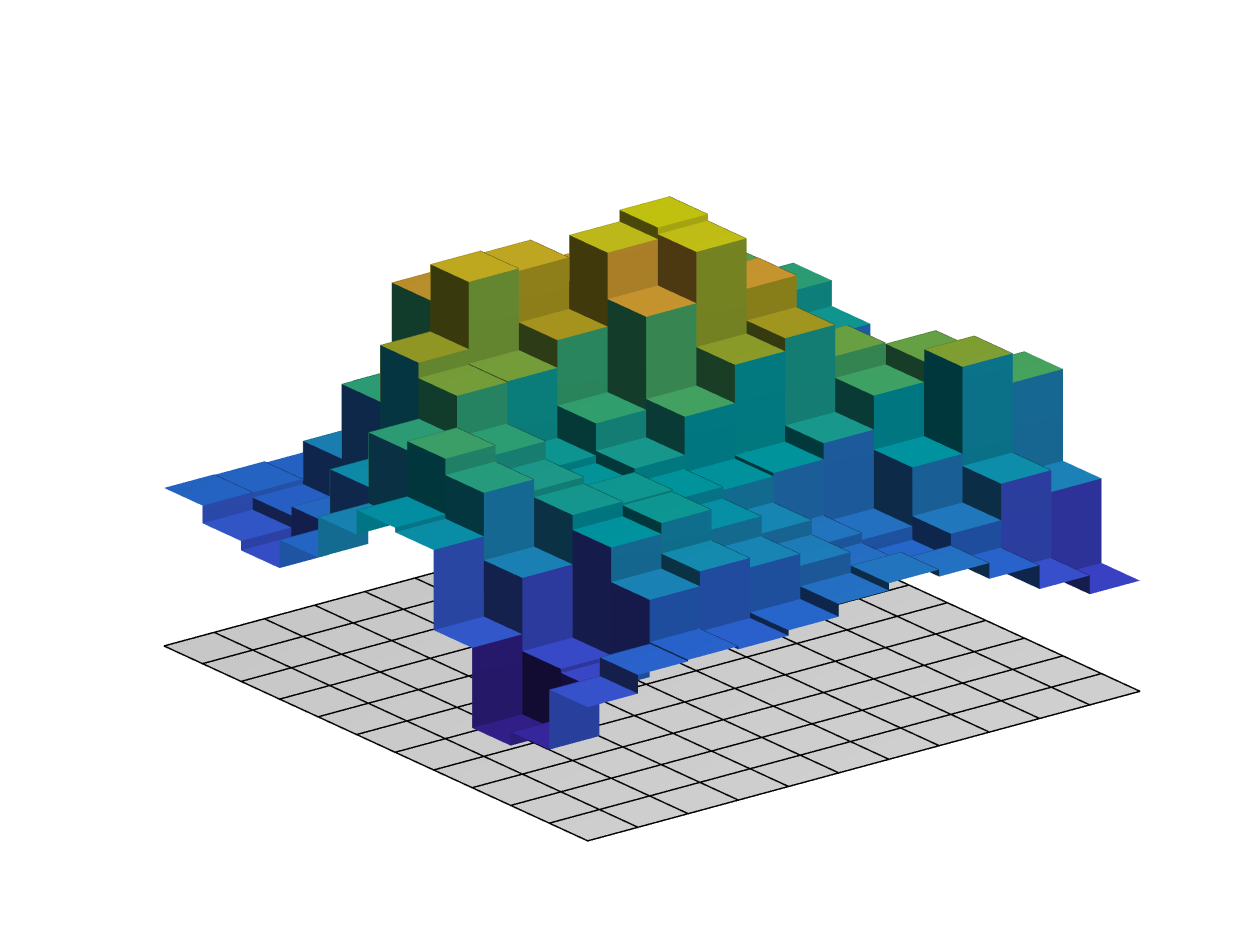}
\caption{$Q[\mu(x,t) u(x,t)]$}
\end{subfigure}
\caption{\textit{Measurement operator.} \textbf{(a)} The absorbed energy is given by $\mu u$ depicted in the figure. \textbf{(b)} By averaging over pixels we get the measurement $Q[\mu u]$ which is a piecewise-constant function on $\Omega_x$ subordinate to the pixel partition.}
\label{fig:Q-operator}
\end{figure}
In the following experiments we focused on the high-dimensional setting where ELMs are used. Given some unknown parameter $t^\dagger\in \Omega_t$, we wish to recover the potential $\mu(x,t^\dagger)$ from the measurable quantity $Q_{\text{obs}}=Q[\mu(x,t^\dagger)u(x,t^\dagger)]$.

We trained the ELM on problem \eqref{eq:PDE-QPAT} with the data
\begin{align}
f(x)=0,\quad g(x)=\frac{x_1^2x_2^2}{16}+1
\label{eq:pde-data}
\end{align}
The base potential was set to $\mu_0(x)=1$ and the other potentials in \eqref{eq:potential} were defined to be Gaussians 
\begin{align}
\mu_i(x)=r_i\exp\!\left(-\frac{\|x-p_i\|^2}{\sigma_i^2}\right),\quad i=1,2,\dots,\Nt
\label{eq:gauss-potentials}
\end{align}
of varying widths $\sigma_i$, amplitudes $r_i$ and centers $p_i$. The parameters for each Gaussian were uniformly sampled from the following ranges: $r_i\in [35,55]$, $\sigma_i\in [0.05,0.15]$ and $p_i\in \Omega_x$. 

In the first experiment we considered the case where the parameter space $\Omega_t$ was 50 dimensional, i.e., $\Nt=50$. To reconstruct the potential $\mu(x,t^\dagger)$, for a randomly selected target parameter $t^\dagger$, we minimized the loss $L(t)$ in \eqref{eq:loss-def} using the gradient descent scheme \eqref{eq:gd-update}. In each iteration we used a backtracking line-search algorithm to determine the step size $\steplen$. As our initial guess $t^{(0)}$, we sampled a random vector in $\Omega_t$. The details of this numerical experiment are found in Table~\ref{tab:exp1-details}.
\begin{table}
\centering
\caption{\emph{Experimental details.} Parameter choices in the initial potential reconstruction experiment using a measurement operator with a fixed pixel resolution ($N_m=25\times 25$) and full coverage of the domain $\Omega_x$.}
\label{tab:exp1-details}
\small
\begin{tabular}{lccccc}
\toprule
$\Omega_x$-grid      & $h_x$  & $\Nt$& $J$  & $M$   & $N_m$ \\
\midrule
$30\times 30$ & 0.069 & 50 & 10000 & 40000 & $25\times 25$ \\
\bottomrule
\end{tabular}
\end{table}

\begin{rem}[Approximate Reference Solution]
As we don't have the exact solution $u(x,t^{\dagger})$, which is needed to construct $Q_{\text{obs}}$, we approximated it by the FEM solution on a finer $40\times 40$ mesh. This was done throughout all experiments in this subsection.
\end{rem}

We performed 200 gradient descent iterations, so our estimate of $\mu(x,t^\dagger)$ is $\mu(x,\hat{t})$ with $\hat{t}=t^{(200)}$. Figure~\ref{fig:vis_pot} shows the evolution of the predicted potential over the iterations together with the true potential. The loss and the reconstruction errors in the parameter and the potential in each iteration are shown in Figure~\ref{fig:error-vis_pot}. We can see that the main features of the potential is recovered after around 10 iterations, this is reflected in the quick drop in the loss and errors early in the iterations. 
\begin{figure}
    \centering
    \begin{subfigure}[b]{0.3\textwidth}
        \centering
        \includegraphics[width=\textwidth]{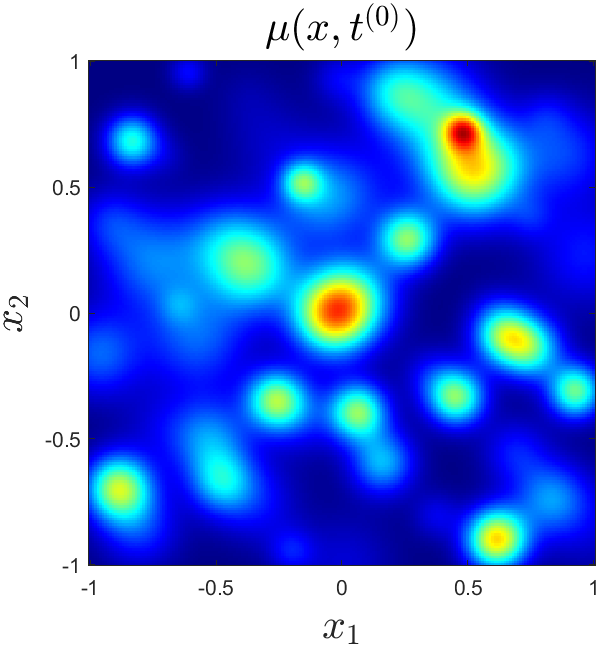}
        \caption{Initial guess}
    \end{subfigure}
    \quad
    \begin{subfigure}[b]{0.3\textwidth}
        \centering
        \includegraphics[width=\textwidth]{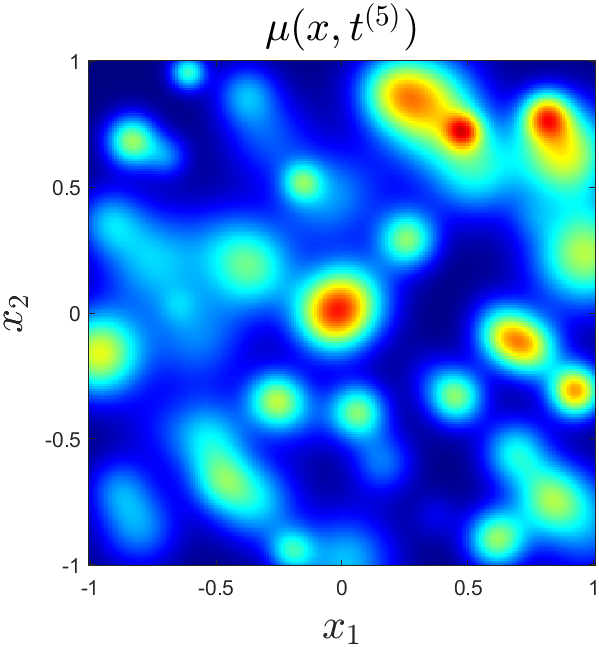}
        \caption{5 iterations}
    \end{subfigure}
    \quad
    \begin{subfigure}[b]{0.3\textwidth}
        \centering
        \includegraphics[width=\textwidth]{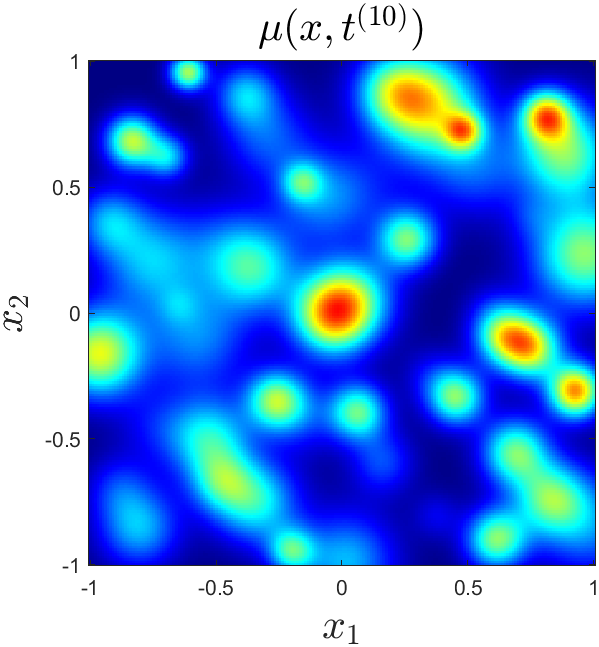}
        \caption{10 iterations}
    \end{subfigure}
    
    \vskip\baselineskip
    \begin{subfigure}[b]{0.3\textwidth}
        \centering
        \includegraphics[width=\textwidth]{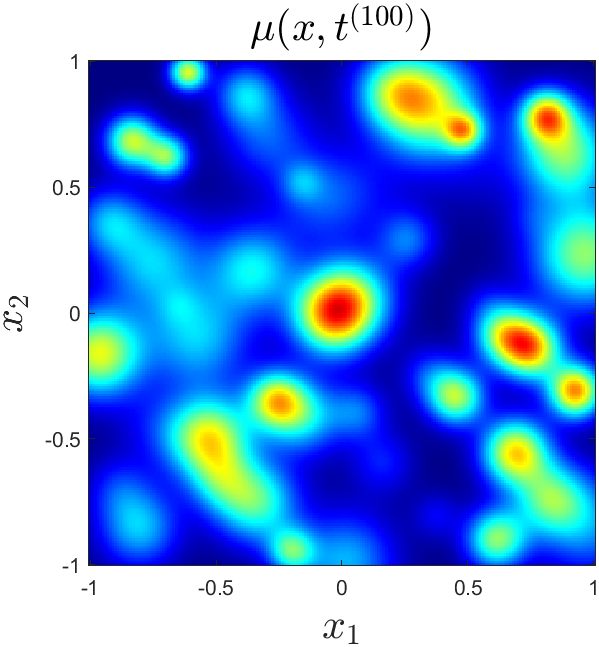}
        \caption{100 iterations}
    \end{subfigure}
    \quad
    \begin{subfigure}[b]{0.3\textwidth}
        \centering
        \includegraphics[width=\textwidth]{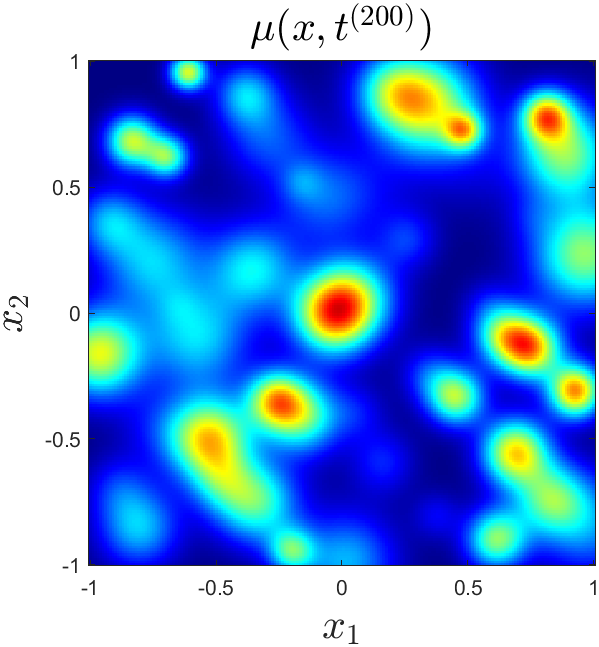}
        \caption{200 iterations}
    \end{subfigure}
    \quad
    \begin{subfigure}[b]{0.3\textwidth}
        \centering
        \includegraphics[width=\textwidth]{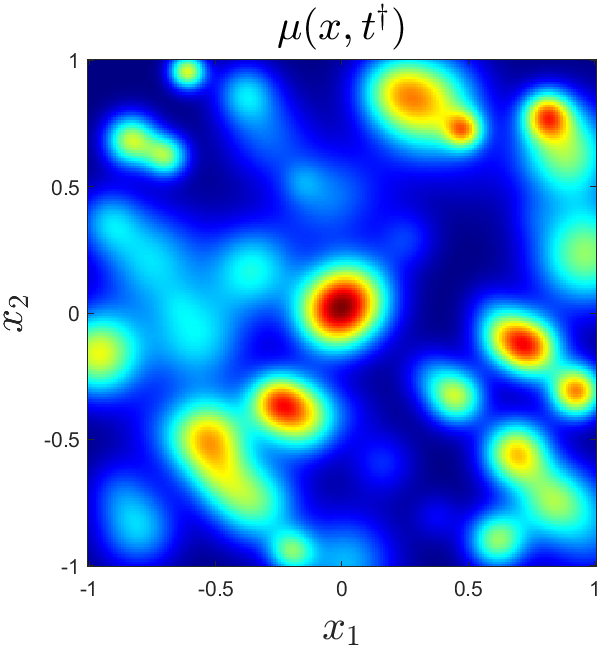}
        \caption{True potential}
    \end{subfigure}
    \caption{\textit{Potential reconstruction}. Illustration of how the potential $\mu$ is successively reconstructed at different stages during the optimization.}
    \label{fig:vis_pot}
\end{figure}

\begin{figure}
\centering
\begin{subfigure}[b]{0.3\textwidth}
\centering
\includegraphics[width=\textwidth]{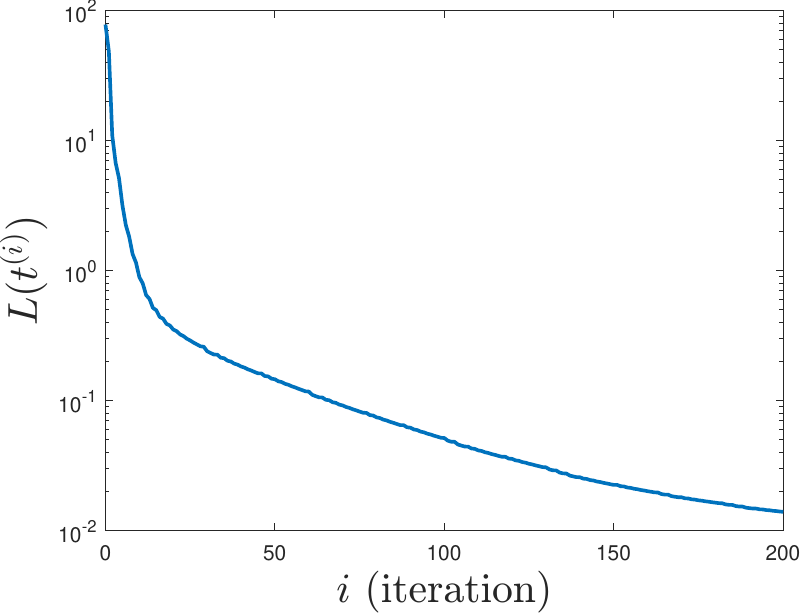}
\caption{}
\end{subfigure}
\hfill
\begin{subfigure}[b]{0.3\textwidth}
\centering
\includegraphics[width=\textwidth]{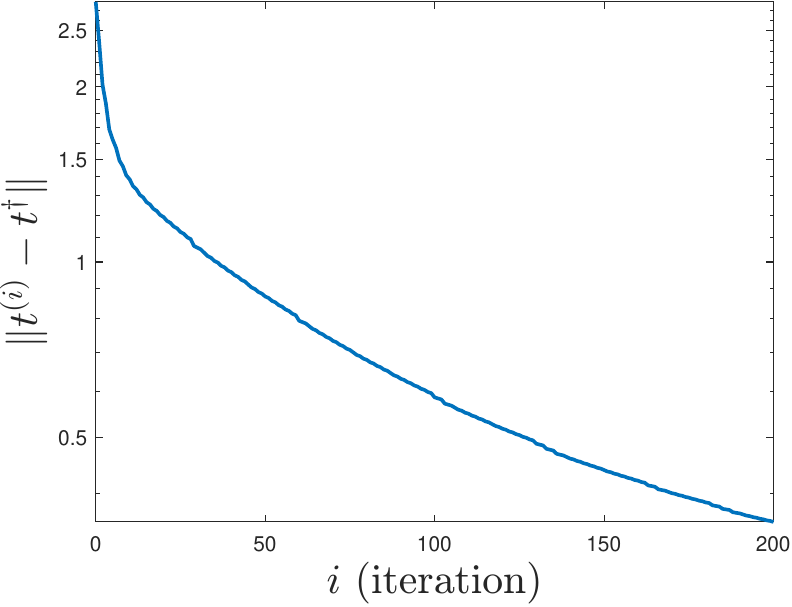}
\caption{ }
\end{subfigure}
\hfill
\begin{subfigure}[b]{0.3\textwidth}
\centering
\includegraphics[width=\textwidth]{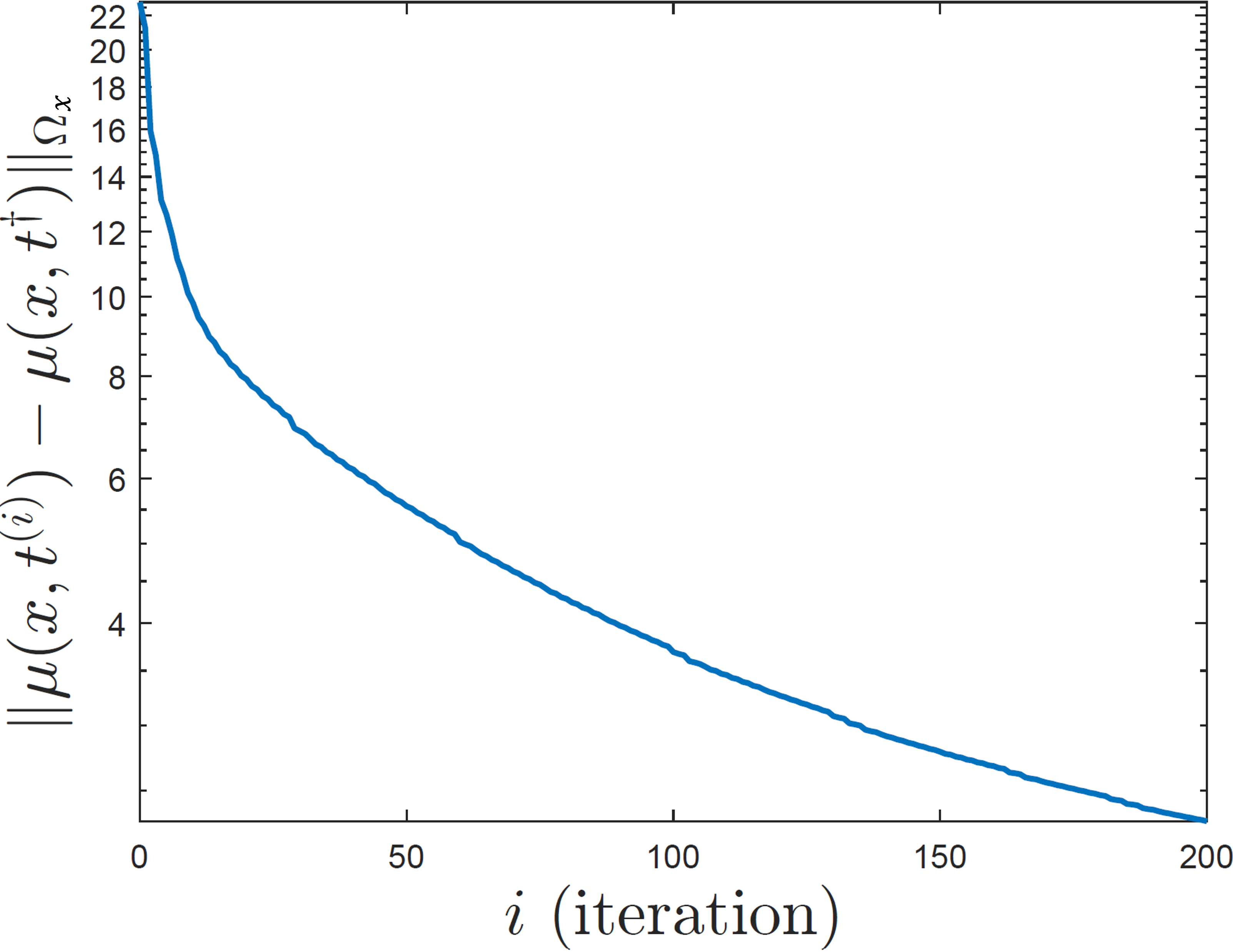}
\caption{ }
\end{subfigure}
\caption{\textit{Loss and error during optimization.} \textbf{(a)} The loss over the iterations, \textbf{(b)} the parameter reconstruction error (measured in the Euclidean norm) over the iterations and \textbf{(c)} the potential reconstruction error over the iterations.}
\label{fig:error-vis_pot}
\end{figure}

To assess the influence of the observation operator on the inverse reconstruction, we conducted a series of tests in which different properties of the measurement operator $Q$ were varied: \emph{resolution}, \emph{domain coverage} and \emph{noise levels}. Through these variations, we systematically explored how the resolution, coverage and noise characteristics of the observation operator affect the accuracy and stability of the recovered parameters.

\paragraph{Resolution.} We varied the number of pixels covering $\Omega_x$, thereby controlling the resolution (pixel size) of the observation operator relative to the underlying finite element mesh. This allowed us to investigate the effect of coarse versus fine observation scales. Figure~\ref{fig:pixel-sizes} shows the effect of the $Q$-operator at different measurement resolutions corresponding to varying the number of pixels $N_m$. We used the PDE-data in \eqref{eq:pde-data} and potentials of the form in \eqref{eq:gauss-potentials} but with amplitudes $r_i\in [5,10]$ and widths $\sigma_i\in [0.2,0.3]$. The details for this experiment are found in Table~\ref{tab:exp2-details}.

\begin{figure}
    \pdfpxdimen=\dimexpr 1in/100\relax
    \centering
    \begin{subfigure}[b]{0.45\textwidth}
        \centering
        \includegraphics[trim={159px 100px 116px 269px},clip,width=\textwidth]{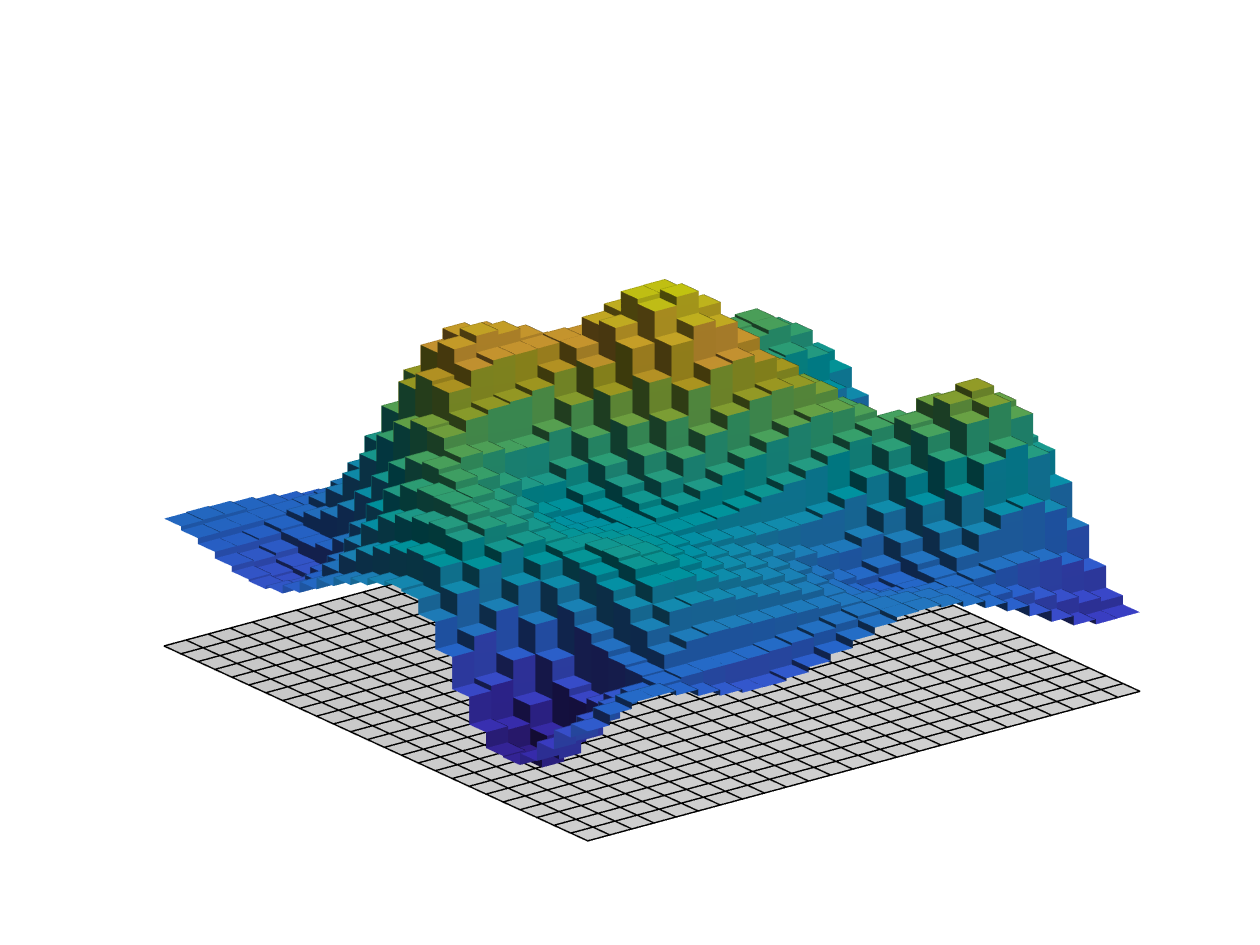}
        \caption{$N_m=25\times 25$}
    \end{subfigure}
    \qquad
    \begin{subfigure}[b]{0.45\textwidth}
        \centering
        \includegraphics[trim={159px 100px 116px 269px},clip,width=\textwidth]{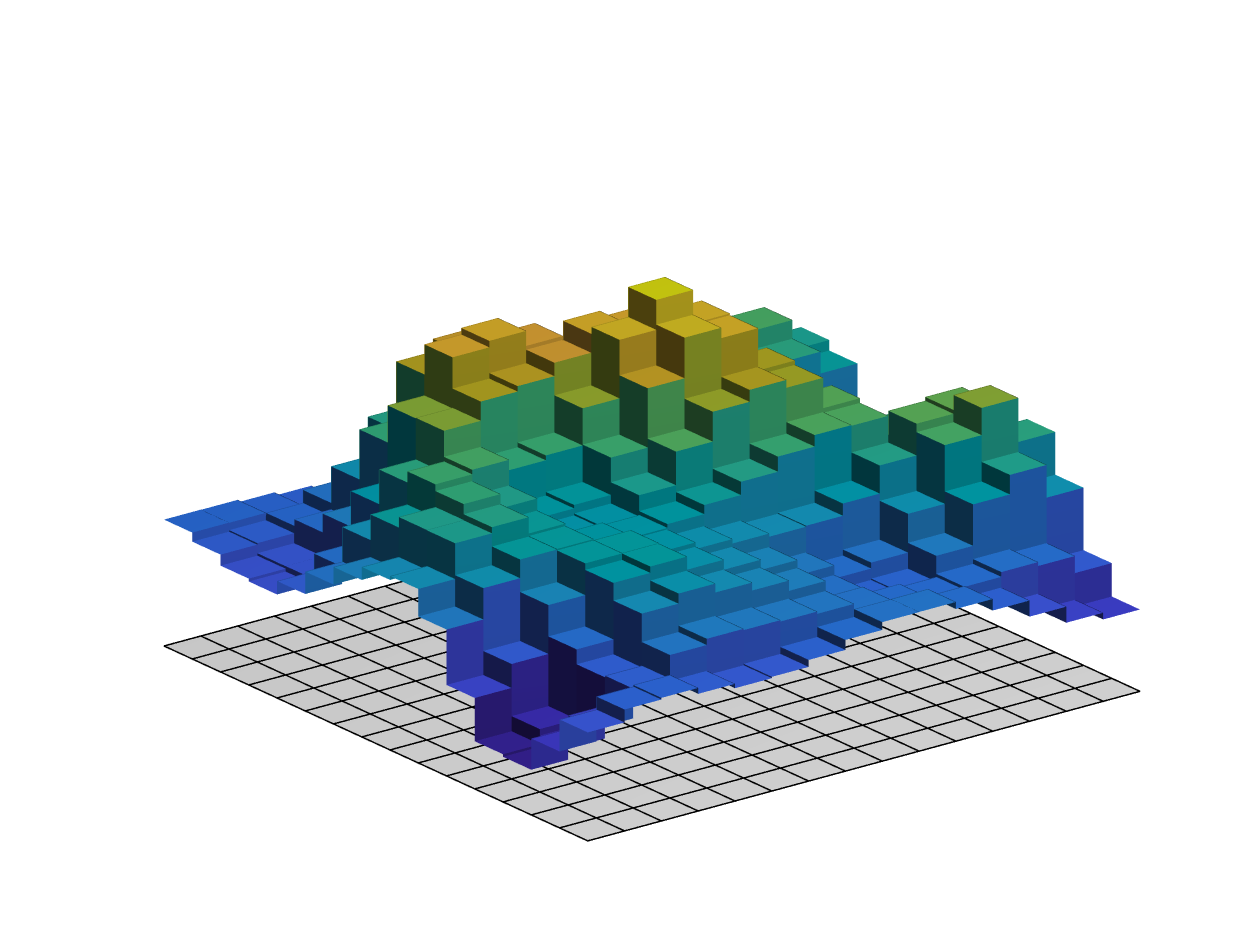}
        \caption{$N_m=15\times 15$}
    \end{subfigure}

    \vspace{0.75em}
    \begin{subfigure}[b]{0.45\textwidth}
        \centering
        \includegraphics[trim={159px 100px 116px 185px},clip,width=\textwidth]{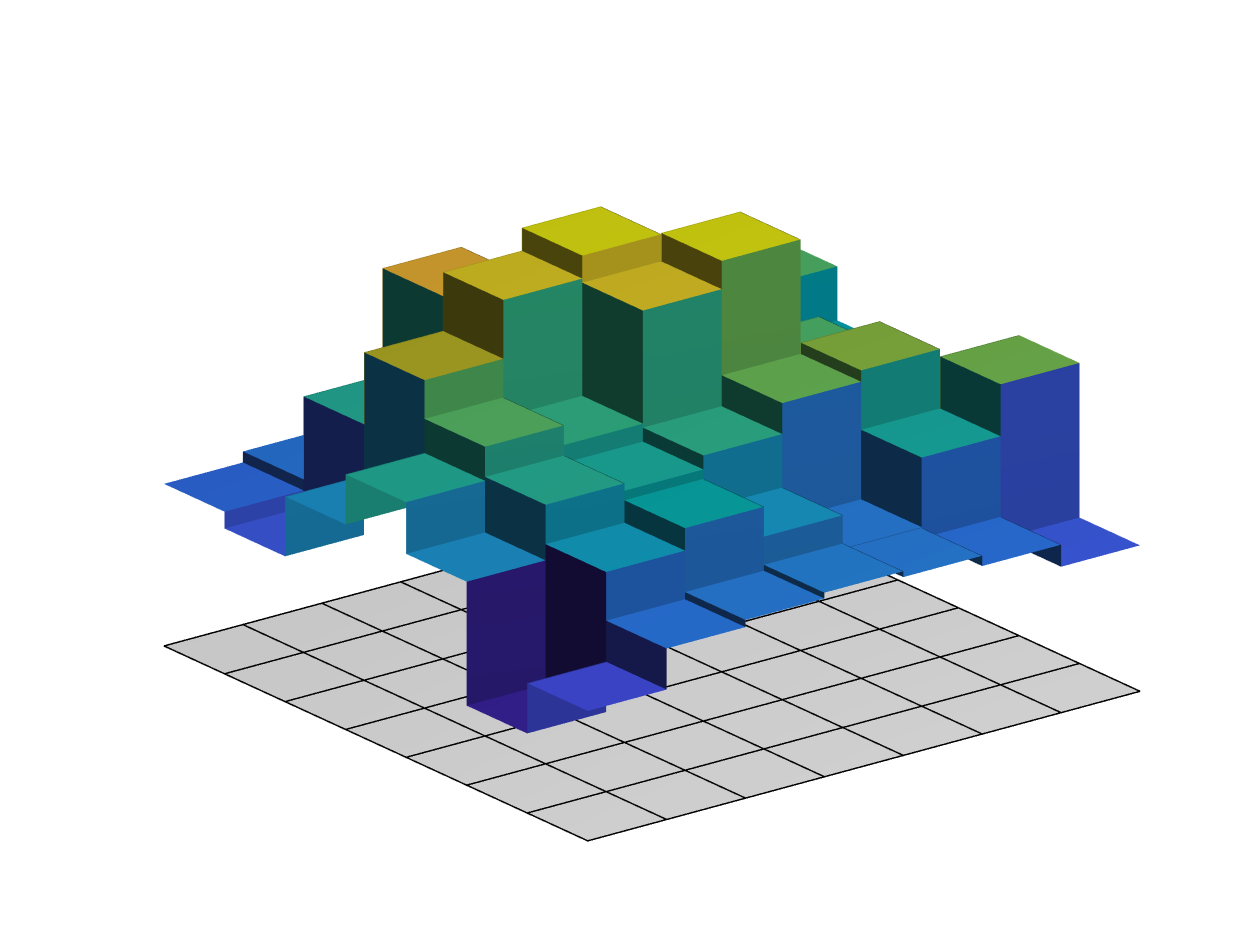}
        \caption{$N_m=7\times 7$}
    \end{subfigure}
    \qquad
    \begin{subfigure}[b]{0.45\textwidth}
        \centering
        \includegraphics[trim={159px 100px 116px 185px},clip,width=\textwidth]{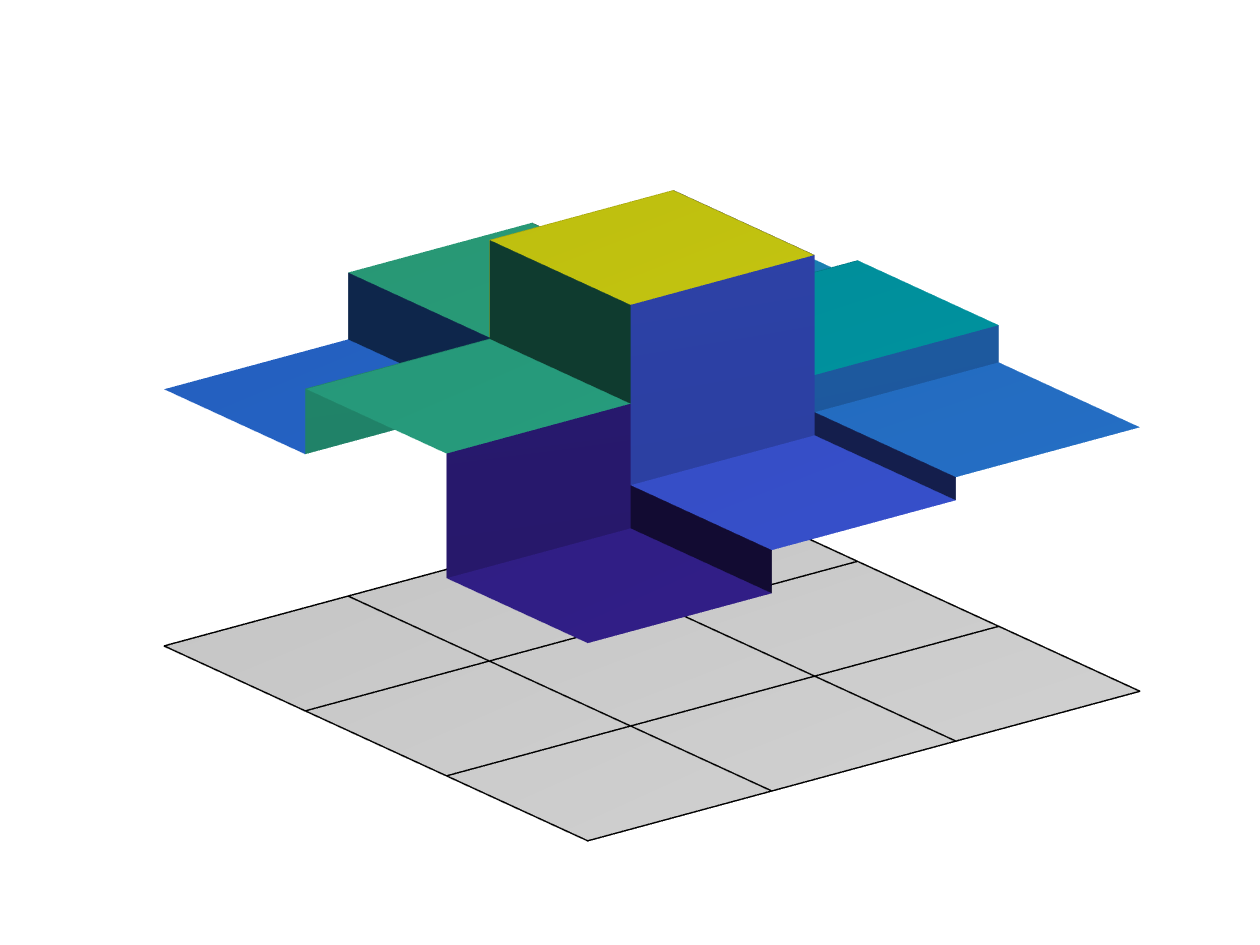}
        \caption{$N_m=3\times 3$}
    \end{subfigure}
    \caption{\emph{Measurement operators with different resolutions.}
    Four examples of the information over $\Omega_x$ seen by the measurement operator $Q$ using different pixel resolutions.}
    \label{fig:pixel-sizes}
\end{figure}

\begin{table}
\centering
\caption{\emph{Experimental details (varied resolution).} Parameter choices for the potential reconstruction experiment when varying the pixel resolution of the measurement operator.}
\label{tab:exp2-details}
\small
\begin{tabular}{lccccc}
\toprule
$\Omega_x$-grid    & $h_x$  & $\Nt$ & $J$   & $M$ & $N_m$ \\
\midrule
$30\times 30$ & 0.069 & 20  & 10000 & 40000  & varied \\
\bottomrule
\end{tabular}
\end{table}

As the number of pixels decreases, information about the absorbed energy $u(x,t)\mu(x,t)$ is progressively lost. Therefore, we expect a reduction in the reconstruction quality when lowering the resolution. This can be seen in Figure~\ref{fig:error-pixelsizes} where the errors for different pixel resolutions are plotted.

\begin{figure}
\centering
\begin{subfigure}[b]{0.45\textwidth}
\centering
\includegraphics[width=\textwidth]{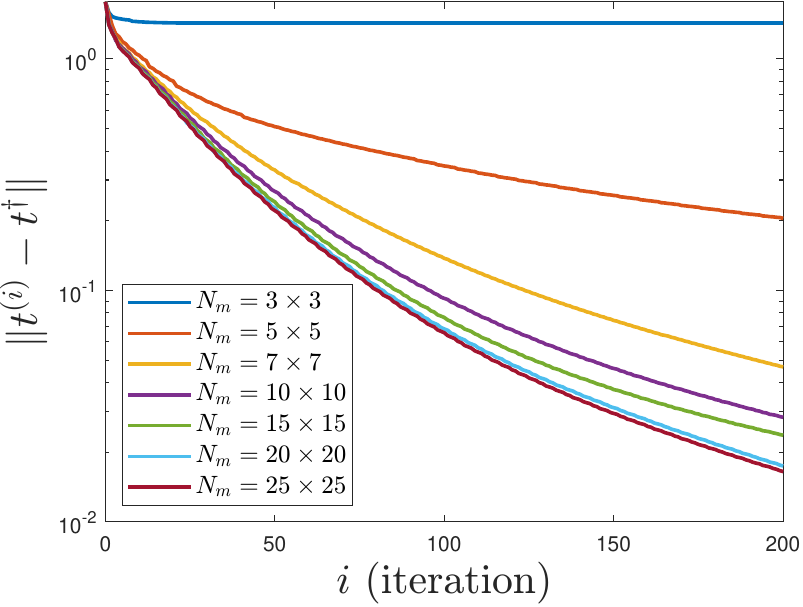}
\caption{}
\end{subfigure}
\qquad
\begin{subfigure}[b]{0.45\textwidth}
\centering
\includegraphics[width=\textwidth]{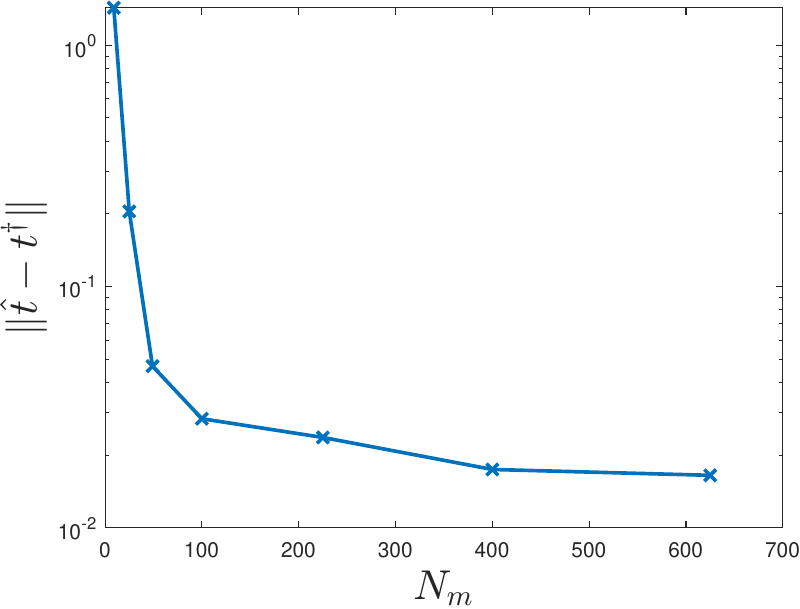}
\caption{ }
\end{subfigure}
\caption{\emph{Effect of pixel resolution.} Errors for the potential reconstruction when varying the pixel resolution of the measurement operator. \textbf{(a)} The reconstruction error over the iterations for different pixel resolutions. \textbf{(b)} The final reconstruction error (measured in the Euclidean norm) as a function of the number of pixels.}
\label{fig:error-pixelsizes}
\end{figure}

The reduction in the error, when increasing the number of pixels, levels out when the pixel resolution approaches the spatial mesh resolution ($30\times 30$). In this regime, each pixel is comparable in size to the triangles in the FEM-mesh. Recall that we replaced the exact solution $u$ in $Q\big[\mu(x,t)u(x,t)\big]$ with the surrogate $Q\big[\mu(x,t)\uh(x,t)\big]$ where $h_x$ is the mesh-size. Consequently, refining the pixel resolution beyond the spatial mesh resolution does not have any significant impact on the reconstruction quality since the additional spatial detail cannot be captured by the underlying finite element discretization.

\paragraph{Domain Coverage.} To model limited sensor coverage or missing data, we restricted $Q$ to act only on a subset of the pixels. In practice, this was implemented by removing the corresponding components from the observations, so that the loss functional $L(t)$ only included contributions from the observed portion of the domain. This is exemplified in Figure~\ref{fig:pixel-fractions}.

\begin{figure}
    \pdfpxdimen=\dimexpr 1in/100\relax
    \centering
    \begin{subfigure}[b]{0.45\textwidth}
        \centering
        \includegraphics[trim={159px 100px 116px 269px},clip,width=\textwidth]{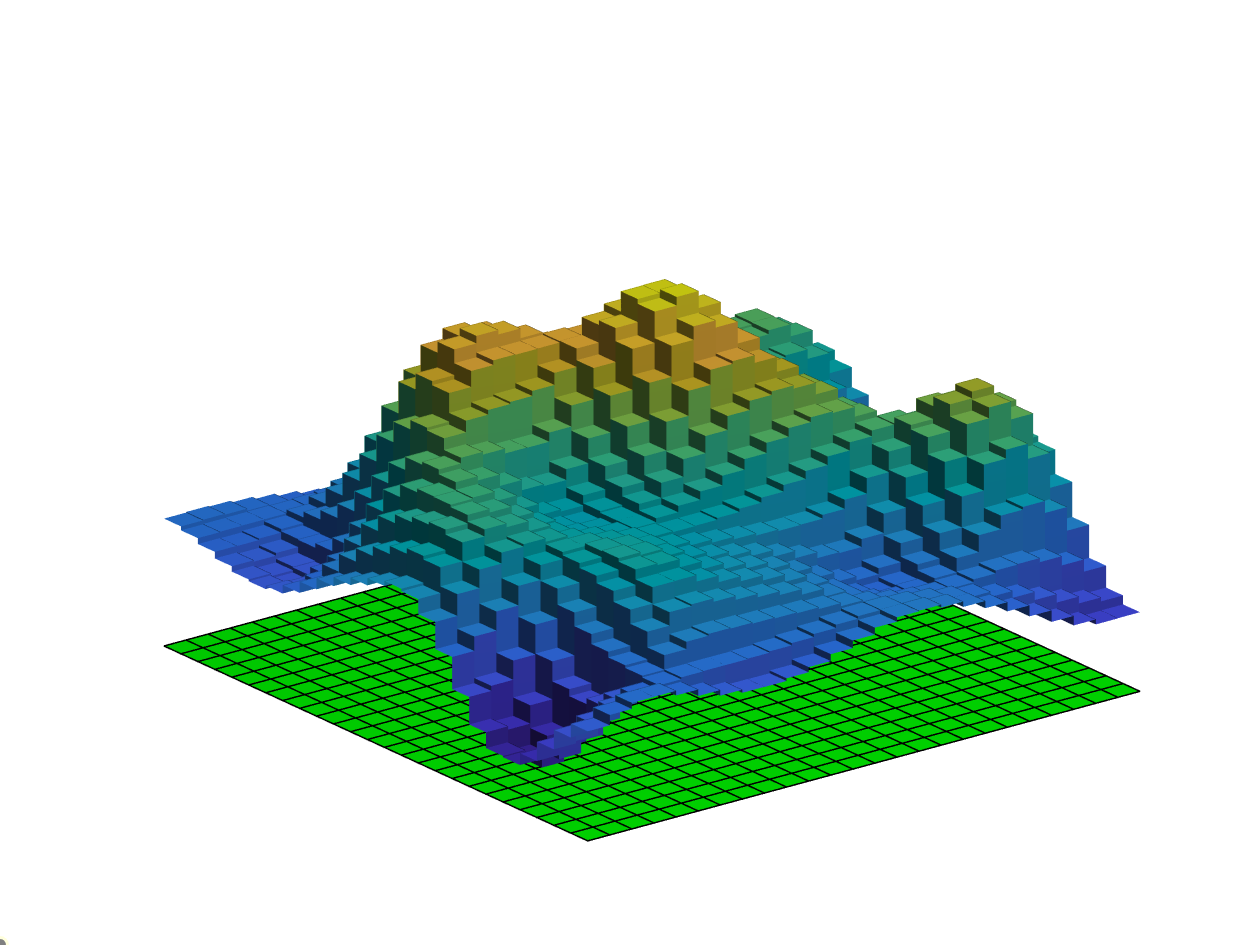}
        \caption{100\%}
    \end{subfigure}
    \qquad
    \begin{subfigure}[b]{0.45\textwidth}
        \centering
        \includegraphics[trim={159px 100px 116px 269px},clip,width=\textwidth]{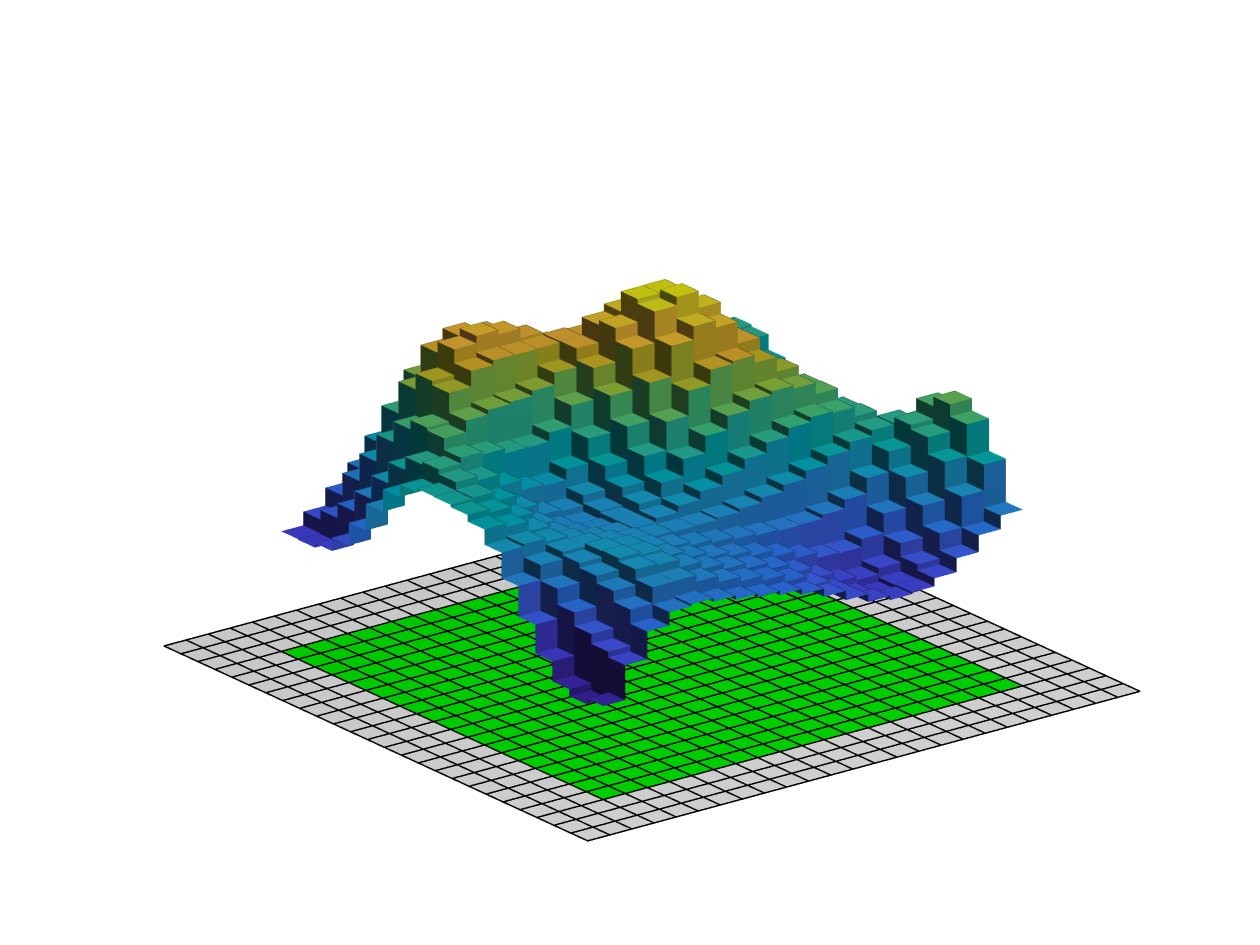}
        \caption{57.76\%}
    \end{subfigure}

    \vspace{0.75em}
    \begin{subfigure}[b]{0.45\textwidth}
        \centering
        \includegraphics[trim={159px 100px 116px 269px},clip,width=\textwidth]{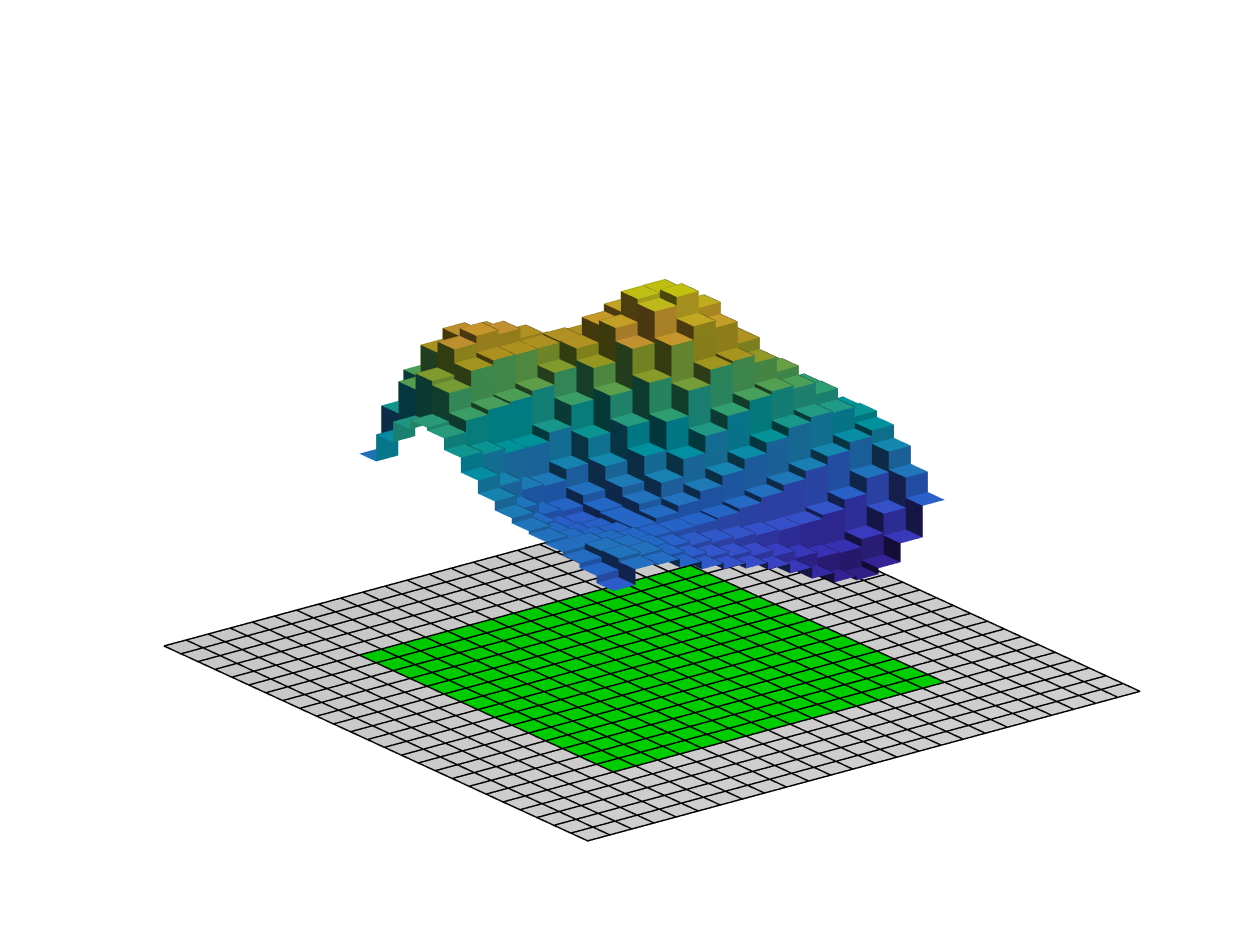}
        \caption{36\%}
    \end{subfigure}
    \qquad
    \begin{subfigure}[b]{0.45\textwidth}
        \centering
        \includegraphics[trim={159px 100px 116px 269px},clip,width=\textwidth]{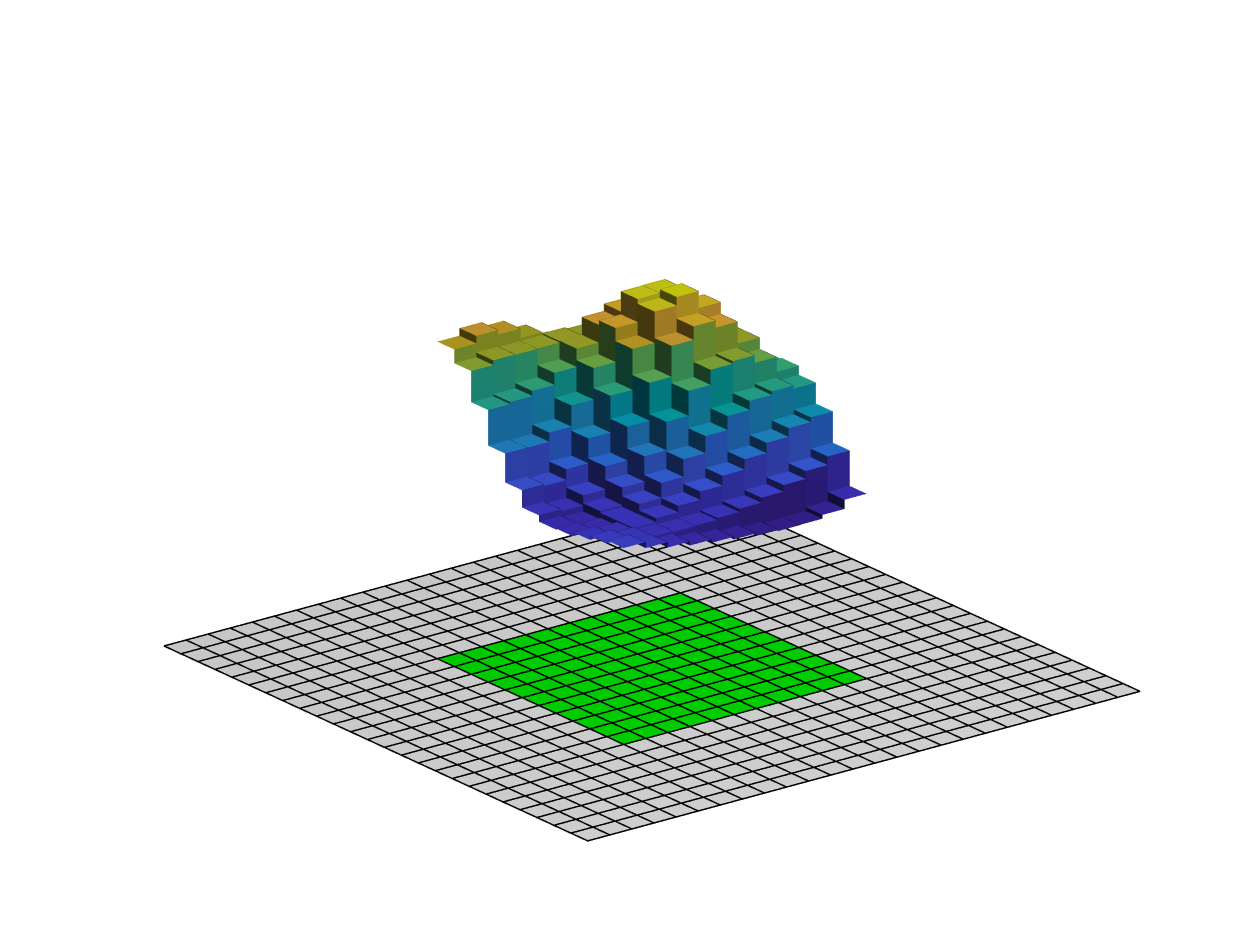}
        \caption{19.36\%}
    \end{subfigure}
    \caption{\emph{Measurement operators with different domain coverage.} Four examples of partial pixel coverage for the measurement operator $Q$, where only the given percentage of the pixels in $\Omega_x$ is seen by the operator.}
    \label{fig:pixel-fractions}
\end{figure}

In this experiment we fixed $\Nt=20$ and reused the PDE-data in \eqref{eq:pde-data} together with the potentials in \eqref{eq:gauss-potentials} with amplitudes $r_i\in [5,10]$ and widths $\sigma_i\in [0.2,0.3]$. Other details of the experimental setup are found in Table~\ref{tab:exp3-details}. 

\begin{table}
\centering
\caption{\emph{Experimental details (varied domain coverage).} Parameter choices for the potential reconstruction experiment when varying the domain coverage of the measurement operator $Q$.}
\label{tab:exp3-details}
\small
\begin{tabular}{lccccc}
\toprule
$\Omega_x$-grid   & $h_x$  & $\Nt$ & $J$   & $M$ & $N_m$ \\
\midrule
$30\times 30$ & 0.069 & 20  & 10000 & 40000  & $25\times 25$ \\
\bottomrule
\end{tabular}
\end{table}

When the coverage of $\Omega_x$ by $Q$ is limited, we expect poor reconstruction quality as essential information is lost. This can be seen in Figure~\ref{fig:error-pixelfractions} below. As the proportion of the observed part of $\Omega_x$ decreases, the error in the reconstructed parameter increases.

\begin{figure}
\centering
\begin{subfigure}[b]{0.45\textwidth}
\centering
\includegraphics[width=\textwidth]{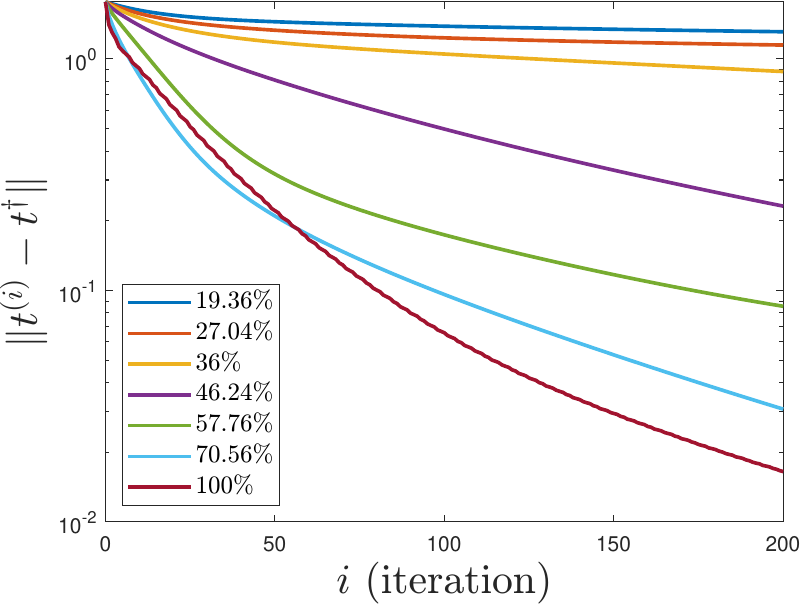}
\caption{}
\end{subfigure}
\qquad
\begin{subfigure}[b]{0.45\textwidth}
\centering
\includegraphics[width=\textwidth]{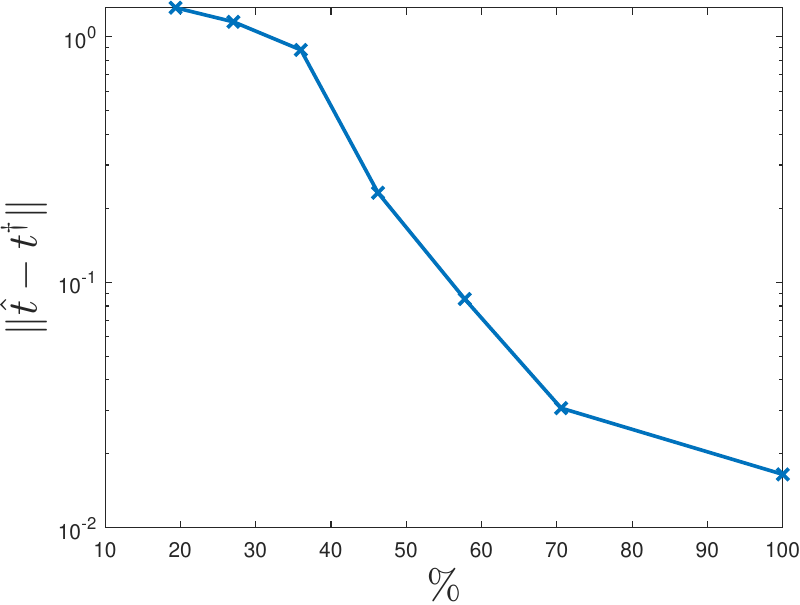}
\caption{ }
\end{subfigure}
\caption{\emph{Effect of domain coverage.} Errors for the potential reconstruction when varying the percentage of the domain seen by the measurement operator $Q$. \textbf{(a)} The reconstruction error (measured in the Euclidean norm) over the iterations for different levels of coverage. \textbf{(b)} The final reconstruction error as a function of the coverage.}
\label{fig:error-pixelfractions}
\end{figure}

\paragraph{Noise Levels.} To investigate the influence of measurement noise on the reconstruction quality, we adopt a multiplicative noise model applied to the pixel observations. Note that we can rewrite the loss in \eqref{eq:loss-def} in vector form as
\begin{align}
L(t)=\frac{|P|}{2}\|q_{\text{obs}}-q(t)\|^2
\end{align}
where $q_{\text{obs}},q(t)\in\IR^{N_m}$ are vectors containing the observed and predicted pixel values from $Q_{\text{obs}}$ and $Q(t)$ respectively. In the noisy setting, the observed vector is replaced by a perturbed version $\tilde{q}_{\text{obs}}$ with elements
\begin{align}
\tilde{q}_{\text{obs},i}=q_{\text{obs},i}\big(1+\rho\epsilon_i \big), \quad \epsilon_i \sim \mathcal{N}(0,1),\quad i=1,2,\hdots, N_m
\end{align}
so that each pixel is corrupted by relative Gaussian noise with amplitude proportional to its magnitude. The proportionality constant $\rho\in(0,1)$ determines the noise level. The corresponding least-square loss becomes

\begin{align}
L(t)=\frac{|P|}{2}\|\tilde{q}_{\text{obs}}-q(t)\|^2
\end{align}

\begin{rem}[Weighted Loss Function]
If information about the measurement variances is available, it can be incorporated directly into the loss functional. The weighted least-squares formulation then reads
\begin{align}
\tilde{L}(t)=\frac{|P|}{2}\big(\tilde{q}_{\text{obs}}-q(t)\big)^{\top}\Sigma^{-1} \big(\tilde{q}_{\text{obs}}-q(t)\big)
\end{align}
where $\Sigma$ denotes the covariance matrix. In the present setting, the noise on the different pixels is assumed to be independent, so that $\Sigma$ is diagonal with entries
\begin{align}
[\Sigma]_{ii}=(\rho\, q_{\text{obs},i})^2
\end{align}
This weighting effectively normalizes the residuals by their expected variance, ensuring pixels with larger values (and consequently larger noise variance) are not disproportionately penalized in the reconstruction. This corresponds to the maximum likelihood estimator under the Gaussian noise assumption with covariance $\Sigma$.
\end{rem}

To assess the robustness of the reconstruction with respect to noise, we fixed $q_{\text{obs}}$ and generated multiple realizations of $\tilde{q}_{\text{obs}}$ for different choices of $\rho$. For each noise level, 20 independent noisy instances were simulated, and for each instance 200 gradient descent iterations were performed. The same PDE-data as in \eqref{eq:pde-data} was used together with potentials of the form \eqref{eq:gauss-potentials} with constant amplitudes $r_i=50$ and widths $\sigma_i=0.1$ but with random centers. Table~\ref{tab:exp4-details} shows the remaining details of the experiment.

\begin{table}
\centering                 
\caption{\emph{Experimental details (varied noise level).} Parameter choices for the potential reconstruction experiment when varying the noise amplitude in the measurements.}
\label{tab:exp4-details}
\small
\begin{tabular}{lccccc}
\toprule
$\Omega_x$-grid  & $h_x$  & $\Nt$ & $J$   & $M$ & $N_m$ \\
\midrule
$20\times 20$ & 0.11 & 20  & 2000 & 8000  & $15\times 15$ \\
\bottomrule
\end{tabular}
\end{table}
For completeness, we report the results for both the losses $L(t)$ and $\tilde{L}(t)$. The mean reconstruction errors as a function of the relative noise amplitude $\rho$ are shown in Figure~\ref{fig:error-noise}.

\begin{figure}
\centering
\includegraphics[width=0.45\linewidth]{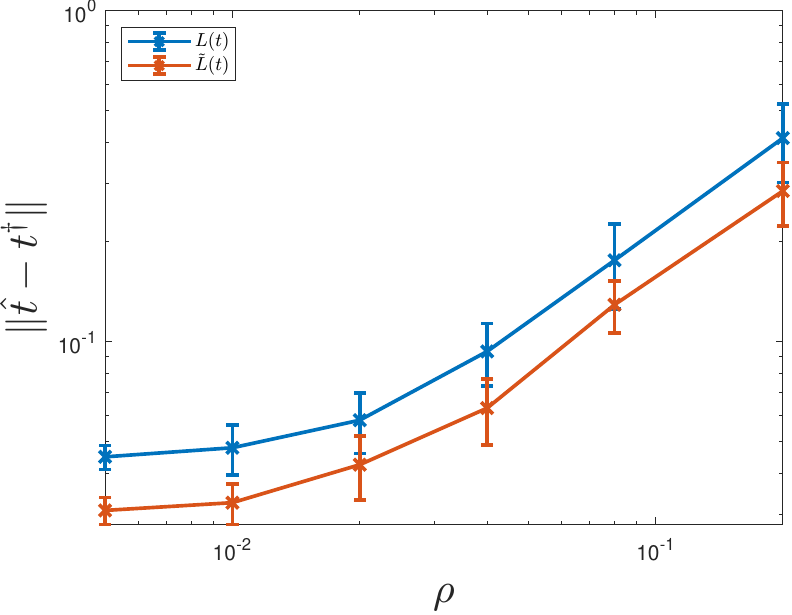}
\caption{\emph{Effect of noise level.} The mean reconstruction error, measured in the Euclidean norm, after 200 iterations as a function of noise amplitude for the original loss functional $L(t)$ and the weighted loss functional $\tilde{L}(t)$.}
\label{fig:error-noise}
\end{figure}
We can see that the performance is better when using the weighted loss. In practice, we may not know the variances as they depend on the pixel values themselves and we only observe noisy measurements of them. However, one could estimate the variances by repeated measurements of the same experiment and then using these estimates in the covariance matrix $\Sigma$.

\section{Conclusions}
\label{sec:conclusions}

We presented a reduced-order modeling framework for parameter-dependent PDEs that
combines finite element discretization in the physical space
$\Omega_x\subset\IR^{\Nx}$ with approximation in the parameter space
$\Omega_t\subset\IR^{\Nt}$. The approach is supported by a detailed analysis of
existence, uniqueness, and regularity of the parametric solution, as well as
joint error estimates that quantify the interaction between spatial and
parametric approximation.

In low-dimensional parameter spaces ($\Nt\le4$), we employ classical interpolation
schemes and establish convergence rates based on Sobolev regularity in the
parameter variable. In higher-dimensional parameter spaces ($\Nt>4$), we replace
classical interpolation by ELM-based surrogates and derive error bounds under
explicit random-feature approximation and stability assumptions. Numerical
experiments indicate that the proposed framework can significantly reduce the
cost of many-query tasks, such as inverse problems, while maintaining accurate
reconstructions.

\paragraph{Main takeaways:}
\begin{itemize}
\item A unified reduced-order workflow accommodates both low- and high-dimensional
parameter spaces by combining a common finite element backbone with either
classical interpolation or randomized ELM surrogates.
\item The analysis provides joint control of spatial and parametric errors, with
rates of order $h_x^s$ in the physical discretization, $h_t^{q+1}$ for classical
parameter interpolation in low dimensions, and $M^{-1/2}$ for ELM-based
approximation in high dimensions under explicit approximation and stability
assumptions.
\item Theoretical reconstruction error bounds for the inverse problem are
supported by numerical experiments, which demonstrate accurate reconstructions
using far fewer full PDE solves than direct approaches, even in the presence of
noise and incomplete data.
\item Once the PDE solves at the interpolation points are available, the ELM
surrogate is inexpensive to evaluate, enabling efficient many-query tasks such
as optimization, uncertainty quantification, and inversion.
\end{itemize}

In future work, we plan to investigate extensions of the framework to other
classes of inverse problems and more general parametric PDE models.

\bigskip
\paragraph{Acknowledgments.}
This research was supported in part by the Wallenberg AI, Au-
tonomous Systems and Software Program (WASP) funded by the Knut and Alice Wallenberg Foundation; the Swedish Research Council Grants Nos.\  2021-04925, 2025-05562; the Swedish Research Programme Essence.
EB acknowledges funding from EPSRC grants EP/T033126/1, EP/V050400/1 and EP/X042650/1.

\bibliographystyle{habbrv}
{
\footnotesize
\bibliography{biblio}
}

\bigskip
\bigskip
{ %
\footnotesize
\noindent
{\bf Authors' addresses:}

\smallskip
\noindent
Erik Burman,  \quad \hfill \addressuclshort\\
{\tt e.burman@ucl.ac.uk}

\smallskip
\noindent
Mats G. Larson,  \quad \hfill \addressumushort\\
{\tt mats.larson@umu.se}

\smallskip
\noindent
Karl Larsson, \quad \hfill \addressumushort\\
{\tt karl.larsson@umu.se}

\smallskip
\noindent
Jonatan Vallin, \quad \hfill \addressumushort\\
{\tt jonatan.vallin@umu.se}

} %

\end{document}